\documentclass[11pt]{amsart}
\usepackage{amsmath} 
\usepackage{amssymb} 
\usepackage{amsthm}
\usepackage{amsfonts}            
\usepackage{mathrsfs}
\usepackage[all]{xy}
\usepackage{mathtools}
\usepackage{mathabx}
\usepackage{colonequals}
\usepackage[colorlinks,
            linkcolor=black,
            anchorcolor=black,
            citecolor=black
            ]{hyperref}

\usepackage{geometry}
\usepackage{setspace}

\title{Arithmetic Siegel-Weil formula on $\mathcal{X}_0(N)$: Singular terms}
\author{Baiqing Zhu}
\date{Jan 2024}
\newcommand{\Q}{\mathbb{Q}}
\newcommand{\Z}{\mathbb{Z}}

\newcommand{\C}{\mathbb{C}}

\newcommand{\F}{\mathbb{F}}

\newcommand{\Xc}{\mathcal{X}_{0}(N)}
\newcommand{\Yc}{\mathcal{Y}_{0}(N)}
\newcommand{\pr}{\prime}
\newcommand{\rar}{\rightarrow}

\newcommand{\CH}{\textup{CH}}
\newcommand{\ACH}{\widehat{\textup{CH}}^{1}(\Xc)}
\newcommand{\ach}{\widehat{\textup{CH}}^{1}(\Xc,\mathcal{S})}
\newcommand{\Ach}{\widehat{\textup{CH}}^{2}(\Xc)}
\newcommand{\hod}{\widehat{\omega}_{N}}

\newcommand{\ad}{\widehat{\mathcal{Z}}}
\newcommand{\LN}{\Gamma_{0}(N)}
\newcommand{\cusp}{\widehat{\mathsf{Cusp}}}
\newcommand{\ch}{\chi_{t}}
\newcommand{\infd}{\widehat{\Delta}_{N}}
\newcommand{\la}{\Delta(N)}
\newcommand{\den}{\textup{Den}}

\newcommand{\Of}{\mathcal{O}_{F}}

\begin{document}
\theoremstyle{definition}        
\newtheorem{definition}{Definition}[subsection] 
\newtheorem{example}[definition]{Example}
\newtheorem{remark}[definition]{Remark}

\theoremstyle{plain}
\newtheorem{theorem}[definition]{Theorem}
\newtheorem{lemma}[definition]{Lemma}         
\newtheorem{proposition}[definition]{Proposition}
\newtheorem{corollary}[definition]{Corollary}

\maketitle
\begin{abstract}
    For arbitrary level $N$, we relate the generating series of codimension 2 special cycles on $\mathcal{X}_{0}(N)$ to the derivatives of a genus 2 Eisenstein series, especially the singular terms of both sides. On the analytic side, we use difference formulas of local densities to relate the singular Fourier coefficients of the genus 2 Eisenstein series to the nonsingular Fourier coefficients of a genus 1 Eisenstein series. On the geometric side, we study the reduction of cusps to compute the divisor class of the Hodge bundle and the heights of special divisors. When $N$ is square-free, this gives a different proof of the main results in the works of Du and Yang \cite{DY19} and Sankaran, Shi, and Yang \cite{SSY22}. 
\end{abstract}
\pagenumbering{roman}
\tableofcontents
\newpage
\pagenumbering{arabic}
\section{Introduction}
\subsection{Background}
The classical Siegel-Weil formula relates certain Siegel Eisenstein series to the arithmetic of quadratic forms, namely it expresses special values of these series as theta functions -- generating series of representation numbers of quadratic forms. Kudla initiated an influential program to establish the arithmetic Siegel-Weil formula relating certain Siegel Eisenstein series to objects in arithmetic geometry. 
\par 
In this article, we study the case of modular curves. Let $N$ be a positive integer, the classical modular curve $\mathcal{Y}_{0}(N)_{\mathbb{C}}$ over $\mathbb{C}$ is defined as the following smooth $1$-dimensional complex curve,
\begin{equation*}
    \mathcal{Y}_{0}(N)_{\mathbb{C}}\coloneqq\textup{GL}_{2}(\mathbb{Q})\backslash \mathbb{H}_{1}^{\pm}\times\textup{GL}_{2}(\mathbb{A}_{f})/\Gamma_{0}(N)(\hat{\mathbb{Z}})\simeq \Gamma_{0}(N)\backslash\mathbb{H}_{1}^{\pm},
\end{equation*}
where $\mathbb{H}_{1}^{\pm}=\mathbb{C}\backslash\mathbb{R}$ and $\mathbb{H}_{1}^{+}=\{z=x+iy\in\mathbb{C}:x\in\mathbb{R}, y\in\mathbb{R}_{>0}\}$ is the upper half plane. The group $\Gamma_{0}(N)(\hat{\mathbb{Z}})$ is the following open compact subgroup of $\textup{GL}_{2}(\mathbb{A}_{f})$,
\begin{equation*}
   \Gamma_{0}(N)(\hat{\mathbb{Z}}) = \left\{x=\begin{pmatrix}
    a & b\\
    Nc & d
    \end{pmatrix}\in \textup{GL}_{2}(\hat{\mathbb{Z}})\,\,:\,\,a,b,c,d\in\hat{\mathbb{Z}}\right\},
\end{equation*}
and $\Gamma_{0}(N)=\Gamma_{0}(N)(\hat{\mathbb{Z}})\bigcap \textup{GL}_{2}(\mathbb{Z})$.
\par
The smooth curve $\mathcal{Y}_{0}(N)_{\mathbb{C}}$ is not proper, its compactification $\mathcal{X}_{0}(N)_{\mathbb{C}} \coloneqq\mathcal{Y}_{0}(N)_{\mathbb{C}}\cup\{\textup{cusps}\}$ is a smooth projective curve over $\mathbb{C}$. Katz, Mazur \cite{KM85} and Cesnavicius \cite{Ces17} constructed an integral model $\Xc$ for the complex curve $\Xc_{\C}$. The model $\Xc$ is a 2-dimensional regular flat Deligne-Mumford stack. For every pair $(T,\mathsf{y})$ where $T$ is a $2\times2$ symmetric matrix with coefficients in $\Q$ and $\mathsf{y}$ is a positive definite $2\times2$ symmetric matrix with coefficients in $\mathbb{R}$, we define the arithmetic special cycle $\widehat{\mathcal{Z}}(T,\mathsf{y})$ on the stack $\mathcal{X}_{0}(N)$ and study their arithmetic degrees. Finally, we prove that these arithmetic degrees are identified with the derivatives of Fourier coefficients of certain Siegel Eisenstein series of genus 2.
\par
When $N$ is an odd, square-free positive integer, the relation has already been obtained for all the pair $(T,\mathsf{y})$ in the work of Sankaran, Shi, and Yang \cite[Theorem 2.14]{SSY22} by computing both sides explicitly based on the previous works of Yang \cite{Yn98} and Kudla, Rapoport and Yang \cite{KRY06}. When $T$ is nonsingular, the relation has been obtained for all the level $N$ in the previous work of the author \cite[Theorem 1.2.1]{Zhu23} by establishing difference formulas on both the analytic and geometric sides. In this article, we give proof for all the pairs $(T,\mathsf{y})$, and all the levels $N$.
\par
When the matrix $T$ is singular of rank 1. On the analytic side, we use difference formulas of local densities to relate the singular Fourier coefficients of an Eisenstein series of genus 2 to the nonsingular Fourier coefficients of an Eisenstein series of genus 1. On the geometric side, the cycle $\widehat{\mathcal{Z}}(T,\mathsf{y})$ is essentially the intersection of a codimension $1$ cycle and the metrized Hodge line bundle on $\Xc$. We compute this intersection number by investigating the irreducible components of the special fiber $\Xc_{\mathbb{F}_p}\coloneqq\Xc\times_{\textup{Spec}\,\Z}\textup{Spec}\,\mathbb{F}_p$ of the model $\Xc$, and the reduction mod $p$ of the cuspidal divisor of the curve $\Xc$.
\par
When the matrix $T=0$, both the analytic side and geometric side can be computed explicitly. Here again, the computation of the analytic side is based on the difference formulas of local densities, while the computation of the geometric side is based on the intersection of irreducible components of the special fiber $\Xc_{\mathbb{F}_p}$.

\subsection{Summary of the main results}
Let $\Delta(N)$ be the following rank 3 quadratic lattice over $\mathbb{Z}$,
\begin{equation}
    \Delta(N) = \left\{x=\begin{pmatrix}
    -Na & b\\
    c & a
    \end{pmatrix}:\,\, a,b,c\in\mathbb{Z}\right\}
\end{equation}
equipped with the quadratic form $x\mapsto \textup{det}(x)$. 
\par
We use $v$ to denote a place of $\mathbb{Q}$. For every finite place $v$, let $\delta_{v}(N)=\Delta(N)\otimes_{\mathbb{Z}}\mathbb{Z}_{v}$ be a rank 3 quadratic lattice over $\mathbb{Z}_{v}$. Let $\mathbb{A}$ be the ring of ad$\grave{\textup{e}}$les over $\mathbb{Q}$. Let $\mathbb{V}=\{\mathbb{V}_{v}\}$ be the incoherent collection of quadratic spaces of $\mathbb{A}$ of rank 3 nearby $\Delta(N)$ at $\infty$, i.e.,
\begin{equation}
    \mathbb{V}_{v}=\delta_{v}(N)\otimes\mathbb{Q}_{v}\,\,\textup{if $v<\infty$, and $\mathbb{V}_{\infty}$ is positive definite.}
    \label{incoherent}
\end{equation}
\par
There is a classical incoherent Eisenstein series $E(\mathsf{z},s,\la^{2})$ (cf. \S\ref{inco}) on the Siegel upper half space of genus 2,
\begin{equation*}
    \mathbb{H}_{2}=\{\mathsf{z}=\mathsf{x}+i\mathsf{y}\,\,\vert\,\,\mathsf{x}\in\textup{Sym}_{2}(\mathbb{R}),\mathsf{y}\in\textup{Sym}_{2}(\mathbb{R})_{>0}\}.
\end{equation*}
This is essentially the Siegel Eisenstein series associated to a standard Siegel-Weil section of the degenerate principal series. The Eisenstein series here has a meromorphic continuation and a functional equation relating $s \leftrightarrow -s$. The central value $E(\mathsf{z},0,\la^{2})=0$ by the incoherence. We thus consider its central derivative
\begin{equation*}
    \partial\textup{Eis}(\mathsf{z},\la^{2})\coloneqq\frac{\textup{d}}{\textup{d}s}\bigg\vert_{s=0}E(\mathsf{z},s,\la^{2}).
\end{equation*}
Associated to the standard additive character $\psi:\mathbb{A}/\mathbb{Q}\rightarrow\mathbb{C}^{\times}$, it has a decomposition into the central derivatives of the Fourier coefficients
\begin{equation*}
    \partial\textup{Eis}(\mathsf{z},\la^{2}) = \sum\limits_{T\in\textup{Sym}_{2}(\mathbb{Q})}\partial\textup{Eis}_{T}(\mathsf{z},\la^{2})
\end{equation*}
\par
On the geometric side, there is a regular integral model of the modular curve $\mathcal{Y}_{0}(N)_{\mathbb{C}}$ over $\mathbb{Z}$ defined by Katz and Mazur: for any scheme $S$, the groupoid $\mathcal{Y}_{0}(N)(S)$ consists of objects $(E\stackrel{\pi}\longrightarrow E^{\prime})$ where $E$, $E^{\prime}$ are elliptic curves over $S$ and $\pi$ is a cyclic isogeny such that $\pi^{\vee}\circ\pi=N$. They proved that $\mathcal{Y}_{0}(N)$ is 2-dimensional regular flat Deligne-Mumford stack (cf. \cite[Theorem 5.1.1]{KM85}), but $\mathcal{Y}_{0}(N)$ is not proper. There is a moduli stack $\mathcal{X}_{0}(N)$ defined in \cite{Ces17} which serves as a ``compactification'' of $\mathcal{Y}_{0}(N)$. It is a proper regular flat 2-dimensional Deligne-Mumford stack that contains $\mathcal{Y}_{0}(N)$ as an open substack, so we can consider the arithmetic intersection theory on $\mathcal{X}_{0}(N)$ following the lines in \cite{Gil09}.
\subsubsection{Special cycles of codimension 1}
The key concept is that of a special cycle. For every integer $t$, we define a closed substack of $\mathcal{Y}_{0}(N)$ as follows: For an object $(E\stackrel{\pi}\longrightarrow E^{\prime})$ of $\mathcal{Y}_{0}(N)(S)$, the stack $\mathcal{Z}(t,\la)$ parameterizes isogenies $j$ between $E$ and $E^{\prime}$ with $j^{\vee}\circ j=t$ and orthogonal to the cyclic isogeny $\pi$. It can be shown that the stack $\mathcal{Z}(t,\la)$ is a generalized Cartier divisor even on the stack $\Xc$ when $t>0$, and it's empty when $t\leq0$. For every positive number $y$, we define the modified special divisor $\mathcal{Z}^{\ast}(t,y,\la)$ to be
\begin{equation*}
    \mathcal{Z}^{\ast}(t,y,\la)=\mathcal{Z}(t,\la)+g(t,y,\la)\cdot\mathsf{Cusp}(\Xc).
\end{equation*}
where $\mathsf{Cusp}(\Xc)$ is the cuspidal divisor of the stack $\Xc$, and $g(t,y,\la)$ is a smooth function in $y$ defined in (\ref{modifying-function}), it's identically $0$ when $t>0$.
\par
For every pair $(t,y)$ such that $t$ is a nonzero integer and $y>0$, a green function $\mathfrak{g}(t,y,\la)$ of the divisor $\mathcal{Z}^{\ast}(t,y,\la)$ is constructed by Kudla \cite[(12.21)]{Kud97} (see also \cite[\S5]{DY19}). We will recall the construction in $\S$\ref{green-function}. Finally, we define the following element in the codimension 1 arithmetic Chow group $\ACH$ of $\Xc$ (Definition \ref{special-divisor}):
\begin{equation}
        \ad(t,y)=(\mathcal{Z}^{\ast}(t,y,\la),\mathfrak{g}(t,y,\la)).
        \label{nonzero-case}
    \end{equation}
These elements $\ad(t,y)$ are invariant under the Atkin-Lehner involution $W_N$ of the stack $\Xc$, i.e., $W_N^{\ast}\ad(t,y)=\ad(t,y)$ (Lemma \ref{invariant}).\par
When $t=0$, the definition of $\ad(0,y)$ is slightly different from (\ref{nonzero-case}). Recall that there is a metrized Hodge line bundle $\hod$ on the stack $\Xc$ (Example \ref{hodge-bundle}), however, this bundle is not invariant under the Atkin-Lehner involution $W_N$ (Corollary \ref{chayi}). Therefore we consider the following element
\begin{equation*}
    \widehat{\omega}=-\hod-W_N^{\ast}\hod,
\end{equation*}
the element $\widehat{\omega}$ is clearly invariant under the Atkin-Lehner involution. We define
\begin{equation}
    \ad(0,y,\la)=\widehat{\omega}+(\mathcal{Z}^{\ast}(0,y,\la),\mathfrak{g}(0,y,\la))-(0,\log y).
\end{equation}
Now for any rational number $t$ which is not an integer, we simply define $\ad(t,y)=0\in\ACH$. Let $\tau=x+iy\in\mathbb{H}_{1}^{+}$, we consider the following generating series with coefficients in $\ACH$,
\begin{equation*}
    \widehat\phi_1(\tau)=\sum\limits_{t\in\mathbb{Q}}\ad(t,y,\la)\cdot q^{t}
\end{equation*}
where $q^{t}=e^{2\pi i \,t\tau}$.
\begin{theorem}[Theorem \ref{global-modularity-2}]
    Let $N$ be a positive integer. The generating series $\widehat\phi_1$ is a nonholomorphic Siegel modular form of genus 1 and weight $\frac{3}{2}$ with values in $\ACH$.
    \label{global-modularity}
\end{theorem}
\begin{remark}
    When $N$ is square-free, the above theorem has been proved by Du and Yang \cite[Theorem 1.1]{DY19}. Similar results for Shimura curves are proved by Kudla, Rapoport, and Yang \cite[Theorem A]{KRY06}.
\end{remark}
\subsubsection{Special cycles of codimension 2}
We are mainly interested in the elements in the codimension 2 arithmetic Chow group $\Ach$. 
\par
Let $T\in \textup{Sym}_{2}(\mathbb{Q})$ be a symmetric matrix. If $T=\mathbf{0}_2$, for every positive definite matrix $\mathsf{y}\in\textup{Sym}_{2}(\mathbb{R})$, we define
\begin{equation*}
    \widehat{\mathcal{Z}}(\mathbf{0}_{2},\mathsf{y})=\widehat{\omega}\cdot \widehat{\omega}+(0,\log\textup{det}\mathsf{y}\cdot[\Omega])\in\Ach.
\end{equation*}
\par
If the rank of $T$ is 1, there exists a matrix $g\in\textup{GL}_2(\mathbb{Z})$ such that ${^{t}g}Tg=\textup{diag}\{0,t\}$ for some nonzero rational number $t$, let $\mathsf{y}=\textup{diag}\{y_1,y_2\}$ be a symmetric matrix with $y_{1},y_{2}>0$, we define the following element in $\Ach$
\begin{equation*}
    \widehat{\mathcal{Z}}(T,\mathsf{y})=
        \widehat{\mathcal{Z}}(t,y_2,\la)\cdot\widehat{\omega}-\left(0,\log y_1\cdot\delta_{\mathcal{Z}^{\ast}(t,y_2,\la)_\C}\right).
\end{equation*}
\par
If $T$ is nonsingular, a detailed definition of the element $\widehat{\mathcal{Z}}(T,\mathsf{y})$ can be found in \cite[(16)]{Zhu23}. Now for any element $\mathsf{z}=\mathsf{x}+i\mathsf{y}\in\mathbb{H}_{2}$, we consider the following generating series with coefficients in $\Ach$,
\begin{equation*}
    \widehat\phi_2(\mathsf{z})=\sum\limits_{T\in\textup{Sym}_2(\mathbb{Q})}\widehat{\mathcal{Z}}(T,\mathsf{y})\cdot q^{T}
\end{equation*}
where $q^{T}=e^{2\pi i \,\textup{tr}\,(T\mathsf{z})}$.
\par
There is an isomorphism $\Ach\simeq\mathbb{C}$ given by the arithmetic degree map $\widehat{\textup{deg}}: \widehat{\textup{CH}}^{2}_{\mathbb{C}}(\mathcal{X}_{0}(N))\rightarrow\mathbb{C}$ constructed by Kudla, Rapoport and Yang \cite[$\S$2.4]{KRY06} (see also (\ref{degreemap})). Our main result is the following
\begin{theorem}[Theorem \ref{mainglobal}]
    Let $N$ be a positive integer. The generating series $\widehat\phi_2$ is a nonholomorphic Siegel modular form of genus 2 and weight $\frac{3}{2}$. More precisely, under the isomorphism $\widehat{\textup{deg}}: \widehat{\textup{CH}}^{2}_{\mathbb{C}}(\mathcal{X}_{0}(N))\stackrel{\sim}\rightarrow\mathbb{C}$
\begin{equation*}
    \widehat\phi_2(\mathsf{z})= \frac{\psi(N)}{24}\cdot\partial\textup{Eis}(\mathsf{z},\la^{2}),
\end{equation*}
here $\psi(N)=N\cdot\prod\limits_{p\vert N}(1+p^{-1})$.
\label{mainglobal1}
\end{theorem}
\begin{remark}
    When $N$ is square-free, the above theorem has been proved by Sankaran, Shi, and Yang \cite{SSY22}. Similar results for Shimura curves are proved by Kudla, Rapoport, and Yang \cite[Theorem B]{KRY06}.
\end{remark}

\subsection{Strategy of the proof of main results}
In this section, we mainly explain the proof of Theorem \ref{mainglobal1} in this section, while Theorem \ref{global-modularity} will come as a byproduct.
\par
We will prove Theorem \ref{mainglobal1} term-by-term, i.e., for any symmetric matrix $T\in\textup{Sym}_2(\mathbb{Q})$, we prove that for any $\mathsf{z}=\mathsf{x}+i\mathsf{y}\in\mathbb{H}_2$,
\begin{equation}
    \widehat{\textup{deg}}\,\widehat{\mathcal{Z}}(T,\mathsf{y})\cdot q^{T}=\frac{\psi(N)}{24}\cdot\partial\textup{Eis}_T(\mathsf{z},\la^{2}).
    \label{decomposition}
\end{equation}
When $T$ is nonsingular, this has been proved in \cite[Theorem 1.2.1]{Zhu23}. In this article, we will focus on the case that $T$ is singular. In the following we refer to $\widehat{\textup{deg}}\,\widehat{\mathcal{Z}}(T,\mathsf{y})$ as the geometric side and $\partial\textup{Eis}_T(\mathsf{z},\la^{2})$ the analytic side.
\subsubsection{The singular coefficients of Eisenstein series}
Our primary goals in the analytic side are: (a). When the rank of $T$ is 1, we relate the singular Fourier coefficients of the Eisenstein series $E_{T}(\mathsf{z},s,\la^{2})$ to nonsingular Fourier coefficients of another Eisenstein series of lower genus. (b). When $T=0$, we compute the exact value of $\partial\textup{Eis}_{\mathbf{0}_2}(\mathsf{z},\la^{2})$.
\par
Before explaining the main idea, we give some necessary definitions. For any two integral quadratic $\mathbb{Z}_{p}$-lattices $L$ and $M$. Let $\textup{Rep}_{M,L}$ be the scheme of integral representations, an $\mathbb{Z}_{p}$-scheme such that for any $\mathbb{Z}_{p}$-algebra $R$, $\textup{Rep}_{M,L}(R)=\textup{QHom}(L\otimes_{\mathbb{Z}_{p}}R, M\otimes_{\mathbb{Z}_{p}}R)$, where $\textup{QHom}$ denotes the set of quadratic module homomorphisms. The local density of integral representations is defined to be
\begin{equation*}
    \textup{Den}(M,L)=\lim\limits_{d\rightarrow\infty}\frac{\#\textup{Rep}_{M,L}(\mathbb{Z}_{p}/p^{d})}{p^{d\cdot \textup{dim}(\textup{Rep}_{M,L})_{\mathbb{Q}_{p}}}}.
\end{equation*}
Let $H_{2}^{+}=\mathbb{Z}_{p}^{2}$ be the rank 2 quadratic $\mathbb{Z}_{p}$-lattice equipped with the quadratic form $q_{H_{2}^{+}}(x,y)=xy$. For any positive integer $k$, let $H_{2k}^{+}=\left(H_{2}^{+}\right)^{\obot k}$ be the rank $2k$ quadratic lattice obtained by orthogonal direct sum of $k$ copies of $H_2^{+}$.
\par
Let's explain the idea shortly. As is well-known, the Fourier coefficients of Eisenstein series are (linear combinations) of Whittaker functions which are essentially products of local representation densities of quadratic lattices. Our main tool is the difference formula of local representation densities of quadratic lattices proved in \cite[\S7.2]{Zhu23} (see also Theorem \ref{anadecom}) which roughly says that for any number $l\in\mathbb{Z}_p$ and any quadratic lattice $M$, the following two local densities are related
\begin{equation}
    \den(\langle -l\rangle\obot H_{2k+2}^{+},M)\leftrightarrow\den(H_{2k+4}^{+},M\obot\langle l\rangle).
    \label{diff-relation}
\end{equation}
This relation has two applications:\\
$(a)$.\, Let $l=N$, the quadratic lattice $\langle -l\rangle\obot H_{2k+2}^{+}$ is isometric to $\delta_p(N)\obot H_{2k}^{+}$. The relation (\ref{diff-relation}) reduces the computation of a local density function with level structure to that of an unramified local density function, which has very explicit expression and functional equation by the works of Cho and Yamauchi \cite[(3.4)]{CY20}, Ikeda \cite[Theorem 4.1]{Ike17}, see also the works of Li and Zhang \cite[$\S$3]{LZ22}.\\
$(b)$.\, Taking the limit $\nu_p(l)\rightarrow\infty$, the right-hand side of (\ref{diff-relation}) converges to the Whittaker function with singular coefficient, while the left-hand side of (\ref{diff-relation}) converges to $\textup{Den}(H_{2k+2}^{+},M)$ (Lemma \ref{raodong}), i.e., we have the following relation,
\begin{equation}
    \textup{Den}(H_{2k+2}^{+},M)=\lim\limits_{\nu_p(l)\rightarrow\infty}\den(H_{2k+4}^{+},M\obot\langle l\rangle).
    \label{singular-coe}
\end{equation}
\par
When the rank of $T$ is 1, we assume $T=\textup{diag}\{0,t\}$ with $t\neq0$ for simplicity, let $\mathsf{y}=\textup{diag}\{y_1,y_2\}$ be a positive definite symmetric matrix, then there is a well-known relation (see \cite[Lemma 5.4]{GS19})
\begin{equation*}
    E_{T}(i\mathsf{y},s,\Delta(N)^{2})=y_1^{s/2}y_2^{-3/4}W_{t}(g_{iy_2},s+\frac{1}{2},\Delta(N))+(y_1y_2)^{-3/4}W_{T}(g_{i\mathsf{y}},s,\Delta(N)^{2}),
\end{equation*}
where both of the terms $W_{t}(g_{iy_2},s+\frac{1}{2},\Delta(N))$ and $W_{T}(g_{i\mathsf{y}},s,\Delta(N)^{2})$ are product of local Whittaker functions (see (\ref{whit-int}) and (\ref{whit-sch}) for the precise definition).
\par
Since $t\neq0$, the term $W_{t}(g_{iy_2},s+\frac{1}{2},\Delta(N))$ is already the nonsingular Fourier coefficient of an Eisenstein series of the lower genus. Let's now consider the term $W_{T}(g_{i\mathsf{y}},s,\Delta(N)^{2})$, which is the product of local Whittaker functions $W_{T,v}$ over all the places $v$ of $\mathbb{Q}$. The relation between $W_{T,v}$ and $W_{t,v}$ is sketched in the following way: Let $p$ be a finite prime, for any integer $m$, let $T_m=\textup{diag}\{p^{m},t\}$ be a $2\times2$ nonsingular matrix. By the well-known relation between Whittaker functions and local densities (Proposition \ref{non-homo}), we have
\begin{equation*}
    W_{T,p}(1,k,1_{\delta_p(N)^{2}})=\lim\limits_{m\rightarrow\infty}W_{T_m,p}(1,k,1_{\delta_p(N)^{2}})\leftrightarrow\den(\delta_p(N)\obot H_{2k}^{+},\langle t\rangle\obot\langle p^{m}\rangle).
\end{equation*}
Applying $(a)$ of (\ref{diff-relation}) twice,
\begin{equation}
    \den(\delta_p(N)\obot H_{2k}^{+},\langle t\rangle\obot\langle p^{m}\rangle)\stackrel{l=N}\leftrightarrow\den(H_{2k+4}^{+},\langle t\rangle\obot\langle p^{m}\rangle\obot\langle N\rangle)\stackrel{l=p^{m}}\leftrightarrow\den(\langle-p^{m}\rangle\obot H_{2k+2}^{+},\langle t\rangle\obot\langle N\rangle).
    \label{step-1}
\end{equation}
Taking limit $m\rightarrow\infty$, the formula (\ref{singular-coe}) tells us that
\begin{equation}
    \lim_{m\rightarrow\infty}\den(\langle-p^{m}\rangle\obot H_{2k+2}^{+},\langle t\rangle\obot\langle N\rangle)=\den(H_{2k+2}^{+},\langle t\rangle\obot\langle N\rangle).
    \label{step-2}
\end{equation}
Applying $(a)$ again, we get
\begin{equation}
    \den(H_{2k+2}^{+},\langle t\rangle\obot\langle N\rangle)\stackrel{l=N}\leftrightarrow\den(\delta_{p}(N)\obot H_{2k-2}^{+},\langle t\rangle).
    \label{step-3}
\end{equation}
The last term $\den(\delta_{p}(N)\obot H_{2k-2}^{+},\langle t\rangle)$ is clearly related to $W_{t,p}(1,k-\frac{1}{2},1_{\delta_p(N)})$ by Proposition \ref{non-homo}. Therefore we have successfully built the bridge from $W_{T,v}$ and $W_{t,v}$ for a finite prime $v$.
\par
However, we warn the reader that all the symbols ``$\leftrightarrow$'' are not exact equal, it just means that there are some relations between the two sides of the symbol. Detailed calculations based on the principles above are given in $\S$\ref{singular-coefficients}. Our final result is the following proposition.
\begin{proposition}[Proposition \ref{2to1mei}]
    Let $T$ be a $2\times2$ matrix of rank 1 which can be diagonalized to $\textup{diag}\{0,t\}$ with $t\neq0$. Let $\mathsf{y}=\textup{diag}\{y_1,y_2\}$ be a positive definite symmetric matrix, then for any complex number $k$, we have
    \begin{align*}
        E_{T}(i\mathsf{y},k,\la^{2})&=y_{1}^{k/2}E_{t}(iy_2,\frac{1}{2}+k,\la)\\
        &+y_1^{-k/2}\cdot\frac{k-1}{k+1}\cdot N^{-k}\cdot\frac{\Lambda(2-2k)}{\Lambda(2+2k)}\cdot\prod\limits_{p\vert N}\beta_{p}(k)\cdot E_t(iy_2,\frac{1}{2}-k,\la),
    \end{align*}
    where $\beta_p(s)$ is a rational function in $p^{-s}$.
    \label{intro-analytic}
\end{proposition}
\begin{remark}
   The same formula in Proposition \ref{intro-analytic} has been proved when $N$ is squarefree in \cite[Lemma 4.13]{SSY22} by explicit computations of local densities. Our proof is more conceptual and can be easily generalized to higher dimensions by the principles displayed in (\ref{step-1}), (\ref{step-2}) and (\ref{step-3}).
\end{remark}
When $T=\mathbf{0}_{2}$, all the local Whittaker functions $W_{T,v}(1_{v},s,1_{\delta_v(N)^{2}})$ can be computed explicitly by combining Wedhorn's computations \cite[\S2.11]{Wed07} and the difference formula.

\subsubsection{Vertical components of the Hodge line bundle}
The primary goal on the geometric side is to compute the height pairings of (metrized) special divisors $\widehat{\mathcal{Z}}(t,y,\la)$ and the modified metrized Hodge line bundle $\widehat{\omega}$, and the self-intersection numbers of the metrized line bundle $\widehat{\omega}$.
\par
The main difficulty comes from the fact that the bundle $\widehat{\omega}$ contains vertical components. These components should be a linear combination of irreducible components of the reduction mod $p$ of the stack $\Xc$ for prime numbers $p$. We need to know the explicit multiplicities of these irreducible components appearing in $\widehat{\omega}$ and their intersections.
\par
We first study the special fiber of the stack $\Xc$. Let $p$ be a prime number, by the works of Katz and Mazur, if $n=\nu_p(N)\geq0$ is the $p$-adic valuation of the number $N$, then $\Xc_p\coloneqq\Xc\times_{\textup{Spec}\,\mathbb{Z}}\textup{Spec}\,\mathbb{F}_p$ has $n+1$ irreducible components $\mathcal{X}_p^{a}(N)$ and they meet each other at every supersingular point, where the index $a$ satisfies that $-n\leq a\leq n$ and has the same parity with $n$. Moreover, every component $\mathcal{X}_p^{a}(N)$ has two natural morphisms to the stack $\mathcal{X}_0(Np^{-n})_p$, one of them is an isomorphism while the other one is finite flat of degree $p^{\vert a\vert}$ (see Theorem \ref{insideY}). Later in $\S$\ref{int-on-p}, we compute the intersection numbers between these irreducible components based on the explicit local equations of these components at supersingular points obtained by Katz and Mazur \cite[Theorem 13.4.7]{KM85} (see also \cite[Corollary 6.2.7]{Zhu23}).
\par
After that, in $\S$\ref{explictsection} we use an explicit rational section $\Delta_N$ of the bundle $\hod^{\otimes12\varphi(N)}$ to give an explicit expression of this bundle in the Chow group $\CH^{1}(\Xc)$ by computing the element $\textup{div}(\Delta_N)\in\CH^{1}(\Xc)$ (Theorem \ref{div-element}), this idea originates from the work of Du and Yang \cite{DY19}, where they consider the case that $N$ is squarefree.
\par
Our idea of computing the element $\textup{div}(\Delta_N)$ for general $N$ comes from the observation that there is a correspondence between cusps and the special fibers of the stack $\Xc$. Let's explain this in the specific case that $N=p^{n}$ for some integer $n\geq0$. There is a cuspidal divisor $\mathsf{Cusp}(\mathcal{X}_0(p^{n}))$ on the stack $\mathcal{X}_0(p^{n})$, it is a disjoint union of $n+1$ connected components (Proposition \ref{jiandian}),
\begin{equation*}
    \mathsf{Cusp}(\mathcal{X}_0(p^{n}))=\coprod\limits_{\substack{-n\leq a\leq n\\a\equiv n\,\textup{mod}\,2}}\mathsf{C}^{a}(p^{n}).
\end{equation*}
Note that the index set is exactly the same as the index set for the irreducible components of the stack $\mathcal{X}_0(p^{n})_p$! Actually we will prove that the connected component $\mathsf{C}^{a}(p^{n})$ pulls back to the cusp of the curve $\mathcal{X}_p^{a}(p^{n})$ (Proposition \ref{luozainali}). Then we pick a specific cusp lying in the component $\mathsf{C}^{a}(p^{n})$ and consider the local expansion of the section $\Delta_{p^{n}}$ around this point, the multiplicity of $p$ appearing in the local expansion gives the multiplicity of the vertical component of $\mathcal{X}_p^{a}(p^{n})$ in the element $\textup{div}(\Delta_{p^{n}})$. The final result is
\begin{equation*}
        \textup{div}(\Delta_N)=\psi(N)\varphi(N)P_{\infty}(N)+\sum\limits_{p\vert N}f_p(N),
    \end{equation*}
    where $P_{\infty}(N)$ is the connected component of the cuspidal divisor containing the cusp $\infty$, and for any $p\vert N$, $f_p(N)$ is the following vertical divisor,
    \begin{equation*}
        f_p(N)=12p^{n-1}\varphi(N_p)\sum\limits_{\substack{-n\leq a <n\\a\equiv n\,\textup{mod 2}}}\left(\frac{1-p}{2}(n-a)-1\right)\varphi(p^{(n-\vert a \vert)/2})\cdot\mathcal{X}_p^{a}(N).
    \end{equation*}
This expression is the main tool to compute the self intersection of the modified Hodge bundle $\widehat{\omega}$ and the intersection pairing $\widehat{\mathcal{Z}}(t,y,\la)\cdot\widehat{\omega}$.

\subsection{Acknowledgement} 
The author is grateful to Professor Chao Li for his careful reading of the original manuscript and many helpful comments. The author is supported by the Department of Mathematics at Columbia University in the city of New York.

\part{The Analytic Side}
\section{Local densities and difference formulas}
\subsection{Notations on quadratic lattices}
Let $p$ be a prime number. Let $F$ be a nonarchimedean local field of residue characteristic $p$, with ring of integers $\mathcal{O}_{F}$, residue field $\kappa = \mathbb{F}_{q}$ of size $q$, and uniformizer $\pi$. Let $\nu_{\pi} : F \rightarrow \mathbb{Z}\cup\{\infty\}$ be the valuation on $F$ and $\vert\cdot\vert: F \rightarrow \mathbb{R}_{\geq0}$ be the normalized absolute value on $F$. Let $(\cdot,\cdot)_{F}$ be the Hilbert symbol on the local field $F$. 
\par
A quadratic lattice $(L,q_{L})$ is a finite free $\mathcal{O}_{F}$-module equipped with a quadratic form $q_{L}:L\rightarrow F$. The quadratic form $q_{L}$ also induces a symmetric bilinear form $L\times L\stackrel{(\cdot,\cdot)}\longrightarrow F$ by $(x,y)=q_L(x+y)-q_L(x)-q_L(y)$. Let $L^{\vee}=\{x\in L\otimes_{\mathcal{O}_{F}}F:(x,L)\subset \mathcal{O}_{F}\}$. We say a quadratic lattice is integral if $q_{L}(x)\in\mathcal{O}_{F}$ for all $x\in L$, is self-dual if it is integral and $L$ = $L^{\vee}$. There are three invariants associated to a quadratic lattice: the rank of $L$ over $\mathcal{O}_F$, the discriminant $\chi(L)\in\{0,\pm1\}$ and the Hasse invariant $\epsilon(L)\in\{\pm1\}$. We refer the readers to \cite[\S2]{Zhu23} and \cite[\S2]{LZ22} for details about the last two invariants.
\par
When $p$ is odd, we use $H_{k}^{\varepsilon}$ to denote the unique (up to isometry) self-dual lattice of rank $k$ and discriminant $\chi(L)=\varepsilon$ . When $p=2$, let $H_{2n}^{+}=(H_{2}^{+})^{\obot n}$ be a self-dual lattice of rank $2n$, where the quadratic form on $H_{2}^{+}=\mathcal{O}_{F}^{2}$ is given by $(x,y)\in \mathcal{O}_{F}^{2}\mapsto xy$.
\label{quadratic}
\begin{example}
Let $N\in\mathcal{O}_{F}$. Let $\delta_{F}(N)$ be the following rank 3 quadratic lattice over $\mathcal{O}_{F}$,
\begin{equation*}
    \delta_{F}(N) = \left\{x=\begin{pmatrix}
    -Na & b\\
    c & a
    \end{pmatrix}:\,\, a,b,c\in\mathcal{O}_{F}\right\}.
\end{equation*}
equipped with the quadratic form induced by $x\mapsto \textup{det}(x)$. Under the following basis of $\delta_{F}(N)$,
\begin{equation*}
    e_{1}=\begin{pmatrix}
    -N & \\
     & 1
    \end{pmatrix},\,\,e_{2}=\begin{pmatrix}
    \,  & 1\\
     \, & 
    \end{pmatrix},\,\,e_{3}=\begin{pmatrix}
     \, & \,\\
     1& \,
    \end{pmatrix}.
\end{equation*}
the quadratic form can be represented by the following symmetric matrix,
\begin{equation*}
    T=\begin{pmatrix}
    -N & 0 & 0 \\
      0 & 0 & -\frac{1}{2}\\
      0 & -\frac{1}{2} & 0
    \end{pmatrix}.
\end{equation*}
therefore $\chi(\delta_{F}(N))=\left(\frac{-N}{p}\right)$, where $\left(\frac{\cdot}{p}\right)$ is the extended quadratic residue symbol (see \cite[\S2.1]{Zhu23}), and $\epsilon(\delta_{F}(N))=(N,-1)_{F}$. Moreover,
\begin{equation*}
    \delta_{F}(N)^{\vee} = \left\{x=\begin{pmatrix}
    -Na & b\\
    c & a
    \end{pmatrix}:\,\, a\in\frac{1}{2N}\mathcal{O}_{F},b,c\in\mathcal{O}_{F}\right\}.
\end{equation*}
therefore $\delta_{F}(N)^{\vee}/\delta_{F}(N)\simeq\mathcal{O}_{F}/2N$.
\par
Throughout this article, we will mainly focus on the case that $F=\mathbb{Q}_{p}$. In this case, we simply use $\delta_{p}(N)$ to denote the lattice $\delta_{\mathbb{Q}_{p}}(N)$ (as we did in the introduction).
\label{lambda}
\end{example}

\subsection{Local densities of quadratic lattices}
\label{localdensity}
\begin{definition}
Let $L, M$ be two quadratic $\mathcal{O}_{F}$-lattices. Let $\textup{Rep}_{M,L}$ be the scheme of integral representations, an $\mathcal{O}_{F}$-scheme such that for any $\mathcal{O}_{F}$-algebra $R$,
\begin{equation*}
    \textup{Rep}_{M,L}(R)=\textup{QHom}(L\otimes_{\mathcal{O}_{F}}R, M\otimes_{\mathcal{O}_{F}}R),
\end{equation*}
where \textup{QHom} denotes the set of injective module homomorphisms which preserve the quadratic forms. The local density of integral representations is defined to be
\begin{equation*}
    \textup{Den}(M,L)=\lim\limits_{d\rightarrow\infty}\frac{\#\textup{Rep}_{M,L}(\mathcal{O}_{F}/\pi^{d})}{q^{d\cdot \textup{dim}(\textup{Rep}_{M,L})_{F}}}.
\end{equation*}
\end{definition}
\begin{definition}
Let $L, M$ be two quadratic $\mathcal{O}_{F}$-lattices. Let $\textup{PRep}_{M,L}$ be the $\mathcal{O}_{F}$-scheme of primitive integral representations such that for any $\mathcal{O}_{F}$-algebra $R$,
\begin{equation*}
    \textup{PRep}_{M,L}(R)=\{\phi\in\textup{Rep}_{M,L}(R):\textup{$\phi$ is an isomorphism between $L_{R}$ and a direct summand of $M_{R}$}\}.
\end{equation*}
where $L_{R}$ (resp. $M_{R}$) is $L\otimes_{\mathcal{O}_{F}}R$ (resp. $M\otimes_{\mathcal{O}_{F}}R$). The primitive local density is defined to be
\begin{equation*}
    \textup{Pden}(M,L)=\lim\limits_{d\rightarrow\infty}\frac{\#\textup{PRep}_{M,L}(\mathcal{O}_{F}/\pi^{d})}{q^{d\cdot \textup{dim}(\textup{Rep}_{M,L})_{F}}}.
\end{equation*}
\end{definition}
\par
For any $N\in F$, let $(\langle N\rangle,q_{\langle N\rangle})$ be the rank 1 $\mathcal{O}_{F}$-lattice with an $\mathcal{O}_{F}$ generator $l_{N}$ such that $q_{\langle N\rangle}(l_{N})=N$. For $a_1,a_2,\cdots,a_n\in F$, we define $\langle a_1,a_2,\cdots,a_n\rangle\coloneqq\langle a_1\rangle\obot\langle a_2\rangle\obot\cdots\obot\langle a_n\rangle$.
\begin{example}
    When $p$ odd, it has been calculated explicitly that for any $N\in\mathcal{O}_{F}$ (cf. \cite[(3.3.2.1)]{LZ22}), 
\begin{equation}
    \textup{Pden}(H_{k}^{\varepsilon}, \langle N\rangle) = \begin{cases}
    1-q^{1-k}, & \textup{when $k$ is odd and $\pi\,\vert\, N$;}\\
    1+\varepsilon\chi_{F}(N)q^{(1-k)/2}, &\textup{when $k$ is odd and $\pi\nmid N$;}\\
    (1-\varepsilon q^{-k/2})(1+\varepsilon q^{1-k/2}), & \textup{when $k$ is even and $\pi\,\vert\, N$;}\\
    1-\varepsilon q^{-k/2}, & \textup{when $k$ is even and $\pi\nmid N$.}
    \end{cases}
\end{equation}
When $p=2$, the same formula makes sense and holds true only in the case that $k$ is even and $\varepsilon=+1$.
\label{rank1}
\end{example}
\begin{lemma}
    Let $L, M$ be two quadratic $\mathcal{O}_{F}$-lattices such that the quadratic form on $L$ is nondegenerate, then there exists a positive number $m$ such that for any $a\in\mathcal{O}_F$ with $\nu_\pi(a)\geq m$, the following identity holds,
    \begin{equation*}
        \den(\langle a\rangle\obot M,L)=\den(M,L).
    \end{equation*}
    \label{raodong}
\end{lemma}
\begin{proof}
    We assume that the rank of $L$ is $n$ and the rank of $M$ is $m$. Let $\{l_i\}_{i=1}^{n}$ be an $\mathcal{O}_F$-basis of $L$, and let $T=(t_{ij})_{1\leq i,j\leq n}$ be the moment matrix of the this basis under the quadratic form on $L$. For any positive integer $d$, we have
    \begin{align*}
        \textup{Rep}_{\langle a\rangle\obot M,L}(\Of/\pi^{d})=\{(\overline{m}_1^{\pr},\cdots,\overline{m}_n^{\pr})\in (\langle a\rangle/\pi^{d}\langle a\rangle \obot M/\pi^{d}M)^{n}:(\overline{m}_i^{\pr},\overline{m}_j^{\pr})\equiv 2t_{ij}\,\textup{mod}\,2\pi^{d}\,\,\textup{for}\,i,j\geq 1.\}.
    \end{align*}
    \par
    For any element $(\overline{m}_1^{\pr},\cdots,\overline{m}_n^{\pr})\in (\langle a\rangle/\pi^{d}\langle a\rangle \obot M/\pi^{d}M)^{n}$ in the set $\textup{Rep}_{\langle a\rangle\obot M,L}(\Of/\pi^{d})$, let $(\overline{m}_1^{\pr\pr},\cdots,\overline{m}_n^{\pr\pr})\in (M/\pi^{d}M)^{n}$ be the $M/\pi^{d}M$-component of this element. Let $(m_1^{\pr\pr},\cdots,m_n^{\pr\pr})\in M^{n}$ be an arbitrary lift of the element $(\overline{m}_1^{\pr\pr},\cdots,\overline{m}_n^{\pr\pr})$ to $M$, when $\nu_\pi(a)$ and $d$ are sufficiently large, the sublattice spanned by $(m_1^{\pr\pr},\cdots,m_n^{\pr\pr})$ inside $M$ is isometric to $L$ due to the local constancy of equivalent class of nondegenerate quadratic lattice (see the work of Durfee \cite{Dur44}).\par
    Therefore when $\nu_\pi(a)$ and $d$ are sufficiently large, we have a natural map by taking the $M/\pi^{d}M$-component,
    \begin{equation*}
        \textup{Rep}_{\langle a\rangle\obot M,L}(\Of/\pi^{d})\longrightarrow \textup{Rep}_{M,L}(\Of/\pi^{d}).
    \end{equation*}
    hence $\# \textup{Rep}_{\langle a\rangle\obot M,L}(\Of/\pi^{d})=q^{nd}\cdot\#\textup{Rep}_{M,L}(\Of/\pi^{d})$.
    \par
    Notice that $\textup{dim}(\textup{Rep}_{M,L}=mn-\frac{n(n+1)}{2}$ and $\textup{dim}(\textup{Rep}_{\langle a \rangle\obot M,L}=(m+1)n-\frac{n(n+1)}{2}$, hence when $\nu_\pi(a)$ and $d$ are sufficiently large,
    \begin{equation*}
        \frac{\# \textup{Rep}_{\langle a\rangle\obot M,L}(\Of/\pi^{d})}{q^{d\cdot\textup{dim}(\textup{Rep}_{\langle a \rangle\obot M,L})}}=\frac{\# \textup{Rep}_{M,L}(\Of/\pi^{d})}{q^{d\cdot\textup{dim}(\textup{Rep}_{M,L})}}.
    \end{equation*}
    Taking limit $d\rightarrow\infty$, we get $\den(\langle a\rangle\obot M,L)=\den(M,L)$ under the assumption that $\nu_\pi(a)$ is sufficiently large.
\end{proof}
\begin{theorem}
Let $H$ be a self-dual quadratic $\mathcal{O}_{F}$-lattice of finite rank $k$. Let $M$ be an integral quadratic $\mathcal{O}_{F}$-lattice of finite rank $r$. Let $N\in\mathcal{O}_{F}$ be an element of valuation $n$, then 
\begin{equation*}
    \textup{Den}(H,M\obot\langle N\rangle) = \sum\limits_{i=0}^{[n/2]}q^{(2-k+r)i}\cdot\textup{Pden}(H, \langle \pi^{-2i}N\rangle)\cdot\textup{Den}(H(N,i), M).
\end{equation*}
\label{anadecom}
\end{theorem}
\begin{proof}
    This is proved in \cite[\S7]{Zhu23}.
\end{proof}
\begin{remark}
    Let's assume that $H=H_k^{\varepsilon}$ where $k\geq3$ is an integer, then \cite[Example 7.1.2]{Zhu23} shows that for any positive integer $0\leq i\leq [\frac{n}{2}]$ where $n=\nu_\pi(N)$,
    \begin{equation*}
        H_k^{\varepsilon}(N,i)\simeq\langle-N\pi^{-2i}\rangle \obot H_{k-2}^{\varepsilon}.
    \end{equation*}
    \label{qingchu}
\end{remark}
\begin{corollary}
    Let $k$ be a positive integer and $\varepsilon\in\{\pm1\}$, let $M$ be a quadratic lattice of rank $r$ over $\mathcal{O}_{F}$, then for any $a\in\mathcal{O}_{F}$, we have the following identity,
    \begin{equation*}
        \lim\limits_{m\rightarrow\infty}\den(H_k^{\varepsilon},M\obot\langle a\pi^{m}\rangle)=\den(H_{k-2}^{\varepsilon},M)\cdot\frac{\textup{Pden}(H_k^{\varepsilon},\langle \pi\rangle)}{1-q^{2-k+r}}.
    \end{equation*}
    \label{jian1}
\end{corollary}
\begin{proof}
    By Theorem \ref{anadecom} and Remark \ref{qingchu}, we know that for any positive integer $m$,
    \begin{equation*}
        \den(H_k^{\varepsilon},M\obot\langle a\pi^{m}\rangle)-q^{2-k+r}\den(H_k^{\varepsilon},M\obot\langle a\pi^{m-2}\rangle)=\textup{Pden}(H_k^{\varepsilon},\langle a\pi^{m}\rangle)\cdot\den(\langle -a\pi^{m}\rangle\obot H_{k-2}^{\varepsilon},M).
    \end{equation*}
    \par
    Notice that $\textup{Pden}(H_k^{\varepsilon},\langle a\pi^{m}\rangle)$ is independent of the positive integer $m$ by Example \ref{rank1}. Taking $m\rightarrow\infty$, Lemma \ref{raodong} implies the following,
    \begin{equation*}
        (1-q^{2-k+r})\cdot\lim\limits_{m\rightarrow\infty}\den(H_k^{\varepsilon},M\obot\langle a\pi^{m}\rangle)=\den(H_{k-2}^{\varepsilon},M)\cdot\textup{Pden}(H_k^{\varepsilon},\langle \pi\rangle).
    \end{equation*}
\end{proof}
\begin{definition}
Let $n\geq0$, for $\varepsilon\in\{\pm 1\}$, we define the normalizing factors to be
    \begin{align*}
    \textup{Nor}^{\varepsilon}(X,n)&=(1-\frac{1+(-1)^{n+1}}{2}\cdot\varepsilon q^{-(n+1)/2}X)\prod\limits_{1\leq i<(n+1)/2}(1-q^{-2i}X^{2}).
\end{align*}   
\label{normal1}
\end{definition}
\begin{lemma}
Let $L$ be a quadratic lattice of rank $n$. Let $N\in\mathcal{O}_F$. When $p$ is odd, there exist polynomials $\textup{Den}^{\varepsilon}(X,L)$, $\textup{Den}^{\flat\varepsilon}(X,L)$ and $\textup{Den}^{\varepsilon}_{\Delta(N)}(X,L)\in\mathbb{Z}[X]$, such that for positive integers $k$,
\begin{align*}
    &\textup{Den}^{\varepsilon}(X,L)\,\big\vert_{X=q^{-k}} = \frac{\textup{Den}(H_{2k+n+1}^{\varepsilon}, L)}{\textup{Nor}^{\varepsilon}(q^{-k},n)},\,\,\,\textup{Den}^{\flat\varepsilon}(X,L)\,\big\vert_{X=q^{-k}}=\frac{\textup{Den}(H_{2k+n}^{\varepsilon},L)}{\textup{Nor}^{\varepsilon}(q^{-k},n-1)}\cdot(1-\varepsilon\chi(L)q^{-k})\\
    &\textup{Den}^{\varepsilon}_{\Delta(N)}(X,L)\,\big\vert_{X=q^{-k}} = \left\{\begin{aligned}
    &\frac{\textup{Den}(\delta_{F}(N)\obot H_{2k+n-2}^{\varepsilon},L)}{\textup{Nor}^{\varepsilon}(q^{-k},n-1)} , & \textup{When $\pi\,\vert\, N$};\\
    &\frac{\textup{Den}(\delta_{F}(N)\obot H_{2k+n-2}^{\varepsilon},L)}{\textup{Nor}^{\varepsilon\chi_{F}(N)}(q^{-k},n)}, & \textup{When $\pi\nmid N$}.
    \end{aligned}\right.
\end{align*}
We define the derived local density of $L$ and the derived local density of $L$ with level $N$ to be
\begin{equation*}
    \partial\textup{Den}^{\varepsilon}(L) = -\frac{\textup{d}}{\textup{d}X}\bigg|_{X=1}\textup{Den}^{\varepsilon}(X,L),\,\,\, \partial\textup{Den}^{\varepsilon}_{\Delta(N)}(L) = -\frac{\textup{d}}{\textup{d}X}\bigg|_{X=1}\textup{Den}^{\varepsilon}_{\Delta(N)}(X,L).
\end{equation*}
\par
When $p=2$ and $n$ is odd, the polynomial $\textup{Den}^{+}(X,L)$ with the same evaluation formulas at $X=q^{-k}$ exist. When $p=2$ and $n$ is even, the polynomial $\textup{Den}^{+}_{\Delta(N)}(X,L)$ with the same evaluation formula at $X=q^{-k}$ exists. The derived local density $\partial\textup{Den}^{+}(L)$ of $L$ is defined similarly.
\label{poly}
\end{lemma}
\begin{proof}
This is a combination of \cite[Lemma 3.3.2, Definition 3.4.1, Definition 3.5.2]{LZ22}
\end{proof}
We have the following functional equation for $\textup{Den}^{\varepsilon}(X,L)$.
\begin{theorem}
    When the residue characteristic of $\Of$ is not 2. Let $L$ be a quadratic lattice of rank $n$. Then
    \begin{align*}
        \textup{Den}^{\varepsilon}(X,L)&=w^{\varepsilon}(L)\cdot X^{\textup{val}(L)}\cdot \textup{Den}^{\varepsilon}(\frac{1}{X},L),\\
        \textup{Den}^{\flat\varepsilon}(X,L)&=(q^{1/2}X)^{2[\frac{\textup{val}(L)}{2}]}\cdot \textup{Den}^{\flat\varepsilon}(\frac{1}{qX},L).
    \end{align*}
    where $\textup{val}(L)$ is the $\pi$-adic valuation of the moment matrix of $L$ under an $\mathcal{O}_F$-basis, and the sign of the functional equation is equal to 
    \begin{equation*}
        w^{\varepsilon}(L)\coloneqq (\textup{det}\,L,-(-1)^{\tbinom{n+1}{2}}u)_{F}\cdot\epsilon(L)\in\{\pm1\},
    \end{equation*}
    where $u\in\mathcal{O}_{F}^{\times}$ such that $u$ is a square if $\varepsilon=1$, is not a square if $\varepsilon=-1$.
    \label{fe}
\end{theorem}
\begin{proof}
    Both of these two formulas are proved by Ikeda \cite[Theorem 4.1 (1),(2)]{Ike17} when $n$ is odd and $\varepsilon=+1$, the same proof works in general.
\end{proof}
\begin{remark}
    When the field $F$ is dyadic, similar functional equation has been proved by Ikeda and Katsurada \cite[Proposition 2.1]{IK22}. In this article, we will need the functional equation in the following case: Let $F=\mathbb{Q}_2$ and $L=\langle N,t\rangle$ for some $N,t\in\Z$.  When $-tN$ is not a square, let $-d$ be the fundamental discriminant of the field extension $\Q(-tN)/\Q$. When $-tN$ is a square, let $d=-1$, then there exists a positive number $c$ such that
\begin{equation*}
    4Nt=c^{2}d.
\end{equation*}
\par
We have the following functional equation:
\begin{equation}
    \textup{Den}^{\flat+}(X,L)=(2^{1/2}X)^{2\nu_2(c)}\cdot \textup{Den}^{\flat+}(\frac{1}{2X},L).
    \label{at2}
\end{equation}
\label{for2}
\end{remark}

\section{Eisenstein series and Whittaker function}
\subsection{Incoherent Eisenstein series}
\label{inco}
For an positive integer $r$, let $\mathsf{W}_{r}$ be the standard symplectic space over $\mathbb{Q}$ of dimension 2$r$. Let $P=MN\subset \textup{Sp}(\mathsf{W}_{r})$ be the standard Siegel parabolic subgroup, which take the following form under the standard basis of $\mathsf{W}_{r}$,
\begin{align*}
    M(\mathbb{Q}) & = \left\{m(a)=\begin{pmatrix}a & 0\\0 & ^{t}a^{-1}\end{pmatrix}: a\in \textup{GL}_{r}(\mathbb{Q})\right\},\\
    N(\mathbb{Q}) & = \left\{n(b)=\begin{pmatrix}1_{2} & b\\0 & 1_{2}\end{pmatrix}: b\in \textup{Sym}_{r}(\mathbb{Q})\right\}.
\end{align*}
Let $\mathbb{A}$ be the ad$\grave{\textup{e}}$le ring over $\mathbb{Q}$. There is an isomorphism $\textup{Mp}(\mathsf{W}_{r,\mathbb{A}})\stackrel{\sim}\rightarrow \textup{Sp}(\mathsf{W}_{r})({\mathbb{A}})\times \mathbb{C}^{1}$ with the multiplication on the latter is given by the global Rao cycle, therefore we can write an element of $\textup{Mp}(\mathsf{W}_{r,\mathbb{A}})$ as $(g,t)$ where $g\in \textup{Sp}(\mathsf{W}_{r})({\mathbb{A}})$ and $t\in \mathbb{C}^{1}$.
\par
Let $P(\mathbb{A}) = M(\mathbb{A})N(\mathbb{A})$ be the standard Siegel parabolic subgroup of $\textup{Mp}(\mathsf{W}_{r,\mathbb{A}})$ where 
\begin{align*}
    M(\mathbb{A}) & = \left\{(m(a),t): a\in \textup{GL}_{r}(\mathbb{A}), t\in\mathbb{C}^{1}\right\},\\
    N(\mathbb{A}) & = \left\{n(b): b\in \textup{Sym}_{r}(\mathbb{A})\right\}.
\end{align*}
Recall the following incoherent collection of rank $3$ quadratic spaces $\mathbb{V}=\{\mathbb{V}_{v}\}$ over $\mathbb{A}$ we defined in (\ref{incoherent}),
\begin{equation*}
    \mathbb{V}_{v}=\delta_{v}(N)\otimes\mathbb{Q}_{v}\,\,\textup{if $v<\infty$, and $\mathbb{V}_{\infty}$ is positive definite.}
\end{equation*}
then we can verify immediately that $\prod\limits_{v}\epsilon(\mathbb{V}_{v})=-1$. 
\par
Let $\chi:\mathbb{A}^{\times}/\mathbb{Q}^{\times}\rightarrow\mathbb{C}^{\times}$ be the quadratic character given by $\chi(x)=\prod\limits_{v\leq\infty}\chi_{v}(x_{v}) =\prod\limits_{v\leq\infty}(x_{v},-N)_{v}$ for all $x=(x_{v})\in\mathbb{A}^{\times}$. Fix the standard additive character $\psi:\mathbb{A}/\mathbb{Q}\rightarrow\mathbb{C}^{\times}$ such that $\psi_{\infty}(x)=e^{2\pi ix}$. We may view $\chi$ as a character on $M(\mathbb{A})$ by
\begin{equation*}
    \chi(m(a),t)=\chi(\textup{det}(a))\cdot\gamma(\textup{det}(a),\psi)^{-1}\cdot t.
\end{equation*}
and extend it to $P(\mathbb{A})$ trivially on $N(\mathbb{A})$. Here $\gamma(\textup{det}(a),\psi)$ is the Weil index (see the work of Kudla \cite[p. 548]{Kud97}). We define the degenerate principal series to be the unnormalized smooth induction
\begin{equation*}
    I(s,\chi)\coloneqq\textup{Ind}_{P(\mathbb{A})}^{\textup{Mp}(\mathsf{W}_{r,\mathbb{A}})}(\chi\cdot\vert\cdot\vert_{\mathbb{Q}}^{s+(r+1)/2}),\,\,\,s\in\mathbb{C}.
\end{equation*}
For a standard section $\Phi(-,s)\in I(s,\chi)$ (i.e., its restriction to the standard maximal compact subgroup of $\textup{Mp}(\mathsf{W}_{r,\mathbb{A}})$ is independent of $s$), we define the associated Siegel Eisenstein series
\begin{equation*}
    E(g,s,\Phi)=\sum\limits_{\gamma\in P(\mathbb{Q})\backslash \textup{Sp}(\mathbb{Q})}\Phi(\gamma g,s),
\end{equation*}
which converges for $\textup{Re}(s)\gg 0$ and admits meromorphic continuation to $s\in\mathbb{C}$.
\par
Recall that $\mathscr{S}(\mathbb{V}^{r})$ is the space of Schwartz functions on $\mathbb{V}^{r}$. The fixed choice of $\chi$ and $\psi$ gives a Weil representation $\omega = \omega_{\chi,\psi}$ of $\textup{Mp}(\mathsf{W}_{r,\mathbb{A}})\times\textup{O}(\mathbb{V})$ on $\mathscr{S}(\mathbb{V}^{r})$. For $\boldsymbol\varphi\in\mathscr{S}(\mathbb{V}^{r})$, define a function
\begin{equation*}
    \Phi_{\boldsymbol\varphi}(g)\coloneqq\omega(g)\boldsymbol\varphi(0),\,\,\,g\in\textup{Mp}(\mathsf{W}_{r,\mathbb{A}}).
\end{equation*}
Then $\Phi_{\boldsymbol\varphi}(g)\in I(0,\chi)$. Let $\Phi_{\boldsymbol\varphi}(-,s)\in I(s,\chi)$ be the associated standard section, known as the standard Siegel–Weil section associated to $\boldsymbol\varphi$. For $\boldsymbol\varphi\in \mathscr{S}(\mathbb{V}^{r})$, we write $E(g,s,\boldsymbol\varphi)\coloneqq E(g,s,\Phi_{\boldsymbol\varphi})$.
\par
In this paper, we will mainly work on the case that $r=1$ or $2$.

\subsection{Whittaker functions and local densities}
Let $v$ be a finite place of $\mathbb{Q}$. Define the local degenerate principal series to be the unnormalized smooth induction
\begin{equation*}
     I_{v}(s,\chi_{v})\coloneqq\textup{Ind}_{P(\mathbb{Q}_{v})}^{\textup{Mp}(\mathsf{W}_{r,v})}(\chi_{v}\cdot\vert\cdot\vert_{v}^{s+(r+1)/2}),\,\,\,s\in\mathbb{C}.
\end{equation*}
For any function $\Phi_v\in I_{v}(s,\chi_{v})$, we define the local Whittaker function by the following integral,
\begin{equation}
    W_{T,v}(g_{v},s,\Phi_{v})=\int\limits_{N(\mathbb{Q}_v)}\Phi_v(w_r^{-1}n(b)g_v,s)\psi_v(-\textup{tr}(\frac{1}{2}Tb))\textup{d}n(b),\,\,\, w_r=\begin{pmatrix}
        0 & 1_{r}\\
        -1_{r} & 0
    \end{pmatrix}.
    \label{whit-int}
\end{equation}
The function $W_{T,v}(g_{v},s,\Phi_{v})$ has analytic continuation to $s\in\mathbb{C}$.
\par
The fixed choice of $\chi_{v}$ and $\psi_{v}$ gives a local Weil representation $\omega_{v}=\omega_{\chi_{v},\psi_{v}}$ of $\textup{Mp}(\mathsf{W}_{r,v})\times\textup{O}(\mathbb{V}_{v})$ on the Schwartz function space $\mathscr{S}(\mathbb{V}_{v}^{r})$. We define the local Whittaker function associated to $\boldsymbol\varphi_{v}$ and $T\in\textup{Sym}_{r}(\mathbb{Q}_{v})$ to be
\begin{equation}
    W_{T,v}(g_{v},s,\boldsymbol\varphi_{v})\coloneqq W_{T,v}(g_{v},s,\Phi_{\boldsymbol\varphi_{v}}),
    \label{whit-sch}
\end{equation}
where $\Phi_{\boldsymbol\varphi_{v}}(g_{v})\coloneqq\omega_{v}(g_{v})\boldsymbol\varphi_{v}(0)\in I_{v}(0,\chi_{v})$ and $\Phi_{\boldsymbol\varphi_{v}}(-,s)$ is the associated standard section.

\par
The relationship between Whittaker functions and local densities is encoded in the following proposition.
\begin{proposition}
Suppose $v\neq\infty$. Let $L$ be an integral quadratic $\mathbb{Z}_{v}$-lattice of rank $1$ or $2$. Suppose that the quadratic form of $L$ is represented by a number $t\in\mathbb{Q}_v$ (when the rank of $L$ is 1) or a matrix $T\in \textup{Sym}_{r}(\mathbb{Q}_{v})$ after a choice of $\mathbb{Z}_{v}$-basis of $L$, we have the following identity,
\begin{equation}
    W_{t,v}(1,k+\frac{1}{2},1_{\delta_v(N)})=(-1,-N)_v \vert 2N\vert_v^{1/2} \cdot\gamma(\mathbb{V}_v)\cdot \textup{Den}(\delta_v(N)\obot H_{2k}^{+}, L),\,\,\,\textup{when the rank of $L$ is 1}.
\end{equation}
\begin{equation}
    W_{T,v}(1,k,1_{\delta_v(N)^{2}})=\vert 2\vert_{v}^{3/2}\cdot \vert  N\vert_{v}\cdot\gamma(\mathbb{V}_{v})^{2}\cdot\textup{Den}(\delta_v(N)\obot H_{2k}^{+}, L),\,\,\,\textup{when the rank of $L$ is 2}.
\end{equation} 
where the constant $\gamma(\mathbb{V}_{v})=\gamma(\textup{det}(\mathbb{V}_{v}),\psi_{v})^{-1}\cdot\epsilon(\mathbb{V}_{v})\cdot\gamma(\psi_{v})^{-3}$, $\gamma(\textup{det}(\mathbb{V}_{v}),\psi_{v})$ and $\gamma(\psi_{v})$ are Weil indexes (cf. Appendix of \cite{Rao93}).
\label{non-homo}
\end{proposition}
\begin{proof}
Both of the two formulas are proved in \cite[Lemma 5.7.1]{KRY06}.    
\end{proof}
For the place $v=\infty$, there is also a well-developed theory about the Whittaker function at $\infty$. Let $W_{T,\infty}(g,s,\la^{2})$ (resp. $W_{t,\infty}(g,s,\la)$) be the Whittaker functions for the Gaussian function on the space $\mathbb{V}_{\infty}^{2}$ (resp. $\mathbb{V}_{\infty}$), both of them can be computed explicitly, for example, the works of Kudla, Rapoport and Yang \cite[Lemma 8.6]{KRY04}, \cite[Theorem 5.2.7, Proposition 5.7.7]{KRY06}.
\begin{definition}
    For a symmetric matrix $T\in\textup{Sym}_{r}(\Q)$, and a factorizable section $\Phi=\otimes_{v}\Phi_v\in I(s,\chi)$ where $\Phi_{\infty}$ is the standard section associated to the Gaussian function on the space $\mathbb{V}_{\infty}^{r}$, for a complex number $s\in\C$ such that $\textup{Re}(s)\gg 0$, we define the global Whittaker function to be
    \begin{equation*}
        W_T(g,s,\boldsymbol\varphi)=W_{T,\infty}(g_{\infty},s,\la^{r})\cdot\prod\limits_{v<\infty}W_{T,v}(g_v,s,\boldsymbol\varphi_v),\,\,g=(g_v)_{v}\in\textsf{Mp}(\mathsf{W}_{r,\mathbb{A}}).
    \end{equation*}
    the above product can be meromorphically continued to the whole complex plane.
    \label{gloWhitt}
\end{definition}

\subsection{Fourier coefficients}
We have a Fourier expansion of the Siegel Eisenstein series defined above,
\begin{equation*}
    E(g,s,\boldsymbol\varphi)=\sum\limits_{T\in\textup{Sym}_{r}(\mathbb{Q})}E_{T}(g,s,\boldsymbol\varphi),
\end{equation*}
where
\begin{equation*}
    E_{T}(g,s,\boldsymbol\varphi)=\int_{\textup{Sym}_{r}(\mathbb{Q})\backslash \textup{Sym}_{r}(\mathbb{A})}E(n(b)g,s,\boldsymbol\varphi)\psi(-\textup{tr}(Tb))\textup{d}n(b),
\end{equation*}
the Haar measure $\textup{d}n(b)$ is normalized to be self-dual with respect to $\psi$.
\begin{lemma}
    Let $T\in\textup{Sym}_r(\Q)$ be a nonsingular matrix, let $\boldsymbol{\varphi}\in\mathscr{S}(\mathbb{V}^{r})$ be a factorizable Schwartz function, we have
\begin{equation*}
    E_{T}(g,s,\boldsymbol\varphi)=W_{T}(g,s,\boldsymbol\varphi).
\end{equation*}
\label{E=WW}
\end{lemma}
\begin{proof}
  This lemma follows from unravelling the integral defining the function $E_T$, see \cite[\S11.2]{LZ22}.
\end{proof}

\subsection{Classical incoherent Eisenstein series}
\label{cinco}
The hermitian symmetric domain for $\textup{Sp}(\mathsf{W}_{r})$ is the Siegel upper half space
\begin{equation*}
    \mathbb{H}_{r}=\{\mathsf{z}=\mathsf{x}+i\mathsf{y}\,\,\vert\,\,\mathsf{x}\in\textup{Sym}_{r}(\mathbb{R}),\mathsf{y}\in\textup{Sym}_{r}(\mathbb{R})_{>0}\}.
\end{equation*}
\par
When $r=1$, let $\tau=x+iy\in\mathbb{H}^{+}$ with $x,y\in\mathbb{R}$ and $y$ is positive. Define the classical incoherent Eisenstein series to be
\begin{equation}
    E(\tau,s,\boldsymbol\varphi)=y^{-3/4}\cdot E(g_{\tau},s,\boldsymbol\varphi),\,\,g_{\tau}=[n(x)m(y^{1/2}),1]\in\textup{Mp}(\mathsf{W}_{r,\mathbb{A}}).
    \label{toclas1}
\end{equation}
\par
In this paper, we will focus on the case that $\boldsymbol\varphi =1_{\Delta(N)\otimes\hat{\Z}}\otimes\boldsymbol\varphi_{\infty}\in\mathscr{S}(\mathbb{V})$, where $1_{\Delta(N)\otimes\hat{\Z}}$ is the characteristic function of the rank $3$ $\hat{\Z}$-lattice $\Delta(N)\otimes\hat{\Z}$ and $\boldsymbol\varphi_{\infty}$ is the Gaussian function $\boldsymbol\varphi_{\infty}(x)=e^{-\pi q_{\infty}(x)}$. For the fixed choice of the Schwartz function $ \boldsymbol\varphi$ above and any $t\in\mathbb{Q}$, we write
\begin{equation}
    E(\tau,s,\Delta(N))\coloneqq E(\tau,s,1_{\Delta(N)\otimes\hat{\Z}}\otimes\boldsymbol\varphi_{\infty}),\,\,\,E_t(\tau,s,\Delta(N))\coloneqq E_t(\tau,s,1_{\Delta(N)\otimes\hat{\Z}}\otimes\boldsymbol\varphi_{\infty}).
\end{equation}
We will also need the normalized genus $1$ Eisenstein series $\mathcal{E}(\tau,s,\Delta(N))$ defined as follows,
\begin{equation}
    \mathcal{E}(\tau,s,\Delta(N))=C_N(s) E(\tau,s-\frac{1}{2},\Delta(N)),\,\,\textup{where}\,\,C_N(s)=-\frac{s}{4\pi}\Lambda(2s)\left(\prod\limits_{p\vert N}(1-p^{-2s})\right)N^{(3s+1)/2}.
    \label{twiestedeis}
    \end{equation}
    Similarly, for any $t\in\mathbb{Q}$, $\mathcal{E}_t(\tau,s,\Delta(N))=C_N(s) E_t(\tau,s-\frac{1}{2},\Delta(N))$.
\par
When $r=2$, let $\mathsf{z}=\mathsf{x}+i\mathsf{y}\in\mathbb{H}_{2}$ with $\mathsf{x},\mathsf{y}\in\textup{Sym}_{2}(\mathbb{R})$ and $\mathsf{y}={^{t}a}\cdot a$ is positive definite. Define the classical incoherent Eisenstein series to be
\begin{equation}
    E(\mathsf{z},s,\boldsymbol\varphi)=\chi_{\infty}(m(a))^{-1}\vert\textup{det}(m(a))\vert^{-3/2}\cdot E(g_{\mathsf{z}},s,\boldsymbol\varphi),\,\,g_{\mathsf{z}}=[n(x)m(a),1]\in\textup{Mp}(\mathsf{W}_{r,\mathbb{A}}).
    \label{toclas2}
\end{equation}
We write the central derivatives as,
\begin{equation}
    \partial\textup{Eis}(\mathsf{z},\boldsymbol\varphi) \coloneqq E^{\prime}(\mathsf{z},0,\boldsymbol\varphi),\,\,\partial\textup{Eis}_{T}(\mathsf{z},\boldsymbol\varphi) \coloneqq E_{T}^{\prime}(\mathsf{z},0,\boldsymbol\varphi).
    \label{classi}
\end{equation}
Then we have a Fourier expansion
\begin{equation*}
     \partial\textup{Eis}(\mathsf{z},\boldsymbol\varphi)=\sum\limits_{T\in\textup{Sym}_{2}(\mathbb{Q})}\partial\textup{Eis}_{T}(\mathsf{z},\boldsymbol\varphi).
\end{equation*}
\par
In this paper, we will focus on the case that $ \boldsymbol\varphi =1_{\left(\Delta(N)\otimes\hat{\Z}\right)^{2}}\otimes\boldsymbol\varphi_{\infty}\in\mathscr{S}(\mathbb{V}^{2})$, where $\boldsymbol{\varphi}=1_{\left(\Delta(N)\otimes\hat{\Z}\right)^{2}}$ is the characteristic function of the rank $3$ $\hat{\Z}$-lattice $\Delta(N)\otimes\hat{\Z}$ and $\boldsymbol\varphi_{\infty}$ is the Gaussian function $\boldsymbol\varphi_{\infty}(x)=e^{-\pi\,\textup{tr}\,T(x)}$. For the fixed choice of the Schwartz function $ \boldsymbol\varphi$ above and any $T\in\textup{Sym}_2(\mathbb{Q})$, we write
\begin{align}
    & E(\mathsf{z},s,\Delta(N)^{2})\coloneqq E(\mathsf{z},s,1_{\left(\Delta(N)\otimes\hat{\Z}\right)^{2}}\otimes\boldsymbol\varphi_{\infty}),\,\,\,E_T(\mathsf{z},s,\Delta(N)^{2})\coloneqq E_T(\mathsf{z},s,1_{\left(\Delta(N)\otimes\hat{\Z}\right)^{2}}\otimes\boldsymbol\varphi_{\infty}),\\
    &\partial\textup{Eis}(\mathsf{z},\Delta(N)^{2})\coloneqq \partial\textup{Eis}(\mathsf{z},1_{\left(\Delta(N)\otimes\hat{\Z}\right)^{2}}\otimes\boldsymbol\varphi_{\infty}),\,\,\,\partial\textup{Eis}_{T}(\mathsf{z},\Delta(N)^{2}) \coloneqq  \partial\textup{Eis}_{T}(\mathsf{z},1_{\left(\Delta(N)\otimes\hat{\Z}\right)^{2}}\otimes\boldsymbol\varphi_{\infty}).
        \label{eis}
\end{align}
Similarly, for any $T\in\textup{Sym}_2(\mathbb{Q})$, $W_T(g,s,\Delta(N)^{2})\coloneqq W_T(g,s,1_{\left(\Delta(N)\otimes\hat{\Z}\right)^{2}}\otimes\boldsymbol\varphi_{\infty})$.

\begin{lemma}
Let $T=\textup{diag}\{0,t\}$ with $t\neq0$, let $\mathsf{y}=\textup{diag}\{y_1,y_2\}$ be a positive definite symmetric matrix, then 
\begin{equation*}
    E_{T}(i\mathsf{y},s,\Delta(N)^{2})=y_1^{s/2}y_2^{-3/4}W_{t}(g_{iy_2},s+\frac{1}{2},\Delta(N))+(y_1y_2)^{-3/4}W_{T}(g_{i\mathsf{y}},s,\Delta(N)^{2}).
\end{equation*}
\label{EtoW}
\end{lemma}
\begin{proof}
    This lemma follows from \cite[Lemma 4.11, formula (38)]{SSY22}.  
\end{proof}

\section{Singular Fourier coefficients of Eisenstein series}
\label{singular-coefficients}
In this section, we study the Fourier coefficient $E_{T}(\mathsf{z},s,\la^{2})$ of the Eisenstein series defined in $\S$\ref{cinco} when the matrix $T$ is singular.\par 
We first fix some notations. For any prime number $p$, we define the local zeta function $\zeta_p(s)=(1-p^{-s})^{-1}$, we also define $\zeta_{\infty}(s)=\pi^{-s/2}\Gamma(\frac{s}{2})$. Let $\Lambda(s)=\prod\limits_{v}\zeta_v(s)$ be the completed zeta function where $v$ ranges over all the places of $\Q$, then $\Lambda(s)$ can be meromorphically continued to the whole complex plane and satisfies the functional equation $\Lambda(s)=\Lambda(1-s)$.
\subsection{The rank of $T$ is $0$} In this case, the matrix $T=\mathbf{0}_{2}$.
\begin{lemma}
    Let $n=\nu_{p}(N)\geq0$ be the $p$-adic valuation of $N$, let $A_{p}(s)$ be the following function for $s\in\C$,
    \begin{equation*}
        A_{p}(s)=\vert N\vert _{p}\cdot\begin{cases}
           \frac{1}{1+p^{-1-s}},&\textup{if $n=0$;}\\
           \frac{1}{1+p^{-1-s}}\cdot\frac{1-p^{(1-s)(n+1)}}{1-p^{1-s}}-p^{-2s}\frac{1}{1+p^{-1-s}}\cdot\frac{1-p^{(1-s)(n-1)}}{1-p^{1-s}}, &\textup{if $n\geq1$.}
        \end{cases}
    \end{equation*}
    The following identity holds,
    \begin{equation*}
        W_{\mathbf{0}_{2},p}(1,k,1_{\delta_{p}(N)^{2}})=\vert 2\vert_{p}^{3/2}(N,-1)_{p}\cdot\frac{\zeta_{p}(2k-1)}{\zeta_{p}(2k+2)}\cdot A_{p}(k).
    \end{equation*}
    \label{analytic00}
\end{lemma}
\begin{proof}
    For any rank $2$ nondegenerate quadratic lattice $M$ and integer $k\geq0$,
    \begin{equation}
        W_{T,p}(1,k,1_{\delta_{p}(N)^{2}})=\vert 2\vert^{3/2}\vert N\vert_{p}(N,-1)_{p}\cdot\textup{Nor}^{+}(X,1)\cdot\textup{Den}_{\Delta(N)}^{+}(X,M)\big\vert_{X=p^{-k}}.
        \label{WD}
    \end{equation}
    Suppose the Gross-Keating invariant of the quadratic lattice $M$ is $\textup{GK}(M)=(a,b)$ for integers $a\leq b$ and $a,b\gg n$.\\
    $\bullet$\,If $n\equiv a\,\,(\textup{mod}\,2)$, by the formula in \cite[$\S$2.11]{Wed07}, we have
    \begin{align*}
        \textup{Den}^{+}(X,M\obot\langle N\rangle)&=\frac{(1-(pX)^{n+1})(1-(pX^{2})^{(n+a)/2-i})}{(1-pX)(1-pX^{2})}+p^{(n+a)/2}X^{a}\frac{(1-X^{n+1})(1-(\Tilde{\xi}X)^{b-a+1})}{(1-X)(1-(\Tilde{\xi}X))}\\&-\frac{p^{(n+a)/2-1}X^{b+2}}{1-p^{-1}X^{2}}\left(\frac{1-X^{n+1}}{1-X}-\frac{(p^{-1}X^{2})^{(n+a)/2}-p^{(n-a)/2+1}X^{a-1}}{1-pX^{-1}} \right).
    \end{align*}
    where $\Tilde{\xi}\in\{\pm1\}$ depends on the lattice $M$.\\
    $\bullet$\,If $n\nequiv a\,\,(\textup{mod}\,2)$, by the formula in \cite[$\S$2.11]{Wed07}, we have
    \begin{align*}
        \textup{Den}^{+}(X,M\obot\langle N\rangle)&=\frac{1-(pX)^{n+1}}{(1-pX)(1-pX^{2})}-\frac{p^{(n+a+1)/2}(X^{n+a+1}-X^{a})}{(1-X^{-1})(1-pX^{2})}
        \\&-\frac{p^{(n+a-1)/2}X^{b+1}}{1-p^{-1}X^{2}}\left(\frac{1-X^{n+1}}{1-X}-\frac{(p^{-1}X^{2})^{(n+a+1)/2}-p^{(n-a+1)/2}X^{a}}{1-pX^{-1}} \right).
    \end{align*}
    In both cases, when $0<X\leq p^{-1}$, we have the following formula
    \begin{equation*}
        \lim_{a,b\rightarrow\infty}\textup{Den}^{+}(X,M\obot\langle N\rangle)=\frac{1-(pX)^{n+1}}{(1-pX)(1-pX^{2})}.
    \end{equation*}
    therefore when $0<X\leq p^{-1}$,
    \begin{align*}
        \lim_{a,b\rightarrow\infty}\textup{Den}^{+}_{\Delta(N)}(X,M)&=\lim_{a,b\rightarrow\infty}\left(\textup{Den}^{+}(X,M\obot\langle N\rangle)-X^{2}\textup{Den}^{+}(X,M\obot\langle p^{-2}N\rangle)\right)\\
        &=\frac{1-X^{2}+(p^{n-1}-p^{n+1})X^{n+1}}{(1-pX)(1-pX^{2})}.
    \end{align*}
    Note that $\textup{Nor}^{+}(X,1)=1-p^{-1}X$, the lemma follows by combining (\ref{WD}) and the following formula,
    \begin{equation*}
        W_{\mathbf{0}_{2},p}(1,k,1_{\delta_{p}(N)^{2}})=\vert 2\vert^{3/2}\vert N\vert_{p}(N,-1)_{p}\cdot(1-p^{-k-1})\cdot\lim_{a,b\rightarrow\infty}\textup{Den}^{+}_{\Delta(N)}(p^{-k},M)
    \end{equation*}
\end{proof}
In the following, let $\Lambda(s)=\Gamma(s)\pi^{-s/2}\zeta(s)$ be the normalized Riemann zeta function.
\begin{proposition}
    Let $\mathsf{z}=\mathsf{x}+i\mathsf{y}\in\mathbb{H}_{2}$, we have
    \begin{equation*}
        \partial\textup{Eis}_{\mathbf{0}_{2}}(\mathsf{z},\Delta(N)^{2})=\log\textup{det}(\mathsf{y})+2-4\frac{\Lambda^{\pr}(-1)}{\Lambda(-1)}-\sum\limits_{p\vert N}\frac{-np^{n+1}+2p^{n}+np^{n-1}-2}{p^{n-1}(p^{2}-1)}\cdot\log p
    \end{equation*}
    \label{analytic0}
\end{proposition}
\begin{proof}
    By similar arguments in \cite[$\S$4.4]{SSY22}, we have 
    \begin{equation*}
        (\textup{det}\mathsf{y})^{3/4}\cdot E_{\mathbf{0}_{2}}(\mathsf{z},s,\Delta(N)^{2})=W_{\mathbf{0}_{2}}(g_{\mathsf{z}},s,\Delta(N)^{2})+\sum\limits_{\gamma\in\Gamma_{\infty}\backslash\textup{SL}_{2}(\mathbb{Z})}B(m(\gamma)g,s)+\Phi_{1_{\Delta(N)^{2}}}(g,s)
    \end{equation*}
    where the derivative of the middle term $\sum\limits_{\gamma\in\Gamma_{\infty}\backslash\textup{SL}_{2}(\mathbb{Z})}B(m(\gamma)g,s)$ at $s=0$ is 0, and
    \begin{equation*}
        W_{\mathbf{0}_{2}}(g_{\mathsf{z}},s,\Delta(N)^{2})=\textup{det}(\mathsf{y})^{-s/2+3/4}\frac{(s-1)\Lambda(2s-1)}{(s+1)\Lambda(2s+2)}\cdot\prod\limits_{p\vert N}A_{p}(s),\,\,\,\Phi_{1_{\Delta(N)^{2}}}(g,s)=\textup{det}(\mathsf{y})^{s/2+3/4}.
    \end{equation*}
    The proposition follows from combining the above formulas and Lemma \ref{analytic00}.
\end{proof}

\subsection{The rank of $T$ is 1.}
Let $N$ be a positive integer, let $t$ be a nonzero integer, we fix a $2\times2$ matrix $T=\textup{diag}\{0,t\}$. Now we are going to define a quadratic character $\ch:\mathbb{A}^{\times}\rightarrow\C^{\times}$: When $-tN$ is not a square in $\Q$, the character $\ch$ is the quadratic Dirichlet character attached to the quadratic field extension $\Q(-tN)/\Q$; when $-tN$ is a square, let $\ch$ be the trivial character. When $-tN$ is not a square, let $-d$ be the fundamental discriminant of the field extension $\Q(-tN)/\Q$. When $-tN$ is a square, let $d=-1$, then there exists a positive number $c$ such that
\begin{equation*}
    4Nt=c^{2}d.
\end{equation*}
\par
For any prime number $p$, we define the local zeta function associated to $\ch$ to be $L_p(s,\ch)=(1-\ch(p)p^{-s})^{-1}$ and $L_{\infty}(s,\ch)$ to be the following,
\begin{equation*}
    L_{\infty}(s,\ch)=\vert d\vert^{s/2}\cdot\begin{cases}
        \pi^{-(s+1)/2}\Gamma(\frac{s+1}{2}), &\textup{if $t>0$};\\
        \pi^{-s/2}\Gamma(\frac{s}{2}), &\textup{if $t<0$}.
    \end{cases}
\end{equation*}
Let $\Lambda(s,\ch)\coloneqq\prod\limits_{v}L_v(s,\ch)$ be the completed $L$-function of the quadratic character $\ch$, then $\Lambda(s,\ch)$ can be meromorphically continued to the whole complex plane and satisfies the functional equation $\Lambda(s,\ch)=\Lambda(1-s,\ch)$.
\par
Define the following function for positive integers $k$,
\begin{equation}
    g_p(k)=\frac{\textup{Den}(H_{2k+4}^{+},\langle t, Np^{-2}\rangle)}{\textup{Den}(H_{2k+4}^{+},\langle t, N\rangle)},
    \label{gp}
\end{equation}
Notice that if the $p$-adic valuation of $N$ is $0$ or $1$, the function $g_p(k)=0$.
\begin{lemma}
    The function $g_p(k)$ is a rational function in $p^{-k}$, hence it can be meromorphically defined over $\C$. The function $g_p$ also has the following functional equation,
    \begin{equation*}
        g_p(k)=p^{2k+1}g(-k-1).
    \end{equation*}
    \label{functionaleq}
\end{lemma}
\begin{proof}
We only consider the case that the $p$-adic valuation of $N$ is greater or equal to $2$. By definition made in Lemma \ref{poly}, when $k$ is a positive integer,
    \begin{align*}
        g_p(k)=\frac{\textup{Den}(H_{2k+4}^{+},\langle t, p^{-2}N\rangle)}{\textup{Den}(H_{2k+4}^{+},\langle t, N\rangle)}=\frac{\textup{Den}^{\flat+}(p^{-k-1},\langle t,p^{-2}N\rangle)}{\textup{Den}^{\flat+}(p^{-k-1},\langle t,N\rangle)}=\frac{\textup{Den}^{\flat+}(X,\langle t,p^{-2}N\rangle)}{\textup{Den}^{\flat+}(X,\langle t,N\rangle)}\bigg\vert_{X=p^{-k-1}}.
    \end{align*}
    Therefore $g_p(k)$ is a rational function in $p^{-k}$ by Lemma \ref{poly}, hence it can be meromorphically defined over $\C$. The functional equation is a consequence of Theorem \ref{fe}.
\end{proof}
\begin{lemma}
     Let $p$ be a prime number, for any positive integers $k$,
    \begin{align}
        \frac{W_{T,p}(1,k,1_{\delta_p(N)^{2}})}{\textup{Den}^{\flat+}(p^{-k},\langle t, N\rangle)}&=\vert2\vert_p\vert N\vert_p^{1/2}\gamma(\mathbb{V}_p)^{2}\cdot\frac{1-p^{-2k}g_p(k-1)}{1-\ch(p)p^{-k}}\cdot\begin{cases}
            1-p^{-2k-2}, &\textup{if $p\nmid N$;}\\
            1-p^{-k-1}, &\textup{if $p\,\vert\, N$.}
        \end{cases},\label{2de}\\
        \frac{W_{t,p}(1,k+\frac{1}{2},1_{\delta_p(N))}}{\textup{Den}^{\flat+}(p^{-k-1},\langle t, N\rangle)}&=\vert 2N\vert_p^{1/2}\gamma(\mathbb{V}_p)\cdot(-1,N)_p\cdot\frac{1-p^{-2k-1}g_p(k)}{1-\ch(p)p^{-k-1}}\cdot\begin{cases}
            1-p^{-2k-2}, &\textup{if $p\nmid N$;}\\
            1-p^{-k-1}, &\textup{if $p\,\vert\, N$.}
        \end{cases}\label{1de}.
    \end{align}
    \label{W/D}
\end{lemma}
\begin{proof}
    By Proposition \ref{non-homo} and Theorem \ref{anadecom}, we know that
    \begin{align*}
        &W_{T,p}(1,k,1_{\delta_p(N)^{2}})=\vert2\vert_p^{3/2}\vert N\vert_p\gamma(\mathbb{V}_p)^{2}\cdot\lim\limits_{m\rightarrow\infty}\textup{Den}(\delta_p(N)\obot H_{2k}^{+},\langle t, p^{m}\rangle)\\
        &=\vert2\vert_p^{3/2}\vert N\vert_p\gamma(\mathbb{V}_p)^{2}\cdot\lim\limits_{m\rightarrow\infty}\frac{\den(H_{2k+4}^{+},\langle N, t, p^{m}\rangle)-p^{-2k}\den(H_{2k+4}^{+},\langle Np^{-2}, t, p^{m}\rangle)}{\textup{Pden}(H_{2k+4}^{+},\langle N\rangle)}.\\
    \end{align*}
    By Lemma \ref{jian1}, we have
    \begin{align*}
        W_{T,p}(1,k,1_{\delta_p(N)^{2}})=&\vert2\vert_p^{3/2}\vert N\vert_p\gamma(\mathbb{V}_p)^{2}\cdot\frac{\textup{Pden}(H_{2k+4}^{+},\langle p\rangle)}{\textup{Pden}(H_{2k+4}^{+},\langle N\rangle)}\\ &\cdot\frac{\den(H_{2k+2}^{+},\langle N,t\rangle)-p^{-2k}\den(H_{2k+2}^{+},\langle Np^{-2},t\rangle)}{1-p^{-2k}}\\
        =&\vert2\vert_p^{3/2}\vert N\vert_p\gamma(\mathbb{V}_p)^{2}\cdot\den(H_{2k+2}^{+},\langle N,t\rangle)\cdot\frac{\textup{Pden}(H_{2k+4}^{+},\langle p\rangle)}{\textup{Pden}(H_{2k+4}^{+},\langle N\rangle)}\cdot(1-p^{-2k}g_p(k)).
    \end{align*}
    Then the formula (\ref{2de}) follows from the definition of local density polynomial in Lemma \ref{poly}.
    \par
    For the formula (\ref{1de}), we have the following identity by Proposition \ref{non-homo} and Theorem \ref{anadecom},
    \begin{align*}
        W_{t,p}(1,k+\frac{1}{2},1_{\delta_p(N)})&=\vert 2N\vert_p^{1/2}\gamma(\mathbb{V}_p)\cdot(-1,N)_p\cdot\frac{\den(H_{2k+4}^{+},\langle N,t\rangle)-p^{-2k-1}\den(H_{2k+4}^{+},\langle Np^{-2},t\rangle)}{\textup{Pden}(H_{2k+4}^{+},\langle N\rangle)}\\
        &=\vert 2N\vert_p^{1/2}\gamma(\mathbb{V}_p)\cdot(-1,N)_p\cdot\den(H_{2k+4}^{+},\langle N,t\rangle)\cdot\frac{1-p^{-2k-1}g_p(k)}{\textup{Pden}(H_{2k+4}^{+},\langle N\rangle)}.
    \end{align*}
    Then the formula (\ref{1de}) also follows from the definition of local density polynomial in Lemma \ref{poly}.
    \end{proof}
For any prime number $p$, we consider the following meromorphic function for $s\in\C$,
\begin{equation}
    \beta_p(s)=\frac{1+p^{s-1}}{1+p^{-s-1}}\cdot\frac{1-p^{-2s}g_p(s-1)}{1-g_p(s-1)}.
    \label{beta}
\end{equation}
\begin{corollary}
    Let $p$ be a prime number, the following identity holds for any integer $k$,
    \begin{align*}
        \frac{W_{T,p}(1,k,1_{\delta_p(N)^{2}})\cdot\frac{\zeta_p(2k)}{L_p(k,\ch)}}{W_{t,p}(1,\frac{1}{2}-k,1_{\delta_p(N)})\cdot\frac{\zeta_p(2-2k)}{L_p(1-k,\ch)}}=&\vert2\vert_p\vert N\vert_p^{1/2}(-1,N)_p\gamma(\mathbb{V}_p)\cdot p^{(1-2k)\nu_p(c)}\cdot\frac{\zeta_p(2k)}{\zeta_p(2k+2)}\\
        &\times\begin{cases}
            1, & \textup{if $p\nmid N$;}\\
            \beta_p(k), &\textup{if $p\,\vert\, N$}.
        \end{cases}
    \end{align*}
    \label{2to1}
\end{corollary}
\begin{proof}
    By Theorem \ref{fe} and the formula (\ref{at2}) in Remark \ref{for2}, there is a functional equation,
    \begin{equation*}
        \den^{\flat+}(p^{-k},\langle N,t\rangle)=p^{(1-2k)\nu_p(c)}\cdot\den^{\flat+}(p^{k-1},\langle N,t\rangle).
    \end{equation*}
    The corollary follows from the functional equation above and Lemma \ref{W/D}.
\end{proof}
\begin{proposition}
    Let $T$ be a $2\times2$ matrix of rank 1 which can be diagonalized to $\textup{diag}\{0,t\}$ with $t\neq0$. Let $\mathsf{y}=\textup{diag}\{y_1,y_2\}$ be a positive definite symmetric matrix, then for any complex number $k$, we have
    \begin{align*}
        E_{T}(i\mathsf{y},k,\la^{2})&=y_{1}^{k/2}E_{t}(iy_2,\frac{1}{2}+k,\la)\\
        &+y_1^{-k/2}\cdot\frac{k-1}{k+1}\cdot N^{-k}\cdot\frac{\Lambda(2-2k)}{\Lambda(2+2k)}\cdot\prod\limits_{p\vert N}\beta_{p}(k)\cdot E_t(iy_2,\frac{1}{2}-k,\la).
    \end{align*}
    \label{2to1mei}
\end{proposition}
\begin{proof}
When $-tN$ is not a square, let $-d$ be the fundamental discriminant of the field extension $\Q(-tN)/\Q$. When $-tN$ is a square, let $d=-1$, then there exists a positive number $c$ such that
\begin{equation*}
    4Nt=c^{2}d.
\end{equation*}
\par
    For any complex number $k$, the following identity is proved in \cite[Lemma 4.14]{SSY22} (see also \cite[Proposition 5.7.7]{KRY06}),
    \begin{align}
        W_{T,\infty}(g_{i\mathsf{y}},k,\la^{2})\cdot\frac{\zeta_{\infty}(2k)}{L_{\infty}(k,\ch)}=&-2iy_{1}^{-k/2+3/4}\cdot\frac{1-k}{1+k}\cdot\frac{\zeta_{\infty}(2k)}{\zeta_{\infty}(2k+2)}\cdot c^{2k-1}N^{1/2-k}\label{zaiwoq}\\
        &\times W_{t,\infty}(g_{iy_{2}},\frac{1}{2}-k,\la)\cdot\frac{\zeta_{\infty}(2-2k)}{L_{\infty}(1-k,\ch)}.\notag
    \end{align}
    Recall that by Definition \ref{gloWhitt}, $W_{T}(g_{i\mathsf{y}},k,\la^{2})=W_{T,\infty}(g_{i\mathsf{y}},k,\la^{2})\cdot\prod\limits_{p}W_{T,p}(1,k,1_{\delta_p(N)^{2}})$ and $W_{t}(g_{iy_2},k,\la)=W_{t,\infty}(g_{iy_2},k,\la)\cdot\prod\limits_{p}W_{t,p}(1,k,1_{\delta_p(N)})$.
    \par
    Notice that both of the functions $W_{T}(g_{i\mathsf{y}},k,\la^{2})$ and $W_{t}(g_{iy_{2}},k,\la)$ can be meromorphically continued to whole complex plane. By combining the above identity (\ref{zaiwoq}), Corollary \ref{2to1}, and the functional equations $\Lambda(s)=\Lambda(1-s)$, $\Lambda(s,\ch)=\Lambda(1-s,\ch)$, we get
    \begin{equation*}
        W_{T}(g_{i\mathsf{y}},k,\la^{2})=y_{1}^{-k/2+3/4}\cdot\frac{k-1}{k+1}\cdot N^{-k}\cdot\frac{\Lambda(2-2k)}{\Lambda(2+2k)}\cdot\prod\limits_{p\vert N}\beta_{p}(k)\cdot W_t(g_{iy_2},\frac{1}{2}-k,\la).
    \end{equation*}
    Then Lemma \ref{EtoW}, Lemma \ref{E=WW} and formulas (\ref{toclas1}), (\ref{toclas2}) imply that
    \begin{align*}
        E_{T}(i\mathsf{y},k,\la^{2})&=y_{1}^{k/2}E_{t}(iy_2,\frac{1}{2}+k,\la)\\
        &+y_1^{-k/2}\cdot\frac{k-1}{k+1}\cdot N^{-k}\cdot\frac{\Lambda(2-2k)}{\Lambda(2+2k)}\cdot\prod\limits_{p\vert N}\beta_{p}(k)\cdot E_t(iy_2,\frac{1}{2}-k,\la)
    \end{align*}
\end{proof}
\begin{corollary}
    Let $T$ be a $2\times2$ matrix of rank 1 which can be diagonalized to $\textup{diag}\{0,t\}$ with $t\neq0$. Let $\mathsf{y}=\textup{diag}\{y_1,y_2\}$ be a positive definite symmetric matrix, we have
    \begin{align*}
        E_{T}^{\pr}(i\mathsf{y},0,\Delta(N)^{2})=&2E_t^{\pr}(iy_2,\frac{1}{2},\Delta(N))\\
        &+\left(\log y_1+2+\log N+\frac{4\Lambda^{\pr}(2)}{\Lambda(2)}-\sum\limits_{p\vert N}\beta_p^{\pr}(0)\right)E_t(iy_2,\frac{1}{2},\Delta(N)).
    \end{align*}
    \label{jiexicezhuyaos}
\end{corollary}
\begin{proof}
    This follows from Proposition \ref{2to1mei}.
\end{proof}

\part{The Geometric Side}
\section{The geometry of the modular curve $\Xc$}
\subsection{Integral model}
Let $N$ be a positive integer, the classical modular curve $\mathcal{Y}_{0}(N)_{\mathbb{C}}$ over $\mathbb{C}$ is defined as the following smooth $1$-dimensional complex curve,
\begin{equation*}
    \mathcal{Y}_{0}(N)_{\mathbb{C}}\coloneqq\textup{GL}_{2}(\mathbb{Q})\backslash \mathbb{H}^{\pm}\times\textup{GL}_{2}(\mathbb{A}_{f})/\Gamma_{0}(N)(\hat{\mathbb{Z}})\simeq \Gamma_{0}(N)\backslash\mathbb{H}^{+},
\end{equation*}
where $\mathbb{H}^{\pm}=\mathbb{C}\backslash\mathbb{R}$ and $\mathbb{H}^{+}=\{z=x+iy\in\mathbb{C}:x\in\mathbb{R}, y\in\mathbb{R}_{>0}\}$ is the upper half plane. The group $\Gamma_{0}(N)(\hat{\mathbb{Z}})$ is the following open compact subgroup of $\textup{GL}_{2}(\mathbb{A}_{f})$,
\begin{equation*}
   \Gamma_{0}(N)(\hat{\mathbb{Z}}) = \left\{x=\begin{pmatrix}
    a & b\\
    Nc & d
    \end{pmatrix}\in \textup{GL}_{2}(\hat{\mathbb{Z}})\,\,:\,\,a,b,c,d\in\hat{\mathbb{Z}}\right\},
\end{equation*}
and $\Gamma_{0}(N)=\Gamma_{0}(N)(\hat{\mathbb{Z}})\bigcap \textup{SL}_{2}(\mathbb{Z})$. The modular curve $\Yc_{\C}$ parametrizes cyclic $N$-isogenies of elliptic curves over the complex number $\C$ by the following map:
\begin{equation*}
    \LN\tau\in \Gamma_{0}(N)\backslash\mathbb{H}^{+}\cdot\tau\longmapsto (E_{\tau}\stackrel{\times N}\rightarrow E_{N\tau}),
\end{equation*}
where for a given $\tau\in\mathbb{H}^{+}$, the elliptic curve $E_{\tau}$ is given by the complex torus $\C/(\Z+\Z\tau)$. The smooth curve $\mathcal{Y}_{0}(N)_{\mathbb{C}}$ is not proper, its compactification $\mathcal{X}_{0}(N)_{\mathbb{C}} \coloneqq\mathcal{Y}_{0}(N)_{\mathbb{C}}\cup\{\textup{cusps}\}$ is a smooth projective curve over $\mathbb{C}$.
\par
Katz and Mazur \cite{KM85} constructs an integral model of $\Yc$ over $\Z$ by extending the concept of cyclic isogeny to arbitrary base scheme, we refer the readers to \cite[\S3(3.4)]{KM85} and the arthor's previous work \cite[\S4]{Zhu23} for details. The moduli interpretation of the stack $\Yc$ is given as follows: for a scheme $S$, the objects of the groupoid $\mathcal{Y}_{0}(N)(S)$ are cyclic $N$-isogenies $(E\stackrel{\pi}\rightarrow E^{\prime})$ where $E$ and $E^{\prime}$ are elliptic curves over $S$, a morphism between two cyclic isogenies $(E_{1}\stackrel{\pi_{1}}\rightarrow E_{1}^{\prime})$ and $(E_{2}\stackrel{\pi_{2}}\rightarrow E_{2}^{\prime})$ is a pair of isomorphisms of elliptic curves $a:E_{1}\stackrel{\sim}\rightarrow E_{2}$ and $a^{\prime}:E_{1}^{\prime}\stackrel{\sim}\rightarrow E_{2}^{\prime}$ such that $a^{\prime}\circ\pi_{1}  = \pi_{2}\circ a$. 
\par
Cesnavicius \cite{Ces17} defined a compactification of $\Yc$ as the moduli stack of cyclic $N$-isogenies between generalized elliptic curves, we denote this compactification by $\Xc$, its complex fiber is isomorphic to $\Xc_{\C}$.
\par
There are two natural morphisms from $\Yc$ to $\mathcal{Y}_0(1)$ given as follows
\begin{align*}
    p_{1},p_{2}:\Yc&\longrightarrow\mathcal{Y}_0(1)\\
    (E_{1}\stackrel{\pi}\rightarrow E_2)&\stackrel{p_{1}}\longmapsto E_1;\\
    (E_{1}\stackrel{\pi}\rightarrow E_2)&\stackrel{p_{2}}\longmapsto E_2.
\end{align*}
Both of the morphism $p_{1},p_{2}:\Yc\rightarrow\mathcal{Y}_0(1)$ extends to morphisms from $\Xc$ to $\mathcal{X}_0(1)$, we still use the symbols $p_1$ and $p_2$ to denote them.
\begin{theorem}
The stack $\mathcal{X}_{0}(N)$ is a regular proper 2-dimensional Deligne-Mumford stack, both of the morphisms $p_1$ and $p_2$ are finite flat of degree $\psi(N)\coloneqq N\cdot\prod\limits_{p\vert N}(1+p^{-1})$ over $\mathcal{X}_0(1)$. Moreover, the stack $\Xc$ is smooth over $\Z[\frac{1}{N}]$
\end{theorem}
\begin{proof}
This is proved in \cite[Theorem 5.13]{Ces17}, the degree of $p_1$ and $p_2$ are computed in \cite[(13.4.9)]{KM85}.
\end{proof}
\begin{lemma}
    Let $N=M_{1}\cdot M_{2}$ where $\textup{g.c.d.}(M_{1},M_{2})=1$, there is a natural isomorphism
    \begin{equation*}
        \Xc\simeq\mathcal{X}_{0}(M_{1})\times_{p_{1},\mathcal{X}_{0}(1),p_{1}}\mathcal{X}_{0}(M_{2}).
    \end{equation*}
    \label{ddec}
\end{lemma}
\begin{proof}
    This follows from \cite[Corollary 1.10.15]{KM85}.
\end{proof}

\subsection{The Atkin-Lehner involution on $\Xc$}
\begin{lemma}
    Let $S$ be a scheme, and $\pi: E\rightarrow E^{\prime}$ be a cyclic $N$-isogeny between elliptic curves over $S$, then $\pi^{\vee}:E^{\pr}\rightarrow E$ is also a cyclic $N$-isogeny.
    \label{atkin}
\end{lemma}
\begin{proof}
We only need to show that the order $N$ quotient group scheme $E[N]/\textup{ker}(\pi)$ is cyclic, and this is proved in \cite[Corollary 5.5.4(3)]{KM85}.
\end{proof}
With Lemma \ref{atkin}, we can define the Atkin-Lehner involution $W_{N}$ on the stack $\Yc$ by the following,
\begin{align*}
    W_{N}:\Yc&\longrightarrow\Yc\\
    (E\stackrel{\pi}\rar E^{\pr})&\longmapsto(E^{\pr}\stackrel{\pi^{\vee}}\rar E)
\end{align*}
\begin{lemma}
    The Atkin-Lehner involution $W_{N}$ extends to the compactified stack $\Xc$.
\end{lemma}
\begin{proof}
    By \cite[Proposition 4.2.7 (c)]{Ces17}, the cyclicity condition of a $N$-isogeny is closed, hence $\Xc$ is a closed substack of $\mathcal{X}_{0}(1)\times\mathcal{X}_{0}(1)$. Let $\iota:\mathcal{X}_{0}(1)\times\mathcal{X}_{0}(1)\rightarrow\mathcal{X}_{0}(1)\times\mathcal{X}_{0}(1)$ be the involution which switches the two copies of $\mathcal{X}_{0}(1)$, let $\Xc^{\pr}=\Xc_{\mathcal{X}_{0}(1)\times\mathcal{X}_{0}(1), \iota}\mathcal{X}_{0}(1)\times\mathcal{X}_{0}(1)$ be the image of $\Xc$ under the involution $W_N$, it is also a closed substack of $\mathcal{X}_{0}(1)\times\mathcal{X}_{0}(1)$, but $\Xc$ and $\Xc^{\pr}$ have common dense open dense substack $\Yc$, hence $\Xc=\Xc^{\pr}$, therefore the Atkin-Lehner involution $W_{N}$ extends to the compactified stack $\Xc$.
\end{proof}

\subsection{Cusps of the modular curve $\Xc$}
Let $\mathcal{H}\coloneqq\mathcal{X}_0(1)\times\mathcal{X}_0(1)$ be the product of the smooth Deligne-Mumford stack $\mathcal{X}_0(1)$. Let $P_{\infty}(1):\textup{Spec}\,\Z\rightarrow\mathcal{X}_0(1)$ be the cuspidal divisor of the stack $\mathcal{X}_0(1)$ which corresponds to the standard $1$-gon over $\Z$ (we refer to \cite[\S2.1]{Con07} for detailed definitions of standard $1$-gon). We can view $P_{\infty}(1)$ as an element in $\textup{CH}^{1}(\mathcal{X}_0(1))$, then we define the following element in $\textup{CH}^{1}(\mathcal{H})$:
\begin{equation*}
    \mathsf{Cusp}(\mathcal{H})\coloneqq P_{\infty}(1)\times\mathcal{X}_{0}(1)+\mathcal{X}_{0}(1)\times P_{\infty}(1).
\end{equation*}
The two curves $P_{\infty}(1)\times\mathcal{X}_{0}(1)$ and $\mathcal{X}_{0}(1)\times P_{\infty}(1)$ intersects transversally at the point $(P_{\infty}(1),P_{\infty}(1))$.
\par
The natural morphism $\Xc\stackrel{(p_1,p_2)}\longrightarrow\mathcal{H}$ is a closed immersion. We define the cuspidal divisor of $\Xc$ to be the pullback of $\mathsf{Cusp}(\mathcal{H})$, i.e.,
\begin{equation*}
    \mathsf{Cusp}(\Xc)=(p_1,p_2)^{\ast}\mathsf{Cusp}(\mathcal{H})\in\textup{CH}^{1}(\mathcal{X}_0(N)).
\end{equation*}
Let $\cusp(\Xc)$ be the formal completion of $\mathcal{X}_0(N)$ along the cuspidal divisor $\textup{Cusp}(\Xc)$. It is well-known that $\cusp(\mathcal{X}_0(1))\simeq\textup{Spf}\,\Z[[q]]$. The two morphisms $p_1,p_2:\Xc\rightarrow\mathcal{X}_0(1)$ induce two morphisms from $\cusp(\Xc)$ to $\cusp(\mathcal{X}_0(1))$, we still use $p_1$ and $p_2$ to denote them, both of them are finite flat of degree of $\psi(N)$.
\begin{proposition}
    Let $N=p^{n}$ for some non-negative integer $n$, then the formal scheme $\cusp(\mathcal{X}_0(p^{n}))$ is a disjoint union as follows:
    \begin{equation*}
        \cusp(\mathcal{X}_0(p^{n}))=\coprod\limits_{\substack{-n\leq a\leq n\\a\equiv n\,\textup{mod}\,2}}\mathsf{C}^{a}(p^{n}),
    \end{equation*}
    where every $\mathsf{C}^{a}(p^{n})$ is finite flat over $\cusp(\mathcal{X}_0(1))$ via the morphisms $p_1$ and $p_2$, such that $W_N(\mathsf{C}^{a}(p^{n}))=\mathsf{C}^{-a}(p^{n})$. Let $\textup{deg}_1(\mathsf{C}^{a}(p^{n}))$ (resp. $\textup{deg}_2(\mathsf{C}^{a}(p^{n}))$) be the finite flat degree of $\mathsf{C}^{a}(p^{n})$ over $\cusp(\mathcal{X}_0(1))$ via the morphism $p_1$ (resp. $p_2$), then
     \begin{equation*}
        \textup{deg}_1(\mathsf{C}^{a}(p^{n}))=\begin{cases}\varphi(p^{(n-a)/2}), &\textup{if $0\leq a\leq n$};\\
        p^{-a}\varphi(p^{(n+a)/2}), &\textup{if $-n\leq a<0$}.\\   
    \end{cases}
    \end{equation*}
     \begin{equation*}
        \textup{deg}_2(\mathsf{C}^{a}(p^{n}))=\begin{cases}p^{a}\varphi(p^{(n-a)/2}), &\textup{if $0\leq a\leq n$};\\
        \varphi(p^{(n+a)/2}), &\textup{if $-n\leq a<0$}.\\   
    \end{cases}
    \end{equation*}
    More explicitly, if we view $\cusp(\mathcal{X}_0(p^{n}))$ as a $\Z[[q]]$-formal scheme via the morphism $p_1$, then 
    \begin{equation}
        \mathsf{C}^{a}(p^{n})\simeq\begin{cases}\textup{Spf}\,\Z[\zeta_{p^{(n-a)/2}}][[q]], &\textup{if $0\leq a\leq n$};\\
        \textup{Spf}\,\Z[\zeta_{p^{(n+a)/2}}][[q]][z]/(z^{p^{-a}}-\zeta_{p^{(n+a)/2}}q), &\textup{if $-n\leq a<0$},\\   
    \end{cases}
    \label{cusps}
    \end{equation}
    where for any $k\geq1$, $\zeta_{p^{k}}$ is a primitive $p^{k}$-th root of unity.
    \label{jiandian}
\end{proposition}
\begin{proof}
    It is proved by Edixhoven in \cite[\S1.2.2]{Edi90} that if we view $\cusp(\mathcal{X}_0(p^{n}))$ as a $\Z[[q]]$-formal scheme via the morphism $p_1$, then 
     \begin{equation}
        \textup{Spf}\,\Z[[q]]\coprod\textup{Spf}\,\Z[[q^{p^{-n}}]]\coprod\bigsqcup\limits_{\substack{c+d=n\\ c\geq d>0}}\textup{Spf}\,\Z[\zeta_{p^{d}}][[q]]\coprod\bigsqcup\limits_{\substack{c+d=n\\ d> c>0}}\textup{Spf}\,\Z[\zeta_{p^{c}}][[q]][z]/(z^{p^{d-c}}-\zeta_{p^{c}}q),
        \label{cusp}
    \end{equation}
    therefore $\cusp(\mathcal{X}_0(p^{n}))$ is a disjoint union of $n+1$ formal schemes, let $\mathsf{C}^{a}(p^{n})$ be one of the formal schemes according to formula (\ref{cusps}), then $\cusp(\mathcal{X}_0(p^{n}))=\coprod\limits_{\substack{-n\leq a\leq n\\a\equiv n\,\textup{mod}\,2}}\mathsf{C}^{a}(p^{n})$ by (\ref{cusp}). The finite flat degree of $\mathsf{C}^{a}(p^{n})$ over $\cusp(\mathcal{X}_0(1))$ via the morphisms $p_1$ and $p_2$ can be computed explicitly by (\ref{cusps}).
\end{proof}
\begin{remark}
    The isomorphism
    \begin{align*}
        &\cusp(\mathcal{X}_0(p^{n}))\simeq\\
        &\textup{Spf}\,\Z[[q]]\coprod\textup{Spf}\,\Z[[q^{p^{-n}}]]\coprod\bigsqcup\limits_{\substack{c+d=n\\ c\geq d>0}}\textup{Spf}\,\Z[\zeta_{p^{d}}][[q]]\coprod\bigsqcup\limits_{\substack{c+d=n\\ d> c>0}}\textup{Spf}\,\Z[\zeta_{p^{c}}][[q]][z]/(z^{p^{d-c}}-\zeta_{p^{c}}q),
    \end{align*}
    and Lemma \ref{ddec} imply that the formal scheme $\cusp(\Xc)$ has a unique closed and open formal subscheme which is isomorphic to $\textup{Spf}\,\Z[[q]]$ as $\Z[[q]]$-scheme via the morphism $p_1$, we denote the underlying scheme by $P_{\infty}(N)$, it corresponds to a morphism $P_{\infty}(N):\textup{Spec}\,\Z\rightarrow\Xc$. When $N=1$, this definition agrees with our previous definition of $P_{\infty}(1)$.
    \par
    Let $P_{0}(N)=W_N\circ P_{\infty}(N):\textup{Spec}\,\Z\rightarrow\Xc$ be the composition of the automorphism $W_N$ of $\Xc$ and the $\Z$-point $P_{\infty}(N)$. When $N=1$, we have $P_{\infty}(1)=P_{0}(1)$.
    \label{lianggejidian}
\end{remark}
\par
There is an explicit description of the cusps of the complex modular curve $\Xc_{\C}$ (cf. \cite[\S3.8]{DS05}),
\begin{equation*}
    P_{\frac{aQ}{N}},\,\,\textup{where}\,\,Q\vert N\,\,\textup{and}\,\,a\in(\Z/(Q,N/Q))^{\times}.
\end{equation*}
Especially, when $N=p^{n}$ for some non-negative integer $n$,  the cusps of the complex modular curve $\mathcal{X}_0(p^{n})_{\C}$ are
\begin{equation*}
    P_{\frac{a}{p^{k}}},\,\,\textup{where}\,\,0\leq k\leq n\,\,\textup{and}\,\,a\in(\Z/p^{\textup{min}\{k,n-k\}})^{\times}.
\end{equation*}
\begin{lemma}
    Let $P$ be a cusp of the complex modular curve $\mathcal{X}_0(p^{n})_{\C}$, let $\textup{ra}_1(P)$ (resp. $\textup{ra}_2(P)$) be the ramification degree of the morphism $p_1$ (resp. $p_2$) at $P$. Suppose $P=P_{\frac{a}{p^{k}}}$ for some integer $0\leq k\leq n$ and $a\in(\Z/p^{\textup{min}\{k,n-k\}})^{\times}$, then
    \begin{equation}
        \textup{ra}_1(P_{a/p^{k}})=\begin{cases}1, &\textup{if $\frac{n}{2}\leq k\leq n$};\\
        p^{n-2k}, &\textup{if $0\leq k<\frac{n}{2}$}.\\   
    \end{cases}
    \label{ra1}
    \end{equation}
    \begin{equation}
        \textup{ra}_2(P_{a/p^{k}})=\begin{cases}p^{2k-n}, &\textup{if $\frac{n}{2}\leq k\leq n$};\\
        1, &\textup{if $0\leq k<\frac{n}{2}$}.\\   
    \end{cases}
    \label{ra2}
    \end{equation} 
    \label{ramify}
\end{lemma}
\begin{proof}
    For any cusp point $P$, let $\textup{Stab}(P)\in\Gamma(p^{n})$ be the stabilizer of the cusp. For the cusp point $P_{\infty}(p^{n})$, we have
    \begin{equation*}
        \textup{Stab}(P_{\infty}(p^{n}))=\left\{\begin{pmatrix}
            1 & x\\
            0 & 1
        \end{pmatrix}:x\in\Z.\right\}.
    \end{equation*}
    \par
    Let $A\in\Z$ be a lift of $a$ to $\Z$, then $A$ is prime to $p$, hence there exists integers $b,d$ such that $Ab-dp^{k}=1$. Let 
    \begin{equation*}
        \gamma=\begin{pmatrix}
            A & d\\
            p^{k} & b
        \end{pmatrix},
    \end{equation*}
    then $\textup{Stab}(P_{\frac{a}{p^{k}}})=\gamma\textup{Stab}(P_{\infty}(p^{n}))\gamma^{-1}\cap\Gamma_0(p^{n})$. By simple calculations,
    \begin{equation*}
        \gamma\textup{Stab}(P_{\infty}(p^{n}))\gamma^{-1}\cap\Gamma_0(p^{n})=\left\{\begin{pmatrix}
            1-p^{k}Ax & xA^{2}\\
            -p^{2k}x & 1+p^{k}Ax
        \end{pmatrix}: x\in p^{n-2k}\Z\cap\Z.\right\},
    \end{equation*}
    the formula (\ref{ra1}) follows from the calculations above. Notice that $p_2=p_1\circ W_N$, hence the formula (\ref{ra2}) follows from (\ref{ra1}).
\end{proof}
\begin{corollary}
    Let $P_{\frac{a}{p^{k}}}$ be a cusp of the complex modular curve $\mathcal{X}_0(p^{n})_{\C}$, where $0\leq k\leq n$ and $a\in(\Z/p^{\textup{min}\{k,n-k\}})^{\times}$, then $P_{\frac{a}{p^{k}}}$ belongs to the component $\mathsf{C}^{2k-n}(p^{n})$ of $\cusp(\mathcal{X}_0(p^{n}))$, i.e., $P_{\frac{a}{p^{k}}}\in \mathsf{C}^{2k-n}(p^{n})(\C)$.
    \label{cuspC}
\end{corollary}
\begin{proof}
    This follows from formula (\ref{cusps}) in Proposition \ref{jiandian}, and formulas (\ref{ra1}) and (\ref{ra2}) in Lemma \ref{ramify}.
\end{proof}

\subsection{Reduction mod $p$ of $\Xc$}
Let $p$ be a fixed prime, the reduction mod $p$ of the stack $\Xc$ has been studied extensively in \cite[Chapter 13]{KM85}.
For an $\F$-scheme $S$, an integer $m\geq0$ and an elliptic curve over $S$, let $E^{(p^{m})}$ be the $m$-th Frobenius twist of $E$, we use $F^{m}:E\rar E^{(p^{m})}$ to denote the $m$-th iterated Frobenius morphism which has degree $p^{m}$, and $V^{m}:E^{(p^{m})}\rar E$ be the $m$-th iterated Verschiebung morphism which is also the dual morphism of $F^{m}$.
\begin{lemma}
    For any two integer $m\geq0$, the $m$-th iterated Frobenius morphism $F^{m}$ and the $m$-th iterated Verschiebung morphism $V^{m}$ are cyclic of degree $p^{m}$.
\end{lemma}
\begin{proof}
    This is proved in \cite[Thereom 13.3.5 (2)]{KM85}.
\end{proof}
Let $n\geq0$ be an integer, for any integer $a$ such that $0\leq a \leq n$ and $n\equiv a\,\,(\textup{mod}\,2)$, let $t=\frac{n-a}{2}$, we define two stacks $\mathcal{Y}^{a}$ and $\mathcal{Y}^{-a}$ in the following paragraphs.
\par
Let $\mathcal{Y}^{a}$ be the following stack: for an $\F$-scheme $S$, the groupoid $\mathcal{Y}^{a}(S)$ consists of cyclic $p^{n}$-isogenies $(E\stackrel{\pi}\rar E^{\pr})$ such that there exists an isomorphism $\varepsilon:E^{(p^{a})}\simeq E^{\pr}$ making the following diagram commutes,
\begin{equation*}
     \xymatrix{
    E\ar[d]_{F^{n-t}}\ar[r]^{\pi}&E^{\pr}\\
    E^{(p^{n-t})}\ar[r]^{\varepsilon^{(p^{t})}}_{\sim}& E^{\pr(p^{t})}.\ar[u]_{V^{t}}}
\end{equation*}
\par
Let $\mathcal{Y}^{-a}$ be the following stack: for an $\F$-scheme $S$, the groupoid $\mathcal{Y}^{-a}(S)$ consists of cyclic $p^{n}$-isogenies $(E\stackrel{\pi}\rar E^{\pr})$ such that there exists an isomorphism $\varepsilon:E\simeq E^{\pr(p^{a})}$ making the following diagram commutes,
\begin{equation*}
     \xymatrix{
    E\ar[d]_{F^{t}}\ar[r]^{\pi}&E^{\pr}\\
    E^{(p^{t})}\ar[r]^{\varepsilon^{(p^{t})}}_{\sim}& E^{\pr(p^{n-t})}.\ar[u]_{V^{n-t}}}
\end{equation*}
\begin{equation}
    \mathcal{Y}^{a}\rightarrow\mathcal{Y}_{0}(1)\times\mathcal{Y}_{0}(1).
    \label{diagonal}
\end{equation}
\begin{theorem}
    Let $n\geq0$ be an integer, let $\mathcal{Y}_{0}(p^{n})_{\F}\coloneqq\mathcal{Y}_{0}(p^{n})\times_{\Z}\F$ be the reduction mod $p$ of the stack $\mathcal{Y}_{0}(p^{n})$. For any integer $a$ such that $-n\leq a \leq n$ and $n\equiv a\,\,(\textup{mod}\,2)$, the stack $\mathcal{Y}^{a}$ has the following properties:\\
    (a)\,The stack $\mathcal{Y}^{a}$ is a 1-dimensional Deligne-Mumford stack.\\
    (b)\,The natural morphism $\mathcal{Y}^{a}\rar\mathcal{Y}_{0}(p^{n})_{\F}$ is a closed immersion, the composite morphism $\mathcal{Y}^{a}\rar\mathcal{Y}_{0}(p^{n})_{\F}\rar\mathcal{Y}_{0}(1)_{\F}$ is an isomorphism if $a\geq0$; is finite flat of degree $p^{-a}$ if $a\leq0$. \\
    (c)\,The following equality holds: $W_{N}(\mathcal{Y}^{a})=\mathcal{Y}^{-a}$.\\
    (d)\,The stack $\mathcal{Y}^{a}$ contains every supersingular point of $\mathcal{Y}_{0}(p^{n})_{\F}$. Let $P=(E\stackrel{\pi}\rar E^{\pr})$ be a supersingular $\F$-point of $\mathcal{Y}_{0}(p^{n})_{\F}$. Let $\mathcal{O}_{P}$ be the completed local ring of the stack $\mathcal{Y}_{0}(1)_{\F}\times\mathcal{Y}_{0}(1)_{\F}$ at $P$, let $\mathcal{O}_{a,P}$ be the completed local ring of $\mathcal{Y}^{a}$ at the point corresponding to the pair $(E,E^{\pr})$, then there exists an isomorphism $\mathcal{O}_{P}\simeq\F[[t_{1},t_{2}]]$ such that the closed immersion $\mathcal{Y}^{a}\rar \mathcal{Y}_{0}(1)_{\F}\times\mathcal{Y}_{0}(1)_{\F}$ induces the following isomorphism,
    \begin{equation*}
        \mathcal{O}_{a,P}\simeq
    \begin{cases}
       \F[[t_{1},t_{2}]]/(t_{1}-t_{2}^{p^{a}}), & \textup{if $a\geq 0$;}\\
       \F[[t_{1},t_{2}]]/(t_{1}^{p^{-a}}-t_{2}), & \textup{if $a\leq0$.}\\
    \end{cases}
    \end{equation*}\\
    (e)\,Over the open substack $\mathcal{Y}_{0}(p^{n})_{\F}^{\textup{ord}}=\mathcal{Y}_{0}(p^{n})_{\F}-\{\textit{supersingular points}\}$ of $\mathcal{Y}_{0}(p^{n})_{\F}$, the morphism $\bigsqcup\limits_{\substack{-n\leq a\leq n\\n\equiv a\,(\textup{mod}\,2)}}\mathcal{Y}^{a}\rightarrow\mathcal{Y}_{0}(p^{n})_{\F}^{\textup{red}}$ is an isomorphism.
    \label{insideY}
\end{theorem}
\begin{proof}
    These are the main results of \cite[\S13]{KM85}.
\end{proof}
\begin{remark}
    Let $\mathcal{X}^{a}$ be the scheme-theoretic closure of $\mathcal{Y}^{a}$ inside $\mathcal{X}_{0}(p^{n})_{\F}\coloneqq\mathcal{X}_{0}(p^{n})\times_{\Z}\F$, then Theorem \ref{insideY} is still true if we replace all the symbol $\mathcal{Y}$ by the symbol $\mathcal{X}$.
\end{remark}
\begin{theorem}
  Let $n=\nu_{p}(N)$ be the $p$-adic valuation of $N$, let $N_{p}=p^{-n}N$. For any integer $a$ such that $-n\leq a \leq n$ and $n\equiv a\,\,(\textup{mod}\,2)$, let $\mathcal{X}_{p}^{a}(N)=\mathcal{X}^{a}\times_{p_1,\mathcal{X}_{0}(1),p_1}\mathcal{X}_{0}(N_{p})$ be a closed substack of $\Xc$, then we have the following identity in $\CH^{1}(\Xc)$:
  \begin{equation*}
      \textup{div}(p)=\sum\limits_{\substack{-n\leq a\leq n\\n\equiv a\,(\textup{mod}\,2)}}\varphi(p^{(n-\vert a\vert)/2})\cdot\mathcal{X}^{a}_{p}(N).
  \end{equation*}
  \label{modpirr}
\end{theorem}
\begin{proof}
    The multiplicities of the irreducible components $\mathcal{X}_p^{a}(N)$ are byproducts of \cite[Theorem 13.3.5, Theorem 13.4.6]{KM85}.
\end{proof}
We now consider the reduction mod $p$ of cusps. Let $n=\nu_{p}(N)$ be the $p$-adic valuation of $N$. For any integer $a$ such that $-n\leq a \leq n$ and $n\equiv a\,\,(\textup{mod}\,2)$, let $N_{p}=p^{-n}N$. Let $\mathsf{C}^{a}(p^{n},N_p)\coloneqq \mathsf{C}^{a}(p^{n})\times_{p_1,\mathcal{X}_0(1),p_1}\mathcal{X}_0(N_p)$, and $\mathsf{C}^{a}_p(p^{n},N_p)$ be its reduction mod $p$. Let $\mathcal{C}_{p}^{n}(p^{n},N_{p})$ (resp. $\mathcal{C}_{p}^{-n}(p^{n},N_{p})$) be the curve $\mathcal{X}_p^{n}(N)$ (resp. $\mathcal{X}_p^{-n}(N)$) over $\mathbb{F}$, and $\mathcal{C}_{p}^{a}(p^{n},N_{p})$ be the non-reduced curve over $\mathbb{F}$ corresponding to $p^{(n-\vert a\vert)/2-1}(p-1)\mathcal{X}^{a}_{p}(N)$ when $-n<a<n$.
\begin{proposition}
    Let $n=\nu_{p}(N)$ be the $p$-adic valuation of $N$, let $N_{p}=p^{-n}N$. For any integer $a$ such that $-n\leq a \leq n$ and $n\equiv a\,\,(\textup{mod}\,2)$, the formal scheme $\mathsf{C}^{a}_p(p^{n},N_p)$ is the formal completion of the curve $\mathcal{C}_{p}^{a}(p^{n},N_{p})$ along its cuspidal locus.
    \label{luozainali}
\end{proposition}
\begin{proof}
Let $\deg_1(\mathcal{C}_{p}^{a}(p^{n},N_{p}))$ (resp. $\deg_2(\mathcal{C}_{p}^{a}(p^{n},N_{p}))$) be the finite flat degree of the curve $\mathcal{C}_{p}^{a}(p^{n},N_{p})$ over $\mathcal{X}_p(1)$ via the morphism $p_1$ (resp. $p_2$). By the definition of the curve $\mathcal{C}_{p}^{a}(p^{n},N_{p})$, we have
\begin{equation*}
        \textup{deg}_1(\mathcal{C}_{p}^{a}(p^{n},N_{p}))=\psi(N_p)\cdot\begin{cases}\varphi(p^{(n-a)/2}), &\textup{if $0\leq a\leq n$};\\
        p^{-a}\varphi(p^{(n+a)/2}), &\textup{if $-n\leq a<0$}.\\   
    \end{cases}
    \end{equation*}
     \begin{equation*}
        \textup{deg}_2(\mathcal{C}_{p}^{a}(p^{n},N_{p}))=\psi(N_p)\cdot\begin{cases}p^{a}\varphi(p^{(n-a)/2}), &\textup{if $0\leq a\leq n$};\\
        \varphi(p^{(n+a)/2}), &\textup{if $-n\leq a<0$}.\\   
    \end{cases}
    \end{equation*}
    Notice that the finite flat degree of the formal scheme $\mathsf{C}^{a}(p^{n},N_p)$ over $\cusp(\mathcal{X}_0(1))$ via the morphism $p_1$ (resp. $p_2$) is calculated in Proposition \ref{jiandian} and coincides with the formulas above. Since the formal completion of $\Xc_{\mathbb{F}_p}$ along its cuspidal locus is a disjoint union of $n+1$ formal schemes $\mathsf{C}^{a}_p(p^{n},N_p)$, the proposition follows.
\end{proof}
As a short summary, we take the example that $N=p^{n}$ and draw the following table.\\\par
\noindent
\begin{tabular}{|c|c|c|c|c|c|c|c|}
\hline Cusp points & $\frac{a}{p^{n}}$ & $\frac{a}{p^{n-1}}$ & $\cdots$ & $\frac{a}{p^{k}}$ & $\cdots$ & $\frac{a}{p}$ & $a$\\
\hline Choice of $a$ & 1 & $(\Z/p\Z)^{\times}$ & $\cdots$ & $(\Z/(p^{k},p^{n-k})\Z)^{\times}$ & $\cdots$ & $(\Z/p\Z)^{\times}$& 1\\
\hline Ramification via $p_1$ & 1 & 1 & $\cdots$ & $\textup{max}\{1,p^{n-2k}\}$ & $\cdots$ & $p^{n-2}$ & $p^{n}$ \\
\hline   Ramification via $p_2$ & $p^{n}$ & $p^{n-2}$ & $\cdots$ & $\textup{max}\{1,p^{2k-n}\}$ & $\cdots$ & $1$ & $1$ \\
\hline Components & $\mathsf{C}^{n}(p^{n})$ & $\mathsf{C}^{n-2}(p^{n})$& $\cdots$ & $\mathsf{C}^{2k-n}(p^{n})$ & $\cdots$ & $\mathsf{C}^{2-n}(p^{n})$ & $\mathsf{C}^{-n}(p^{n})$\\
\hline Mod $p$ & $\mathcal{C}_p^{n}(p^{n})$ & $\mathcal{C}_p^{n-2}(p^{n})$& $\cdots$ & $\mathcal{C}_p^{2k-n}(p^{n})$ & $\cdots$ & $\mathcal{C}_p^{2-n}(p^{n})$ & $\mathcal{C}_p^{-n}(p^{n})$\\
\hline
\end{tabular}

\section{Arithmetic intersection theory on $\Xc$}
\subsection{Arithmetic Chow groups $\widehat{\textup{CH}}^{\bullet}(\Xc)$}
\label{acg}
We apply the arithmetic intersection theory developed by Henri Gillet in \cite{Gil82} and \cite{Gil09} to the regular proper flat Deligne-Mumford stack $\mathcal{X}_{0}(N)$. We obtain the following arithmetic Chow ring of $\mathcal{X}_{0}(N)$,
\begin{equation*}
    \widehat{\textup{CH}}^{\bullet}(\mathcal{X}_{0}(N)) = \bigoplus\limits_{n=0}^{2}\widehat{\textup{CH}}^{n}(\mathcal{X}_{0}(N)).
\end{equation*}
Roughly speaking, a class in $\widehat{\textup{CH}}^{n}(\mathcal{X}_{0}(N))$ is represented by an arithmetic cycle $(\mathcal{Z},g_{\mathcal{Z}})$, where $\mathcal{Z}$ is a codimension $n$ closed substack of $\mathcal{X}_{0}(N)$, with $\mathbb{C}$-coefficients, and $g_{\mathcal{Z}}$ is a Green current for $\mathcal{Z}(\mathbb{C})$, i.e., $g_{\mathcal{Z}}$ is a current on the proper smooth complex curve $\mathcal{X}_{0}(N)_{\mathbb{C}}$ of degree $(n-1,n-1)$ for which there exists a smooth $\omega$ such that
\begin{equation*}
    dd^{c}(g) + \delta_{\zeta} =[\omega].
\end{equation*}
holds; here $[\omega]$ is the current defined by integration against the smooth form $\omega$. The rational arithmetic cycles are those of the form $\widehat{\textup{div}}(f)=(\textup{div}(f), \iota_{\ast}[-\textup{log}(\vert\tilde{f}\vert^{2})])$, where $f\in \kappa(\mathcal{Z})^{\times}$ is a rational function on a codimension $n-1$ integral substack $\iota:\mathcal{Z}\hookrightarrow \mathcal{X}_{0}(N)$, together with classes of the form $(0,\partial\eta+\overline{\partial}\eta^{\prime})$. By definition, the arithmetic Chow group $\textup{CH}^{n}(\mathcal{X}_{0}(N))$ is the quotient of the space of arithmetic cycles by the $\mathbb{C}$-subspace spanned by those rational cycles.
\par
Let $\mathcal{Z}$ be an irreducible codimension $2$ cycle on $\mathcal{X}_{0}(N)$, then $\mathcal{Z}$ is a Deligne-Mumford stack over $\mathbb{F}_{p}$ for some prime number $p$ and the groupoid $\mathcal{Z}(\overline{\mathbb{F}}_{p})$ would be a singleton with a finite automorphism group $\textup{Aut}(\mathcal{Z})$, the rational function field $\kappa(\mathcal{Z})$ of $\mathcal{Z}$ is a finite extension of $\mathbb{F}_{p}$. Clearly $\delta_{\mathcal{Z}}=0$ because $\mathcal{Z}(\mathbb{C})=\varnothing$.
\par
Let $(\mathcal{Z},g)=(\sum\limits_{i}n_{i}[\mathcal{Z}_{i}],g)$ be an arithmetic cycle of codimension $2$ where each $\mathcal{Z}_{i}$ is an irreducible codimension $2$ cycle on $\mathcal{X}_{0}(N)$. We define the degree map as follows,
\begin{align}
    \widehat{\textup{deg}}: \widehat{\textup{CH}}^{2}(\mathcal{X}_{0}(N))&\longrightarrow \mathbb{C},\label{degreemap}\\
    [(\mathcal{Z},g)]&\longmapsto \sum\limits_{i}n_{i}\cdot\frac{\textup{log}\,\vert \kappa(\mathcal{Z}_{i})\vert}{\vert\textup{Aut}(\mathcal{Z}_{i})\vert}+\frac{1}{2}\int\limits_{\mathcal{X}_{0}(N)(\mathbb{C})}g.\notag
\end{align}
here the integration $\int\limits_{\mathcal{X}_{0}(N)_\mathbb{C}}g$ is the integration of the constant function $1$ on $\mathcal{X}_{0}(N)_{\mathbb{C}}$ against the $(1,1)$-current $g$. It is a finite number since the stack $\mathcal{X}_{0}(N)$ is proper. This number is independent of the choice of representing element $(\mathcal{Z},g)$ as a consequence of the product formula (cf. \cite[3.4.3]{Gil82}).\\
\subsection{Extended arithmetic Chow group $\ach$}
\begin{definition}
    Let $\mathcal{S}$ be the set of cusps on the complex curve $\Xc_{\C}$. For $P\in\mathcal{S}$ and $\varepsilon>0$, denote by $B_\varepsilon(P)\subset\Xc_{\C}$ the open disk of radius $\varepsilon$ centered at $P$ and $\Xc_\varepsilon=\Xc_{\C}\backslash\bigcup\limits_{P\in\mathcal{S}}B_\varepsilon(P)$. Let $t$ be a local parameter at a point $P\in\mathcal{S}$, for a line bundle $\mathcal{L}$ on $\Xc$, a singular metric $h_\mathcal{L}$ on the induced complex line bundle $\mathcal{L}_{\infty}$ on $\Xc_\C$ is called hermitian, logarithmically singular (with respect to $\mathcal{S}$), if the following two conditions hold:\\
    (a)\,$h_\mathcal{L}$ is a smooth, hermitian metric on $\mathcal{L}_{\infty}$ restricted to $\Yc_\C$;\\
    (b)\,for each $P\in\mathcal{S}$ and any section $l$ of $\mathcal{L}$, there exists a real number $\alpha_{\widehat{\mathcal{L}},l,P}$ and a positive, continuous function $\varphi_{\widehat{\mathcal{L}},l,P}$ defined on $B_\varepsilon(P)$ and smooth away from the origin such that the equality
    \begin{equation*}
        h_\mathcal{L}\left(l(t)\right) =-\textup{log}(\vert t\vert^{2})^{\alpha_{\widehat{\mathcal{L}},l,P}}\cdot\vert t\vert^{\textup{ord}_{P}(l)}\cdot\varphi_{\widehat{\mathcal{L}},l,P}(t)
    \end{equation*}
    holds for all $t\in B_{\varepsilon}(P)\backslash\{P\}$; furthermore, there exist positive constants $\beta_{\widehat{\mathcal{L}},l,P}$ and $\rho_{\widehat{\mathcal{L}},l,P}$ such that the inequalities
    \begin{equation*}
        \bigg\vert\frac{\partial\varphi_{\widehat{\mathcal{L}},l,P}(t)}{\partial t}\bigg\vert\leq \frac{\beta_{\widehat{\mathcal{L}},l,P}}{\vert t\vert^{1-\rho_{\widehat{\mathcal{L}},l,P}}},\,\,\,\,\bigg\vert\frac{\partial\varphi_{\widehat{\mathcal{L}},l,P}(t)}{\partial \overline{t}}\bigg\vert\leq \frac{\beta_{\widehat{\mathcal{L}},l,P}}{\vert t\vert^{1-\rho_{\widehat{\mathcal{L}},l,P}}},\,\,\,\,\bigg\vert\frac{\partial^{2}\varphi_{\widehat{\mathcal{L}},l,P}(t)}{\partial t\partial\overline{t}}\bigg\vert\leq \frac{\beta_{\widehat{\mathcal{L}},l,P}}{\vert t\vert^{2-\rho_{\widehat{\mathcal{L}},l,P}}}.
    \end{equation*}
\end{definition}
\begin{example}
    Let $\pi^{\textup{univ}}:E^{\textup{univ}}\rightarrow E^{\pr\textup{univ}}$ be the universal cyclic $N$-isogeny between generalized elliptic curves $E^{\textup{univ}}\stackrel{p^{\textup{univ}}}\rightarrow\Xc$ and $E^{\pr\textup{univ}}\stackrel{p^{\pr\textup{univ}}}\rightarrow\Xc$ over the modular curve $\Xc$. Let $\omega_N\coloneqq p^{\textup{univ}}_{\ast}\left(\Omega^{1}_{E^{\textup{univ}}/\Xc}\right)$. This bundle can be metrized on the complex curve $\Yc_{\C}$ in the following way: let $f$ be a section of $\omega_{N}$, for any $\tau=u+iv\in\mathcal{H}^{+}$, the metric $\Vert\cdot\Vert$ at $\tau$ is determined by the formula
\begin{equation*}
    \Vert f\Vert_{\tau}^{2}=2\sqrt{\pi}e^{-\frac{\gamma}{2}}v\cdot\vert f(\tau)\vert^{2},
\end{equation*}
where $\gamma=-\Gamma^{\pr}(1)$ is the Euler-Mascheroni constant, this metric is hermitian and logarithmically singular with respect to the set $\mathcal{S}$ by the work of Du and Yang \cite[Theorem 5.1]{DY19}. In the rest of this paper, we will use $\hod=(\omega_N,\Vert\cdot\Vert)$ to denote this hermitian, logarithmically singular line bundle.
\label{hodge-bundle}
\end{example}
Let $\widehat{\textup{Pic}}(\Xc,\mathcal{S})$ be the group of isomorphism classes of hermitian, logarithmically singular line bundle on $\Xc$ with respect to the set $\mathcal{S}$, we also denote it by $\ach$ and call it the extended arithmetic Chow group of $\Xc$ with respect to the set $\mathcal{S}$.
\begin{definition}
    Let $\widehat{\mathcal{L}}=(\mathcal{L},h_\mathcal{L})$ and $\widehat{\mathcal{M}}=(\mathcal{M},h_\mathcal{M})$ be two hermitian, logarithmically singular line bundles on $\Xc$ with respect to the set $\mathcal{S}$, let $l, m$ be non-trivial, global sections, whose induced divisors on $Xc_\C$ have no points in common. Then the generalized arithmetic intersection number $\widehat{\mathcal{L}}\cdot\widehat{\mathcal{M}}$ is defined by
    \begin{equation*}
        \widehat{\mathcal{L}}\cdot\widehat{\mathcal{M}}\coloneqq (l,m)_{\textup{fin}}+\langle l,m\rangle_{\infty};
    \end{equation*}
    here $(l,m)_{\textup{fin}}$ is defined by Serre's Tor-formula, which specializes to 
    \begin{equation*}
        (l,m)_{\textup{fin}}=\sum\limits_{x\in\Xc}\log\#\mathcal{O}_{\Xc,x}/(l_x,m_x),
    \end{equation*}
    where $l_x,m_x$ are the local equations of $l,m$ respectively at the point $x\in\Xc$ and
    \begin{align}
        \langle l,m\rangle_{\infty}=&-\log (h_\mathcal{M}(m))(\textup{div}(l)-\sum\limits_{P\in\mathcal{S}}\textup{ord}_{P}(l)\cdot P)+\sum\limits_{P\in\mathcal{S}}\textup{ord}_P(l)\left(\alpha_{\widehat{\mathcal{M}},m,P}-\log(\varphi_{\widehat{\mathcal{M}},m,P}(0))\right)\label{pairing}\\\notag
        &-\lim\limits_{\varepsilon\rightarrow0}\left(\sum\limits_{P\in\mathcal{S}}\textup{ord}_P(l)\cdot\alpha_{\widehat{\mathcal{M}},m,P}\cdot\log(-\log\varepsilon^{2})+\int\limits_{\Xc_{\varepsilon}}\log h_\mathcal{L}(l)\cdot c_1(\widehat{\mathcal{M}})\right).       
    \end{align}
    \label{pairings}
\end{definition}
\begin{proposition}
    The formula (\ref{pairing}) induces a bilinear, symmetric pairing
    \begin{equation*}
        \ach\times\ach\longrightarrow\C.
    \end{equation*}
\end{proposition}
\begin{proof}
    This is proved by Kühn in \cite{Küh01}.
\end{proof}
\begin{example}
    The pairing $\hod\cdot\hod$ has been computed by Kühn \cite{Küh01} and Bost, adjusting for the normalization of the metric in \cite[Lemma 7.3]{DY19}, the result is
    \begin{equation*}
        \hod\cdot\hod=\langle\hod,\hod\rangle=\frac{\psi(N)}{24}\left(\frac{1}{2}-\frac{\Lambda^{\pr}(-1)}{\Lambda(-1)}\right).
    \end{equation*}
    \label{wnself}
\end{example}

\subsection{Intersection numbers of irreducible components of $\Xc_{\mathbb{F}_p}$}
\label{int-on-p}
Let $\widehat{\mathcal{X}}^{a}_{p}(N)=(\mathcal{X}^{a}_{p}(N),0)$ be the corresponding class in the arithmetic Chow group $\ACH$. Recall that there is an intersection pairing $\langle\cdot,\cdot\rangle:\ACH\times\ACH\rightarrow\C$.
\begin{lemma}
    Let $n=\nu_{p}(N)$ be the $p$-adic valuation of $N$, let $N_{p}=p^{-n}N$, for integers $a,b$ such that $-n\leq a\neq b\leq n$ and $a,b\equiv n\,\,(\textup{mod}\,2)$, we have
    \begin{equation}
    \begin{split}
    \langle\widehat{\mathcal{X}}^{a}_{p}(N),\widehat{\mathcal{X}}^{b}_{p}(N)\rangle=\log(p)\cdot
        \begin{cases}
       \frac{\psi(N_{p})(p-1)}{24}p^{\textup{min}\{\vert a\vert,\vert b\vert\}}, & \textup{if $ab\geq 0$;}\\
       \frac{\psi(N_{p})(p-1)}{24}, & \textup{if $ab\leq0$.}\\
    \end{cases}
    \end{split}
    \label{differentint}
    \end{equation}
    For integer $a$ such that $-n\leq a\leq n$ and $a\equiv n\,\,(\textup{mod}\,2)$, we have
    \begin{equation}
    \begin{split}
         \langle\widehat{\mathcal{X}}^{a}_{p}(N),\widehat{\mathcal{X}}^{a}_{p}(N)\rangle=\log(p)\cdot
         \begin{cases}
      -\frac{\psi(N_{p})p^{\vert a\vert}}{12}, & \textup{if $\vert a\vert \neq n$;}\\
      -\frac{\psi(N_{p})(p-1)p^{n-1}}{24}, & \textup{if $\vert a\vert=n$.}\\
    \end{cases}
    \end{split}
    \label{sameint}
    \end{equation}
    \label{int}
\end{lemma}
\begin{proof}
Let first assume $a,b\geq0$ and $a\neq b$, then by the definition of the pairing $\langle\cdot,\cdot\rangle$ and (d), (e) of Theorem \ref{insideY},
\begin{align*}
    \langle\widehat{\mathcal{X}}^{a}_{p}(N),\widehat{\mathcal{X}}^{b}_{p}(N)\rangle&=\sum\limits_{P\in\mathcal{X}^{\textup{ss}}_{0}(N)(\F)}\frac{1}{\#\textup{Aut}(P)}\cdot\textup{Length}\left(\F[[t_{1},t_{2}]]/(t_{1}-t_{2}^{p^{a}},t_{1}-t_{2}^{p^{b}})\right)\cdot\log(p)\\
    &=\sum\limits_{P\in\mathcal{X}^{\textup{ss}}_{0}(N)(\F)}\frac{1}{\#\textup{Aut}(P)}\cdot p^{\textup{min}\{a,b\}}\cdot\log(p).
\end{align*}
By (e) of Theorem \ref{insideY}, the stack $\mathcal{X}_{0}(p^{n})$ has the same number of supersingular $\F$-points as the stack $\mathcal{X}_{0}(1)$, while the later one has been computed explicitly by the following formula (for example, see \cite[Corollary 12.4.6]{KM85}),
\begin{equation*}
    \sum\limits_{P\in\mathcal{X}^{\textup{ss}}_{0}(1)(\F)}\frac{1}{\#\textup{Aut}(P)}=\frac{p-1}{24}.
\end{equation*}
Moreover, by Lemma \ref{ddec} we know that $\mathcal{X}_{0}(N)_{\F}\simeq\mathcal{X}_{0}(p^{n})_{\F}\times_{\mathcal{X}_{0}(1)_{\F}}\mathcal{X}_{0}(N_{p})_{\F}$, recall the fact that $\mathcal{X}_{0}(N_{p})_{\F}$ is finite flat of degree $\psi(N_{p})$ over $\mathcal{X}_{0}(1)_{\F}$, hence
\begin{equation*}
    \sum\limits_{P\in\mathcal{X}^{\textup{ss}}_{0}(N)(\F)}\frac{1}{\#\textup{Aut}(P)}=\frac{\psi(N_{p})(p-1)}{24},
\end{equation*}
therefore $\langle\widehat{\mathcal{X}}^{a}_{p}(N),\widehat{\mathcal{X}}^{b}_{p}(N)\rangle=\log(p)\cdot\frac{\psi(N_{p})(p-1)}{24}p^{\textup{min}\{a,b\}}$. The other cases of formula (\ref{differentint}) can be proved similarly.

    Recall that the principal arithmetic divisor associated to the constant function $p$ is
    \begin{equation*}
        \widehat{\textup{div}}(p)=(\textup{div}(p), -\log(p^{2})).
    \end{equation*}
    It is 0 in $\ACH$, hence for any integer $a$ such that $-n\leq a\leq n$ and $a\equiv n\,\,(\textup{mod}\,2)$,
\begin{align}
        \langle\widehat{\mathcal{X}}^{a}_{p}(N), \mathcal{X}^{n}_{p}(N)+\mathcal{X}^{-n}_{p}(N)+\sum\limits_{\substack{-n< b< n\\n\equiv b\,(\textup{mod}\,2)}}p^{(n-\vert b\vert)/2-1}(p-1)\mathcal{X}^{b}_{p}(N)\rangle&=\langle\widehat{\mathcal{X}}^{a}_{p}(N),(\textup{div}(p),0)\rangle\label{prin}\\ 
        &=\langle\widehat{\mathcal{X}}^{a}_{p}(N), (0,\log(p^{2})\rangle=0\notag.       
\end{align}    
therefore formula (\ref{sameint}) can be proved by combining (\ref{differentint}) and (\ref{prin}).
\end{proof}
Let $n=\nu_{p}(N)\geq0$ be an integer and $N_{p}=p^{-n}N$, let $\widehat{\mathcal{X}}_{p}(N)$ be the following element in $\ACH$,
\begin{equation*}
    \widehat{\mathcal{X}}_{p}(N)=\frac{n}{2}\widehat{\mathcal{X}}^{n}_{p}(N)-\frac{n}{2}\widehat{\mathcal{X}}^{-n}_{p}(N)+\sum\limits_{\substack{-n<a<n\\a\equiv n\,\textup{mod}\,2}}\frac{a}{2}\cdot p^{(n-\vert a\vert)/2-1}(p-1)\widehat{\mathcal{X}}^{a}_{p}(N).
\end{equation*}
\begin{corollary}
We have $W_{N}^{\ast}(\widehat{\mathcal{X}}_{p}(N))=-\widehat{\mathcal{X}}_{p}(N)$. Let $n=\nu_{p}(N)\geq0$ be an integer, then
    \begin{equation*}
        \langle\widehat{\mathcal{X}}_{p}(N),\widehat{\mathcal{X}}_{p}(N)\rangle=\frac{\psi(N)}{24}\cdot\frac{-np^{n+1}+2p^{n}+np^{n-1}-2}{p^{n-1}(p^{2}-1)}\cdot\log(p).
    \end{equation*}
    \label{geometric0}
\end{corollary}
\begin{proof}
    The fact $W_{N}^{\ast}(\widehat{\mathcal{X}}_{p}(N))=-\widehat{\mathcal{X}}_{p}(N)$ follows from the definition of the element $\widehat{\mathcal{X}}_{p}(N)$. The intersection number can be computed by Lemma \ref{int}.
\end{proof}

\subsection{An explicit section of the Hodge line bundle}
\label{explictsection}
In this section, we define and study the associated divisor class of an explicit rational section of the line bundle $\hod^{\otimes 12\varphi(N)}$.
\par
For any positive integer $N$, define the following function $a_N$ on the set of positive integers,
\begin{equation*}
    a_N(t)=\sum\limits_{r\vert t}\mu(\frac{t}{r})\mu(\frac{N}{r})\frac{\varphi(N)}{\varphi(N/r)},
\end{equation*}
where $\mu(\cdot)$ is the Möbius function and $r$ ranges over all the positive integers dividing $t$.
\begin{lemma}
The function $a_N$ has the following properties,\\
(a).\,$\sum\limits_{t\vert N}a_N(t)=\varphi(N)$ and $\sum\limits_{t\vert N}t^{-1}a_N(t)=0$.\\
(b).\,Let $p$ be a prime number such that $p\vert N$, for any positive integer $t$ such that $pt\vert N$,
    \begin{equation*}
        a_N(pt)=pa_{p^{-1}N}(t).
    \end{equation*}
    \label{abouta}
\end{lemma}
\begin{proof}
The part (a) has been proved in \cite[Lemma 3.2]{DY19}. We prove part (b) as follows. Let $n=\nu_{p}(N)\geq1$ be the $p$-adic valuation of the integer $N$, and $N_{p}=p^{-n}N$. By definition,
    \begin{align*}
        a_N(pt)&=\sum\limits_{r\vert pt}\mu(\frac{pt}{r})\mu(\frac{N}{r})\frac{\varphi(N)}{\varphi(N/r)}\\
        &=\sum\limits_{\substack{r\vert pt\\p\nmid r}}\mu(\frac{pt}{r})\mu(\frac{N}{r})\frac{\varphi(N)}{\varphi(N/r)}+\sum\limits_{\substack{r\vert pt \\ p\vert r}}\mu(\frac{pt}{r})\mu(\frac{N}{r})\frac{\varphi(N)}{\varphi(N/r)}\\
        &=\sum\limits_{\substack{r\vert pt\\p\nmid r}}\mu(\frac{pt}{r})\mu(\frac{N}{r})\frac{\varphi(N)}{\varphi(N/r)}+\sum\limits_{r^{\pr}\vert t}\mu(\frac{t}{r^{\pr}})\mu(\frac{p^{-1}N}{r^{\pr}})\frac{\varphi(N)}{\varphi(p^{-1}N/r^{\pr})}
    \end{align*}
    If $n\geq2$, we have $\mu(\frac{N}{r})=0$ when $p\nmid r$, then
    \begin{align*}
        a_N(pt)&=\sum\limits_{r^{\pr}\vert t}\mu(\frac{t}{r})\mu(\frac{p^{-1}N}{r^{\pr}})\frac{\varphi(N)}{\varphi(p^{-1}N/r^{\pr})}\\
        &=p\sum\limits_{r^{\pr}\vert t}\mu(\frac{t}{r^{\pr}})\mu(\frac{p^{-1}N}{r^{\pr}})\frac{\varphi(p^{-1}N)}{\varphi(p^{-1}N/r^{\pr})}=pa_{p^{-1}N}(t).
    \end{align*}
    If $n=1$, we have $t\vert p^{-1}N=N_p$, especially, $p\nmid t$,
    \begin{align*}
        a_N(pt)&=\sum\limits_{r\vert pt}\mu(\frac{pt}{r})\mu(\frac{N}{r})\frac{\varphi(N)}{\varphi(N/r)}\\
        &=\sum\limits_{r\vert t}\mu(\frac{pt}{r})\mu(\frac{N}{r})\frac{\varphi(N)}{\varphi(N/r)}+\sum\limits_{r^{\pr}\vert t}\mu(\frac{t}{r^{\pr}})\mu(\frac{p^{-1}N}{r^{\pr}})\frac{\varphi(N)}{\varphi(p^{-1}N/r^{\pr})}\\
        &=\sum\limits_{r\vert t}\mu(\frac{t}{r})\mu(\frac{p^{-1}N}{r})\frac{\varphi(p^{-1}N)}{\varphi(p^{-1}N/r)}+\sum\limits_{r^{\pr}\vert t}\mu(\frac{t}{r^{\pr}})\mu(\frac{p^{-1}N}{r^{\pr}})\frac{(p-1)\varphi(p^{-1}N)}{\varphi(p^{-1}N/r^{\pr})}=pa_{p^{-1}N}(t).
    \end{align*}
\end{proof}
Let $\Delta$ be the modular discriminant function on the upper half plane $\mathbb{H}^{+}$, it is a cusp form of weight $12$ and level $1$ with expansion at $\infty$ given as follows,
\begin{equation}
    \Delta(z)=q\prod\limits_{n\geq1}(1-q^{n})^{24},\,\,\textup{$q=e^{2\pi iz}$ where $z\in\mathbb{H}^{+}$}.
    \label{delta}
\end{equation}
The modular form $\Delta(z)$ is a global section of the line bundle $\widehat{\omega}_1^{\otimes12}$ by the definition made in Katz's work \cite{Katz73}. Similarly, for any positive integer $t$, the modular form $\Delta(tz)$ is a section of the line bundle $\widehat{\omega}_t^{\otimes12}$. Now we construct an explicit section $\Delta_N$ of the line bundle $\hod^{\otimes 12\varphi(N)}$ following the lines in \cite[\S1]{DY19},
\begin{equation}
    \Delta_N(z)=\prod\limits_{t\vert N}\Delta(tz)^{a_N(t)}.
    \label{the-section}
\end{equation}
where $t$ ranges over all the positive integers dividing $N$. We also define the following section:
\begin{equation*}
    \Delta_N^{0}(z)\coloneqq\Delta_N\big\vert_{W_N,12\varphi(N)}(z)=\Delta_N(W_Nz)\cdot(Nz^{2})^{-6\varphi(N)}.
\end{equation*}
It is also a rational section of the metrized line bundle $\hod^{\otimes 12\varphi(N)}$.
\begin{lemma}
    Let $p$ be a prime number such that $p\vert N$, let $n=\nu_{p}(N)\geq1$ and $N_{p}=p^{-n}N$, we have
    \begin{equation*}
        \Delta_N(z)=\frac{\Delta_{N_p}(p^{n}z)^{p^{n}}}{\Delta_{N_p}(p^{n-1}z)^{p^{n-1}}}.
    \end{equation*}
\end{lemma}
\begin{proof}
By the definition of $\Delta_N$, we have
\begin{equation*}
    \Delta_N(z)=\prod\limits_{t\vert N_p}\Delta(tz)^{a_N(t)}\cdot\prod\limits_{p\vert t\vert N}\Delta(tz)^{a_N(t)}.
\end{equation*}
    If $n\geq2$, we have $a_N(t)=0$ when $t\vert N_p$. Combining Lemma \ref{abouta}, we have
    \begin{equation}
        \Delta_N(z)=\prod\limits_{t^{\pr}\vert p^{-1}N}\Delta(pt^{\pr}z)^{a_N(pt^{\pr})}=\left(\prod\limits_{t^{\pr}\vert p^{-1}N}\Delta(pt^{\pr}z)^{a_{p^{-1}N}(t^{\pr})}\right)^{p}=\Delta_{p^{-1}N}(pz)^{p}.
        \label{n2}
        \end{equation}
    If $n=1$, we have $a_N(t)=-a_{N_p}(t)$ when $t\vert N_p$. Combining Lemma \ref{abouta}, we have
    \begin{equation}
        \Delta_N(z)=\left(\Delta_{N_p}(z)\right)^{-1}\cdot\prod\limits_{t^{\pr}\vert N_{p}}\Delta(pt^{\pr}z)^{a_N(pt^{\pr})}=\frac{\Delta_{N_p}(pz)^{p}}{\Delta_{N_p}(z)}.
        \label{n1}
        \end{equation}
    Therefore the lemma follows by induction based on formulas (\ref{n2}) and (\ref{n1}).
\end{proof}
Recall the definition of the cusps $P_{\infty}(N)$ and $P_0(N)$ in Remark \ref{lianggejidian}. We have the following theorem which describes the divisors associated to the rational sections $\Delta_N$ and $\Delta_N^{0}$ explicitly.
\begin{theorem}
    For any positive integer $N$, and any prime number $p$, let $n=\nu_p(N)\geq0$ be the $p$-adic valuation of the integer $N$, then as an element in $\ACH$, we have
    \begin{equation*}
        \textup{div}(\Delta_N)=\psi(N)\varphi(N)P_{\infty}(N)+\sum\limits_{p\vert N}f_p(N),
    \end{equation*}
    \begin{equation*}
        \textup{div}(\Delta_N^{0})=\psi(N)\varphi(N)P_{0}(N)+\sum\limits_{p\vert N}f_p^{0}(N),
    \end{equation*}
    where for any $p\vert N$,
    \begin{equation}
        f_p(N)=12p^{n-1}\varphi(N_p)\sum\limits_{\substack{-n\leq a <n\\a\equiv n\,\textup{mod 2}}}\left(\frac{1-p}{2}(n-a)-1\right)\varphi(p^{(n-\vert a \vert)/2})\cdot\mathcal{X}_p^{a}(N).
        \label{fp}
    \end{equation}
    \begin{equation}
        f_p^{0}(N)=-6n\varphi(N)\mathcal{X}_p^{-n}(N)+12p^{n-1}\varphi(N_p)\sum\limits_{\substack{-n< a \leq n\\a\equiv n\,\textup{mod 2}}}\varphi(p^{(n-\vert a\vert)/2})\left(\frac{1-p}{2}n-1\right)\cdot\mathcal{X}_p^{a}(N).
        \label{lingyige}
    \end{equation}
    \label{div-element}
\end{theorem}
\begin{proof}
    The expansion of $\Delta_N$ at $\infty$ is computed in \cite[Proposition 3.3]{DY19}, there exist some integer $C_N(n)$ such that
    \begin{equation*}
        \Delta_N(z)=q^{\psi(N)\varphi(N)}\prod\limits_{n\geq1}(1-q^{n})^{24C_N(n)},
    \end{equation*}
    hence the vanishing order of $\Delta_N$ at $\infty$ is $\psi(N)\varphi(N)$, and there is no vertical component of $\textup{div}(\Delta_N)$ at the cusp $\textsf{C}^{n}(p^{n},N_p)$.
    Let $k$ be a positive integer such that $0\leq k< n$, we first study the expansion of the modular form $\Delta_{N_p,p^{n}}(z)\coloneqq\Delta_{N_p}(p^{n}z)$ at the component $\mathsf{C}^{2k-n}(p^{n},N_p)$ of $\cusp(\Xc)$, since the point $P_{\frac{1}{p^{k}}}$ belongs to this cusp by Corollary \ref{cuspC}, let
    \begin{equation*}
        \gamma=\begin{pmatrix}
            1 & 0\\
            p^{k} & 1
        \end{pmatrix},\,\,\gamma_{k}=\begin{pmatrix}
            p^{n-k} & -1\\
            1 & 0
        \end{pmatrix},\,\,n_{k}=\begin{pmatrix}
            p^{k} & 1\\
            0 & p^{n-k}
        \end{pmatrix}.
    \end{equation*}
    Note that $p^{n}\cdot\gamma z=\gamma_k n_k z$ for any $z\in\C$.
    \begin{align*}
        \Delta_{N_p,p^{n}}\big\vert_{\gamma,12\varphi(N_p)}(z)&=\Delta_{N_p}(p^{n}\gamma z)\cdot j(\gamma,z)^{-12\varphi(N_p)}\\
        &=\Delta_{N_p}(\gamma_k n_k z)\cdot j(\gamma_k,n_k z)^{-12\varphi(N_p)}\cdot\left(\frac{j(\gamma_k,n_k z)}{j(\gamma,z)}\right)^{12\varphi(N_p)}\\
        &=p^{-12\varphi(N_p)(n-k)}\cdot\Delta_{N_p}\big\vert_{\gamma_k,12\varphi(N_p)}(n_k z).
    \end{align*}
    Moreover, for any $z\in\mathbb{H}^{+}$,
    \begin{align*}
        \Delta_{N_p}\big\vert_{\gamma_k,12\varphi(N_p)}(z)&=\Delta_{N_p}(\frac{p^{n-k}z-1}{z})\cdot j(\gamma_k,z)^{-12\varphi(N_p)}\\
        &=\Delta_{N_p}(-\frac{1}{z})\cdot z^{-12\varphi(N_p)}\\
        &=\prod\limits_{t\vert N_p}\Delta(-\frac{t}{z})^{a_{N_p}(t)}\cdot z^{-12\varphi(N_p)}\\
        &=\prod\limits_{t\vert N_p}\Delta(\frac{z}{t})^{a_{N_p}(t)}\cdot \prod\limits_{t\vert N_p}t^{-12a_{N_p}(t)}.
    \end{align*}
    Let $C_{N_p}=\prod\limits_{t\vert N_p}t^{-12a_{N_p}(t)}$, we know that $\nu_p(C_{N_p})=0$, and
    \begin{equation*}
        \Delta_{N_p,p^{n}}\big\vert_{\gamma,12\varphi(N_p)}(z)=p^{-12\varphi(N_p)(n-k)}\cdot C_{N_p}\cdot \prod\limits_{t\vert N_p}\Delta\left(\frac{p^{k}z+1}{p^{n-k}t}\right)^{a_{N_p}(t)}.
    \end{equation*}
    Similarly, for $\Delta_{N_p,p^{n-1}}(z)=\Delta_{N_p}(p^{n-1}z)$, we have
    \begin{equation*}
        \Delta_{N_p,p^{n}}\big\vert_{\gamma,12\varphi(N_p)}(z)=p^{-12\varphi(N_p)(n-k-1)}\cdot C_{N_p}\cdot \prod\limits_{t\vert N_p}\Delta\left(\frac{p^{k}z+1}{p^{n-k-1}t}\right)^{a_{N_p}(t)}.
    \end{equation*}
    By definition, 
    \begin{align}
        \Delta_N\big\vert_{\gamma,12\varphi(N)}(z)&=\frac{\Delta_{N_p,p^{n}}^{p^{n}}}{\Delta_{N_p,p^{n-1}}^{p^{n-1}}}\bigg\vert_{\gamma,12\varphi(N)}(z)
        =\frac{\left(\Delta_{N_p,p^{n}}\big\vert_{\gamma,12\varphi(N_p)}(z)\right)^{p^{n}}}{\left(\Delta_{N_p,p^{n-1}}\big\vert_{\gamma,12\varphi(N_p)}(z)\right)^{p^{n-1}}}\notag\\
        &=p^{-12\varphi(N_p)\left(p^{n}(n-k)-p^{n-1}(n-k-1)\right)}\cdot C_{N_p}^{p^{n}-p^{n-1}}\cdot\prod\limits_{t\vert N_p}\left(\frac{\Delta\left(\frac{p^{k}z+1}{p^{n-k}t}\right)^{p^{n}}}{\Delta\left(\frac{p^{k}z+1}{p^{n-k-1}t}\right)^{p^{n-1}}}\right)^{a_{N_p}(t)}.\label{complete}
    \end{align}
    By the expansion of $\Delta$ at $\infty$ given in the formula (\ref{delta}), we know that the last term in formula (\ref{complete}) doesn't vanish. Recall that the cusp $P_{\frac{1}{p^{k}}}$ lies in the component $\mathsf{C}^{2k-n}(p^{n},N_p)$ of $\cusp(\Xc)$ and mod $p$ reduction $\mathsf{C}^{2k-n}_p(p^{n},N_p)$ of the formal scheme $\mathsf{C}^{2k-n}(p^{n},N_p)$ is the completion of curve $\mathcal{C}_{p}^{2k-n}(p^{n},N_p)=p^{(n-\vert 2k-n \vert)/2-1}(p-1)\mathcal{X}_p^{2k-n}(N)$ along its cuspidal locus, hence the multiplicity of $\mathcal{C}_{p}^{2k-n}(p^{n},N_p)$ in $\textup{div}(\Delta_N)$ is $-12\varphi(N_p)\left(p^{n}(n-k)-p^{n-1}(n-k-1)\right)$.
\end{proof}
\begin{corollary}
    For any positive integer $N$, we have the following identity in $\ACH$,
    \begin{equation*}
        \hod-W_N^{\ast}\hod=\sum\limits_{p\vert N}\widehat{\mathcal{X}}_p(N).
    \end{equation*}
    \label{chayi}
\end{corollary}
\begin{proof}
    We know that 
    \begin{equation*}
        \hod=(\textup{div}(\Delta_N),-\log\Vert\Delta_N\Vert^{2})=(\textup{div}(\Delta_N^{0}),-\log\Vert\Delta_N^{0}\Vert^{2}),
    \end{equation*}
    therefore
    \begin{align*}
        \hod-W_N^{\ast}\hod&=(\textup{div}(\Delta_N^{0}),-\log\Vert\Delta_N^{0}\Vert^{2})-W_N^{\ast}(\textup{div}(\Delta_N),-\log\Vert\Delta_N\Vert^{2})\\
        &=(\sum\limits_{p\vert N}\left(f_p^{0}(N)-W_N^{\ast}f_p(N)\right),0)=\sum\limits_{p\vert N}\widehat{\mathcal{X}}_p(N).
    \end{align*}
\end{proof}

\section{Arithmetic special cycles}
\subsection{The codimension 1 case: Arithmetic special divisors in $\ACH$}
\label{codim11}
In the section, we define an element in the codimension 1 arithmetic Chow group $\ACH$ of $\Xc$ for any pair $(t,y)$, where $t$ is an integer and $y$ is a positive real number.
\subsubsection{Special divisors}
\label{codd1}
We first define an element in the Chow group $\textup{CH}^{1}(\Xc)$. When $t\leq0$, we define $\mathcal{Z}(t,\la)$ to be $0$ in the Chow group $\textup{CH}^{1}(\Xc)$. Let's now consider the case that $t>0$.
\begin{definition}
For every positive integer $t$. The stack $\mathcal{Z}(t,\la)$ is defined as follows, its fibered category over a scheme $S$ consists of the following objects,
\begin{equation*}
   ((E\stackrel{\pi}\longrightarrow E^{\prime}), j),
\end{equation*}
where $(E\stackrel{\pi}\longrightarrow E^{\prime})$ is an object in $\mathcal{Y}_{0}(N)(S)$, $j\in\textup{Hom}_S(E,E^{\pr})$ is an isogeny such that, 
\begin{equation*}
    j\circ\pi^{\vee}+\pi\circ j^{\vee}=0\,\,\,\textup{and}\,\,\,j^{\vee}\circ j=t.
\end{equation*}
\label{spe}
\end{definition}
It is a well-known fact that the stack $\mathcal{Z}(t,\la)$ is a generalized Cartier divisor on the stack $\Xc$, a detailed proof can be found in \cite[Proposition 4.3.7]{Zhu23} (see also \cite[Lemma 2.3]{SSY22}), therefore it can be viewed as an element in the Chow group $\textup{CH}^{1}(\Xc)$.
\par
Up to now, we have already constructed elements $\mathcal{Z}(t,\la)$ in the Chow group $\textup{CH}^{1}(\Xc)$ for any integer $t$, but these elements need to be modified due to the fact that $\Yc_{}$ is not compact. For any real number $s$, define the function $\beta_{s}$ to be
\begin{equation}
    \beta_{s}(r)=\int_{1}^{\infty}e^{-rt}t^{-s}\textup{d}t,\,\,\textup{where}\,\,r>0.
\end{equation}
Let $y>0$ be a positive real number, $t$ be an integer, define
\begin{equation}
    g(t,y,\la)=
    \begin{cases}
        \frac{\sqrt{N}}{2\pi\sqrt{y}}\beta_{3/2}(-4\pi ty), &\textup{if $-Nt>0$ is a square;}\\
        \frac{\sqrt{N}}{2\pi\sqrt{y}}, &\textup{if $t=0$;}\\
        0, &\textup{otherwise.}
    \end{cases}
    \label{modifying-function}
\end{equation}
Recall that $\mathsf{Cusp}(\Xc)$ is the cuspidal divisor of the stack $\Xc$, we define the modified special divisor to be
\begin{equation}
    \mathcal{Z}^{\ast}(t,y,\la)=\mathcal{Z}(t,\la)+g(t,y,\la)\cdot\mathsf{Cusp}(\Xc).
    \label{modified}
\end{equation}
It is an element in $\CH^{1}(\Xc)$.
\subsubsection{Green functions}
\label{green-function}
Now we are going to construct the Green functions for these cycles, following the work of Kudla \cite{Kud97}. Let $\mathbb{D}=\{z\in\Delta(N)\otimes_{\mathbb{Z}}\mathbb{C}:(z,z)=0,\,(z,\overline{z})<0\}\,/\,\mathbb{C}^{\ast}\subset\mathbb{P}(\Delta(N)\otimes\C)$, for any $x\in\Delta(N)\otimes\mathbb{R}$ and $[z]\in\mathbb{D}$, let $\textup{R}(x,[z])=-\vert(x,z)\vert^{2}\cdot(z,\overline{z})^{-1}$. For any element in $\tau=x+iy\in\mathbb{H}^{+}$, we let
\begin{equation*}
    h(\tau)\coloneqq\frac{1}{\sqrt{N}y}\begin{pmatrix}
        Nx & N(x^{2}+y^{2})\\
        -1 & -x
    \end{pmatrix}.
\end{equation*}
We also define an element $c(\tau)$ of $\mathbb{D}$ associated to the element $\tau$ as follows,
\begin{equation*}
    c(\tau)=\left[\begin{pmatrix}
        -N\tau & -N\tau^{2}\\
        1 & \tau
    \end{pmatrix}\right].
\end{equation*}
\par
Now we construct Green forms for the special cycles defined in $\S$\ref{codd1}. For any pair $(t,y)$ where $t$ is an integer and $y$ is a positive real number, we define
\begin{equation*}
    \mathfrak{g}(t,y,\la)([z])=\sum\limits_{\substack{x\in\Delta(N),x\neq0\\(x,x)=t}}\beta_{1}(2\pi\textup{R}(y^{1/2}x,[z])).
\end{equation*}
It is invariant under the action of $\Gamma_{0}(N)$, hence the function $\mathfrak{g}(d,y,\la)([z])$ is smooth on the open subset $\Yc_{\C}-\mathcal{Z}(t)_{\C}$.
\par
For the pair $(t,y)$, we also consider the following differential on the upper half plane 
\begin{equation}
    \omega(t,y,\la)=\sum\limits_{\substack{x\in\Delta(N),x\neq0\\(x,x)=t}}\left(y(x,h(\tau))^{2}-\frac{1}{2\pi}\right)e^{-2\pi\textup{R}(x,c(\tau))}\cdot\frac{\textup{d}\tau\wedge\textup{d}\overline{\tau}}{-2i\textup{Im}(\tau)^{2}}\in\Omega^{1,1}(\mathbb{H}^{+}).
    \label{differentialform}
\end{equation}
It is invariant under the action of $\Gamma_{0}(N)$ and can be extended to the cusps, hence the differential $(1,1)$-form $\omega(t,y,\la)$ can be descended to a differential $(1,1)$-form on the complex modular curve $\Xc_{\C}$.
\par
Following \cite{DY19}, we modify the bundle $\hod$ in the following way,
\begin{equation}
    \widehat{\omega}\coloneqq-\hod-W_N^{\ast}\hod=-2\hod+\sum\limits_{p\vert N}\widehat{\mathcal{X}}_{p}(N),
    \label{modii}
\end{equation}
notice that the last identity follows from Corollary \ref{chayi}.
\begin{theorem}
For any pair $(t,y)$ such that $t$ is an integer and $y$ is a positive number,\\
    (a).\,The function $\mathfrak{g}(t,y,\la)([z])$ is a Green function for the modified divisor $\mathcal{Z}^{\ast}(t,y,\la)$, and the element $(\mathcal{Z}^{\ast}(t,y,\la),\mathfrak{g}(t,y,\la)([z]))$ belongs to the extended arithmetic Chow group $\ach$, more precisely,
    \begin{equation*}
        \textup{d}\textup{d}^{c}\mathfrak{g}(t,y,\la)+\delta_{\mathcal{Z}^{\ast}(t,y,\la)_{\C}}=[\omega(t,y,\la)].
    \end{equation*}
    \noindent
    $\bullet$\,When $-Nt>0$ is a square, the function $\mathfrak{g}(t,y,\la)([z])$ has log singularity at the cusps;\\
    $\bullet$\,When $t=0$, the function $\mathfrak{g}(0,y,\la)([z])$ has log-log singularity at the cusps;\\
    $\bullet$\,Otherwise the function $\mathfrak{g}(t,y,\la)([z])$ decreases exponentially at the cusps.\\
    (b).\,The element $\widehat{\omega}+(\mathcal{Z}^{\ast}(0,y,\la),\mathfrak{g}(0,y,\la))$ lies in the usual arithmetic Chow group $\ACH$.
\end{theorem}
\begin{proof}
    The part (a) is proved in \cite[Theorem 5.1]{DY19}, the part (b) is proved in \cite[Theorem 6.6]{DY19}.
\end{proof}
\begin{definition}
    For any integer $t\in\Z$ and real number $y>0$, we define the following element in $\ACH$,
    \begin{equation*}
        \ad(t,y,\la)=
        \begin{cases}
            (\mathcal{Z}^{\ast}(t,y,\la),\mathfrak{g}(t,y,\la)), &\textup{if $t\neq0$;}\\
            \widehat{\omega}+(\mathcal{Z}^{\ast}(0,y,\la),\mathfrak{g}(0,y,\la))-(0,\log y), &\textup{if $t=0$.}
        \end{cases}
    \end{equation*}
    \label{special-divisor}
\end{definition}
\begin{lemma}
    For any integer $t\in\Z$ and real number $y>0$, we have the following identity in $\ACH$,
    \begin{equation*}
        W_N^{\ast}\ad(t,y,\la)=\ad(t,y,\la).
    \end{equation*}
    \label{invariant}
\end{lemma}
\begin{proof}
    When $t=0$, this follows from the definition of $\widehat{\omega}$ in (\ref{modii}). When $t\neq0$, this is proved in \cite[Proposition 6.8]{DY19}.
\end{proof}
\subsection{The codimension 2 case: Arithmetic special cycles in $\Ach$}
In the section, we define an element in the codimension 2 arithmetic Chow group $\Ach$ of $\Xc$ for any pair $(T,\mathsf{y})$, where $T\in\textup{Sym}_2(\mathbb{Q})$ is a half-integral $2\times 2$ symmetric matrix and $\mathsf{y}\in\textup{Sym}_2(\mathbb{R})$ is a positive definite $2\times2$ matrix. 
\par
\subsubsection{The rank of $T$ is 2} In this case, the matrix $T$ is nonsingular. For any positive definite $2\times2$ matrix $\mathsf{y}$ and $T$-adimissible Schwartz function $\boldsymbol{\varphi}\in\mathscr{S}(\mathbb{V}_f^{2})$, an element $\widehat{\mathcal{Z}}(T,\mathsf{y},\boldsymbol{\varphi})\in\Ach$ has been made in \cite[\S4.3 (16)]{Zhu23}. We define $\widehat{\mathcal{Z}}(T,\mathsf{y})=\widehat{\mathcal{Z}}(T,\mathsf{y},1_{\left(\Delta(N)\otimes\hat{\Z}\right)^{2}})$.
\subsubsection{The rank of $T$ is 1}
\begin{lemma}
    Let $T\in\textup{Sym}_2(\mathbb{Q})$ be a half-integral $2\times 2$ symmetric matrix such that $\textup{rank}(T)=1$, then there exists $t\in\Z$ and an element $\gamma\in\textup{GL}_2(\Z)$ such that
    \begin{equation*}
        T={^{t}\gamma}\cdot\begin{pmatrix}
            0 & 0\\
            0 & t
        \end{pmatrix}\cdot\gamma.
    \end{equation*}
    Moreover, the integer $t$ is uniquely determined by $T$, and is invariant upon replacing $T$ by ${^{t}\sigma}T\sigma$ with $\sigma\in\textup{GL}_2(\Z)$.
    \label{reduce1}
\end{lemma}
\begin{proof}
    This is proved in \cite[Lemma 2.10]{SSY22}.
\end{proof}
\begin{definition}
    Let $T\in\textup{Sym}_2(\mathbb{Q})$ be a half-integral $2\times 2$ symmetric matrix such that $\textup{rank}(T)=1$, let $\mathsf{y}\in\textup{Sym}_2(\mathbb{R})$ is a positive definite $2\times2$ matrix. Let $t$ be the integer associated to $T$ by Lemma \ref{reduce1}. Let $y_2=t^{-1}\textup{tr}(T\mathsf{y})$ and $y_1=\frac{\textup{det}\,\mathsf{y}}{y_2}$, both of $y_1$ and $y_2$ are positive numbers. Define
    \begin{equation*}
        \widehat{\mathcal{Z}}(T,\mathsf{y})=\widehat{\mathcal{Z}}(t,y_2,\la)\cdot\widehat{\omega}-\left(0,\log y_1\cdot\delta_{\mathcal{Z}^{\ast}(t,y_2,\la)_\C}\right)
    \end{equation*}
    where $\mathcal{Z}^{\ast}(t,y_2,\la)_\C$ is the complex points of the modified special cycle defined in (\ref{modified}).
\end{definition}
\begin{remark}
    When the rank of $T$ is 1. We notice two invariance properties, which both follows from the definitions:
    \begin{equation}
        \widehat{\mathcal{Z}}({^{t}\gamma}T\gamma,\mathsf{y})=\widehat{\mathcal{Z}}(T,\gamma\mathsf{y}{^{t}\gamma}),\,\,\textup{for any}\,\,\gamma\in\textup{GL}_{2}(\Z).
        \label{inv1}
    \end{equation}
    \begin{equation}
        \widehat{\mathcal{Z}}\left(\begin{pmatrix}
            0 & 0\\
            0 & t
        \end{pmatrix},\theta\mathsf{y}{^{t}\theta}\right)=\widehat{\mathcal{Z}}\left(\begin{pmatrix}
            0 & 0\\
            0 & t
        \end{pmatrix},\mathsf{y}\right),\,\,\textup{for any}\,\,\theta=\begin{pmatrix}
            1 & x\\
            0 & 1
        \end{pmatrix}\,\,\textup{where $x\in\mathbb{R}$}.
        \label{inv2}
    \end{equation}
\end{remark}
\subsubsection{The rank of $T$ is 0} In this case, the matrix $T=\boldsymbol{0}_2$.
\begin{definition}
    Let $\mathsf{y}\in\textup{Sym}_{2}(\mathbb{R})$ be a positive definite matrix, we define the following element in $\Ach$,
    \begin{equation*}
        \widehat{\mathcal{Z}}(\mathbf{0}_{2},\mathsf{y})=\widehat{\omega}\cdot \widehat{\omega}+(0,\log\textup{det}\mathsf{y}\cdot[\Omega]).
    \end{equation*}
\end{definition}
Now for any element $\mathsf{z}=\mathsf{x}+i\mathsf{y}\in\mathbb{H}_{2}$, we consider the following generating series with coefficients in $\Ach$,
\begin{equation*}
    \widehat\phi_2(\mathsf{z})=\sum\limits_{T\in\textup{Sym}_2(\mathbb{Q})}\widehat{\mathcal{Z}}(T,\mathsf{y})\cdot q^{T}
\end{equation*}
where $q^{T}=e^{2\pi i \,\textup{tr}\,(T\mathsf{z})}$. Recall that we have defined a degree map $\widehat{\textup{deg}}: \widehat{\textup{CH}}^{2}(\mathcal{X}_{0}(N))\stackrel{\sim}\rightarrow\mathbb{C}$ in $\S$\ref{acg}.
\begin{theorem}
    Let $N$ be a positive integer. The generating series $\widehat\phi_2$ is a nonholomorphic Siegel modular form of genus 2 and weight $\frac{3}{2}$. More precisely, under the isomorphism $\widehat{\textup{deg}}: \widehat{\textup{CH}}^{2}_{\mathbb{C}}(\mathcal{X}_{0}(N))\stackrel{\sim}\rightarrow\mathbb{C}$
\begin{equation*}
    \widehat\phi_2(\mathsf{z})= \frac{\psi(N)}{24}\cdot\partial\textup{Eis}(\mathsf{z},\la^{2}),
\end{equation*}
here $\psi(N)=N\cdot\prod\limits_{p\vert N}(1+p^{-1})$.
\label{mainglobal}
\end{theorem}
The proof of Theorem \ref{mainglobal} will be given in the next section.

\part{Proof of the main results}
\section{Heights of arithmetic special divisors}
\subsection{Degrees of arithmetic special divisors}
Let $(\mathcal{Z},g)$ be an element in the extended arithmetic Chow group $\ach$, then there exists a smooth $(1,1)$-form $\omega$ on the complex curve $\Xc_{\C}$ such that we have the following identity as Green currents on the complex curve $\Xc_{\C}$,
\begin{equation*}
    \textup{d}\textup{d}^{c}(g)+\delta_{\mathcal{Z}_{\C}}=[\omega].
\end{equation*}
we define the following degree map,
\begin{align*}
    \deg: \ach&\longrightarrow \C\\
    (\mathcal{Z},g)&\longmapsto \deg(\mathcal{Z},g)=\int\limits_{\Xc_{\C}}\omega=\langle(\mathcal{Z},g), (0,2)\rangle.
\end{align*}
\begin{proposition}
    Let $(t,y)$ be a pair such that $t$ is an integer and $y$ is a positive number. Let $\tau=x+iy\in\mathbb{H}^{+}$ and $q=e^{2\pi i\tau}$, we have
    \begin{equation*}
        \deg\widehat{\mathcal{Z}}(t,y,\la)\cdot q^{t}=\frac{2}{\varphi(N)}\mathcal{E}_t(\tau,1,\Delta(N)).
    \end{equation*}
    \label{deg1E}
\end{proposition}
\begin{proof}
    This is proved in \cite[Proposition 6.7]{DY19}.
\end{proof}

\subsection{Height pairings of special divisors and vertical fibers}
Recall that $\mathcal{H}=\mathcal{X}_0(1)\times\mathcal{X}_0(1)$ is the product of the smooth Deligne-Mumford stack $\mathcal{X}_0(1)$ over $\Z$. Let $p$ be a prime number and $\mathcal{H}_p\coloneqq\mathcal{H}\times\textup{Spec}\,\mathbb{F}_p$ be the reduction mod $p$ of $\mathcal{H}$, it is a 2-dimensional smooth Deligne-Mumford stack over $\mathbb{F}_p$. Let $i_p:\mathcal{H}_p\rightarrow\mathcal{H}$ be the closed immersion. Recall that $\mathsf{Cusp}(\mathcal{H})\in\textup{CH}^{1}(\mathcal{H})$ is the cuspidal divisor of the stack $\mathcal{H}$, we define
\begin{equation*}
    \mathsf{Cusp}(\mathcal{H}_p)\coloneqq i_{p}^{\ast}\mathsf{Cusp}(\mathcal{H})\in\textup{CH}^{1}(\mathcal{H}_p).
\end{equation*}
Now we are going to construct an element $\mathcal{Z}^{\sharp}(t)$ in $\textup{CH}^{1}(\mathcal{H})$ for every integer $t$. We first consider the case that $t$ is positive.
\begin{definition}
    For every positive integer $t$. The stack $\mathcal{Z}^{\sharp}(t)$ is defined as follows, its fibered category over a scheme $S$ consists of the following objects,
\begin{equation*}
   ((E, E^{\prime}), j),
\end{equation*}
where $(E, E^{\prime})$ is an object in $\mathcal{H}(S)$, i.e., a pair of generalized elliptic curves. The element $j\in\textup{Hom}_S(E,E^{\pr})$ is an isogeny such that $j^{\vee}\circ j=t$.
\par
We also define the stack $\mathcal{Z}^{\sharp}(t,N)$ as follows, its fibered category over a scheme $S$ consists of the following objects,
\begin{equation*}
   ((E, E^{\prime}), j_1,j_2),
\end{equation*}
where $(E, E^{\prime})$ is an object in $\mathcal{H}(S)$. The element $j_1,j_2\in\textup{Hom}_S(E,E^{\pr})$ is an isogeny such that $j_1^{\vee}\circ j_1=t,j_2^{\vee}\circ j_2=N$ and $j_1^{\vee}\circ j_2+j_2^{\vee}\circ j_1=0$.
\end{definition}
When the integer $t\leq 0$, we give the following definition,
\begin{equation*}
    \mathcal{Z}^{\sharp}(t)\coloneqq\mathsf{Cusp}(\mathcal{H}).
\end{equation*}
Let $p$ be a prime number. For any integer $t$, let $\mathcal{Z}^{\sharp}_p(t)\coloneqq\mathcal{Z}^{\sharp}(t)\times\textup{Spec}\,\mathbb{F}_p$ be the reduction mod $p$ of the stack $\mathcal{Z}^{\sharp}(t)$, it equals to $i_p^{\ast}\mathcal{Z}^{\sharp}(t)$ as an element in $\textup{CH}^{1}(\mathcal{H}_p)$. Let $n=\nu_p(N)$ be the $p$-adic valuation of $N$. Recall that $\mathcal{X}_p^{a}(N)$, where $-n\leq a\leq n$ and $a\equiv n\,\,\textup{mod}\,2$, can also be viewed as codimension $1$ cycles on $\mathcal{H}_p$. We use $\langle\cdot,\cdot\rangle_{\mathcal{H}_p}:\textup{CH}^{1}(\mathcal{H}_p)\times\textup{CH}^{1}(\mathcal{H}_p)\rightarrow\C$ to denote the intersection pairing between codimension $1$ cycles on $\mathcal{H}_p$ defined in Fulton's work \cite[\S6.1]{Fu98}. 
\begin{lemma}
    Let $p$ be a prime number, let $n=\nu_p(N)$ be the $p$-adic valuation of $N$. Let $(t,y)$ be a pair such that $t$ is a nonzero integer and $y$ is a positive number, the following identities hold for any integer $a$ such that $-n\leq a \leq n$ and $a\equiv n\,\,\textup{mod}\,2$.\\
    (a)\,If $t>0$,
    \begin{equation*}
        \langle\widehat{\mathcal{Z}}(t,y,\la),\widehat{\mathcal{X}}_p^{a}(N)\rangle=\left(\sum\limits_{x\in\mathcal{Z}^{\sharp}(t,N)\cap\mathcal{X}_p^{a}(N)(\overline{\mathbb{F}}_p)}\textup{length}\,\mathcal{O}_{\mathcal{H}_p,x}/(f_{t,x},f_{a,x})\right)\cdot\log p,
    \end{equation*}
    where $f_{t,x}$ and $f_{a,x}$ are the local equations of the divisors $\mathcal{Z}(t,y,\Delta(N))$ and $\mathcal{X}_p^{a}(N)$ respectively.\\
    (b)\,If $t<0$,
    \begin{equation*}
        \langle\widehat{\mathcal{Z}}(t,y,\la),\widehat{\mathcal{X}}_p^{a}(N)\rangle=\left(\sum\limits_{x\in\mathcal{Z}^{\sharp}(t,N)\cap\mathcal{X}_p^{a}(N)(\overline{\mathbb{F}}_p)}\textup{length}\,\mathcal{O}_{\mathcal{H}_p,x}/(f_{c,x},f_{a,x})\right)\cdot\log p\cdot g(t,y,\la),
    \end{equation*}
    where $f_{c,x}$ and $f_{a,x}$ are the local equations of the divisors $\mathsf{Cusp}(\mathcal{H}_p)$ and $\mathcal{X}_p^{a}(N)$ respectively.
    \label{toH}
\end{lemma}
\begin{proof}
 For any integer $t$, the divisors $\mathcal{Z}_p^{\sharp}(t)$ and $\mathcal{X}_p^{a}(N)$ on $\mathcal{H}_p$ intersect properly at any point in $x\in\mathcal{Z}^{\sharp}(t,N)\cap\mathcal{X}_p^{a}(N)(\overline{\mathbb{F}}_p)$, therefore these two formulas follow from Serre's Tor formula.
\end{proof}
\begin{corollary}
    Let $p$ be a prime that $n=\nu_p(N)\geq2$, let $a$ be an integer such that $-n<a<n$ and $a\equiv n\,\,\textup{mod}\,2$, then for any pair $(t,y)$ such that $t$ is an integer and $y>0$,
    \begin{equation}
        \langle\widehat{\mathcal{Z}}(t,y,\la),\widehat{\mathcal{X}}_p^{a}(N)\rangle_{\Xc}=\langle\widehat{\mathcal{Z}}(t,y,\Delta(Np^{-2})),\widehat{\mathcal{X}}_p^{a}(Np^{-2})\rangle_{\mathcal{X}_0(Np^{-2})}.
        \label{inv-p-adic}
    \end{equation}
    Notice that we add subscript here to emphasize that the pairing on the left hand side happens in the group $\ACH$ while the right hand side happens in the group $\widehat{\textup{CH}}^{1}(\mathcal{X}_0(Np^{-2}))$.
    \label{reind}
\end{corollary}
\begin{proof}
    We first prove this when $t\neq0$. Recall that the definition of the curve $\mathcal{X}_p^{a}(N)$ is independent of the $p$-adic valuation of $N$, i.e., $\mathcal{X}_p^{a}(N)=\mathcal{X}_p^{a}(Np^{2k})$ as closed subschemes of $\mathcal{H}_p$ for any integer $k$ such that $Np^{2k}\in\Z$. Moreover, as long as the $p$-adic valuation of $Np^{2k}$ greater or equal to $\vert a\vert$, we have
\begin{equation*}
    \mathcal{Z}^{\sharp}(t,N)\cap\mathcal{X}_p^{a}(N)(\overline{\mathbb{F}}_p)=\mathcal{Z}^{\sharp}(t,Np^{2k})\cap\mathcal{X}_p^{a}(N^{2k})(\overline{\mathbb{F}}_p).
\end{equation*}
Therefore (\ref{inv-p-adic}) follows from Lemma \ref{toH}.
\par
Next we consider the case that $t=0$. Same arguments as above implies that
\begin{align*}
    \langle(\mathcal{Z}^{\ast}&(0,y,\la),\mathfrak{g}(0,y,\la),\widehat{\mathcal{X}}_p^{a}(N)\rangle_{\Xc}\\
    &=\langle(\mathcal{Z}^{\ast}(0,y,\Delta(Np^{-2})),\mathfrak{g}(0,y,\Delta(Np^{-2})),\widehat{\mathcal{X}}_p^{a}(Np^{-2})\rangle_{\mathcal{X}_0(Np^{-2})}.
\end{align*}
Recall that $\ad(0,y,\la)=\widehat{\omega}_{\Xc}+(\mathcal{Z}^{\ast}(0,y,\la),\mathfrak{g}(0,y,\la)-(0,\log y)$, here we add the subscript $\Xc$ to the modified Hodge bundle $\widehat{\omega}$ to emphasize that it is defined on the stack $\Xc$. Therefore we only need to show that
\begin{equation*}
    \langle\widehat{\omega}_{\Xc},\widehat{\mathcal{X}}_p^{a}(N)\rangle_{\Xc}=\langle\widehat{\omega}_{\mathcal{X}_0(Np^{-2})},\widehat{\mathcal{X}}_p^{a}(Np^{-2})\rangle_{\mathcal{X}_0(Np^{-2})}.
\end{equation*}
We know that $12\varphi(N)\cdot\hod=(\psi(N)\varphi(N)P_{\infty}(N)+\sum\limits_{p\vert N}f_p(N),-\log\parallel\Delta_N\parallel^{2})$, hence
\begin{align*}
    \langle\widehat{\omega}_{\Xc},\widehat{\mathcal{X}}_p^{a}(N)\rangle_{\Xc}&=-\frac{1}{12\varphi(N)}\langle\widehat{f}_p(N)+W_N^{\ast}\widehat{f}_p(N),\widehat{\mathcal{X}}_p^{a}(N)\rangle_{\Xc}\\
    &=-\frac{n(1-p)-2}{p-1}\cdot\langle(\textup{div}(p),0),\widehat{\mathcal{X}}_p^{a}(N)\rangle_{\Xc}=0.
\end{align*}
Therefore (\ref{inv-p-adic}) is also true for $t=0$.
\end{proof}
\begin{proposition}
    Let $p$ be a prime number, let $n=\nu_p(N)$ be the $p$-adic valuation of $N$. For any pair $(t,y)$ where $t$ is an integer and $y$ is a positive number, let $\tau=x+iy\in\mathbb{H}^{+}$ and $q=e^{2\pi i\tau}$, we have
    \begin{equation*}
        \langle\widehat{\mathcal{Z}}(t,y,\la),\widehat{\mathcal{X}}_p^{n}(N)\rangle\cdot q^{t}=\frac{1}{\varphi(N)}\mathcal{E}_{t}(\tau,1,\la)\log p-\sum\limits_{i=1}^{[n/2]}\frac{p-1}{\varphi(Np^{-2i})}\mathcal{E}_t(\tau,1,\Delta(Np^{-2i}))\log p.
    \end{equation*}
    \label{topdegree}
\end{proposition}
\begin{proof}
    By Theorem \ref{modpirr}, we have the following identity in the arithmetic Chow group $\ACH$:
    \begin{equation*}
        (\textup{div}(p\vert_{\Xc}),0)=\sum\limits_{\substack{-n\leq a\leq n\\n\equiv a\,(\textup{mod}\,2)}}\varphi(p^{(n-\vert a\vert)/2})\cdot\widehat{\mathcal{X}}^{a}_{p}(N)=(0,\log p^{2}).
    \end{equation*}
    We also know that
    \begin{equation*}
        \widehat{\mathcal{X}}_p^{n}(N)+\widehat{\mathcal{X}}_p^{-n}(N)=(\textup{div}(p\vert_{\Xc},0)-(p-1)\sum\limits_{i=1}^{[n/2]}(\textup{div}(p\vert_{\mathcal{X}_0(Np^{-2i})}),0)
    \end{equation*}
    By Lemma \ref{invariant} we know that $W_N^{\ast}\widehat{\mathcal{Z}}(t,y,\la)=\widehat{\mathcal{Z}}(t,y,\la)$ and $W_N^{\ast}\widehat{\mathcal{X}}_p^{n}(N)=\widehat{\mathcal{X}}_p^{-n}(N)$, combining with the formula (\ref{inv-p-adic}) in Corollary \ref{reind}, we get
    \begin{align*}
        \langle\widehat{\mathcal{Z}}(t,&y,\la),\widehat{\mathcal{X}}_p^{n}(N)\rangle\cdot q^{t}=\frac{1}{2}\langle\widehat{\mathcal{Z}}(t,y,\la),\widehat{\mathcal{X}}_p^{n}(N)+\widehat{\mathcal{X}}_p^{-n}(N)\rangle\cdot q^{t}\\
        &=\frac{1}{2}\langle\widehat{\mathcal{Z}}(t,y,\la),(\textup{div}(p\vert_{\Xc},0)-(p-1)\sum\limits_{i=1}^{[n/2]}(\textup{div}(p\vert_{\mathcal{X}_0(Np^{-2i})}),0)\rangle\cdot q^{t}\\
        &=\langle\widehat{\mathcal{Z}}(t,y,\la),(0,\log p)\rangle_{\Xc}\cdot q^{t}-(p-1)\sum\limits_{i=1}^{[n/2]}\langle\widehat{\mathcal{Z}}(t,y,\Delta(Np^{-2i})),(0,\log p)\rangle_{\mathcal{X}_0(Np^{-2i})}\cdot q^{t}.
    \end{align*}
    In the last line we add the subscript $\mathcal{X}_0(Np^{-2i})$ to indicate that the pairing happens in the group $\widehat{\textup{CH}}^{1}(\mathcal{X}_0(Np^{-2i}))$. The proposition then follows from Proposition \ref{deg1E}.
\end{proof}
\begin{theorem}
    Let $N$ be a positive integer. Let $\tau=x+iy\in\mathbb{H}_{1}^{+}$ and $q^{t}=e^{2\pi i \,t\tau}$, the following generating series with coefficients in $\ACH$
\begin{equation*}
    \widehat\phi_1(\tau)=\sum\limits_{t\in\mathbb{Q}}\ad(t,y,\la)\cdot q^{t}
\end{equation*}
 is a nonholomorphic Siegel modular form of genus 1 and weight $\frac{3}{2}$ with values in $\ACH$.
    \label{global-modularity-2}
\end{theorem}
\begin{proof}
Let $n=\nu_p(N)$ be the $p$-adic valuation of the integer $N$. The Corollary \ref{reind} and Proposition \ref{topdegree} implies that for any integer $a$ such that $-n\leq a\leq n$ and $a\equiv n\,\,\textup{mod}\,2$, the pairing $\langle\widehat{\phi}_1(\tau),\widehat{\mathcal{X}}_p^{a}(N)\rangle$ is a Siegel modular form of genus 1 and weight $\frac{3}{2}$. By Proposition \ref{deg1E}, we know that $\textup{deg}\,\ad(t,y,\la)$ is also a Siegel modular form of genus 1 and weight $\frac{3}{2}$. Then the theorem follows the same proof of Theorem 8.4 of \cite{DY19}. 
\end{proof}

\subsection{Heights of arithmetic special divisors}
There is an explicit rational section $\Delta_N$ of the bundle $\hod^{\otimes12\varphi(N)}$ in $\S$\ref{explictsection}, recall that
\begin{equation*}
    \textup{div}(\Delta_N)=\psi(N)\varphi(N)P_{\infty}(N)+\sum\limits_{p\vert N}f_p(N),
\end{equation*}
where $f_p(N)$ is given explicitly in (\ref{fp}). Let $\widehat{f}_p(N)=(f_p(N),0)$, we define the following element in the extended arithmetic Chow group $\ach$ as follows,
\begin{equation*}
    \widehat{\Delta}_{N}=(\psi(N)\varphi(N)P_{\infty}(N),-\log\Vert\Delta_{N}\Vert^{2})=12\varphi(N)\hod-\sum\limits_{p\vert N}\widehat{f}_p(N).
\end{equation*}
\begin{lemma}
    The self pairing of the element $\widehat{\Delta}_{N}$ is
    \begin{equation}
         \langle\widehat{\Delta}_N,\widehat{\Delta}_N\rangle=6\psi(N)\varphi(N)^{2}\left(\frac{1}{2}-\frac{\Lambda^{\pr}(-1)}{\Lambda(-1)}\right)-\sum\limits\langle\widehat{f}_p(N),\widehat{f}_p(N)\rangle.
         \label{sedelta}
    \end{equation}
    where the pairing $\langle\widehat{f}_p(N),\widehat{f}_p(N)\rangle$ can be computed as follows: let $n=\nu_p(N)$ be the $p$-adic valuation of the integer $N$, 
    \begin{equation}
        \langle\widehat{f}_p(N),\widehat{f}_p(N)\rangle=-6\psi(N)\varphi(N)^{2}\cdot\frac{np^{2}+1-n}{p^{2}-1}\log p.
        \label{selffp}
    \end{equation}
    \label{self-int-delta}
\end{lemma}
\begin{proof}
    The cusp $P_{\infty}(N):\textup{Spec}\,\Z\rightarrow\Xc$ factors through the component $\mathsf{C}^{n}(p^{n},N_p)$ of the formal scheme $\cusp(\Xc)$, hence its reduction mod $p$ factors through the irreducible component $\mathcal{X}_p^{n}(N)$,
    \begin{equation*}
        \langle P_{\infty}(N),\widehat{f}_p(N) \rangle=0,
    \end{equation*}
    since the coefficient of the irreducible component $\mathcal{X}_p^{n}(N)$ in $f_p(N)$ is $0$ by (\ref{fp}). The formula then (\ref{sedelta}) follows from the definition of $\widehat{\Delta}_N$.
    \par
    The formula (\ref{selffp}) can be obtained by combining Lemma \ref{int} and (\ref{fp}).
\end{proof}
\begin{proposition}
    Let $n_p=\nu_p(N)$ be the $p$-adic valuation of the integer $N$. Let $(t,y)$ be a pair such that $t$ is an integer and $y$ is a positive number. Let $\tau=x+iy\in\mathbb{H}^{+}$ and $q=e^{2\pi i\tau}$, we have
    \begin{equation*}
        \langle\widehat{\mathcal{Z}}(t,y,\la),\infd\rangle\cdot q^{t}=12\mathcal{E}_{t}^{\pr}(\tau,1,\Delta(N)).
    \end{equation*}
    \label{heightsofspecial}
\end{proposition}
\begin{proof}
    We only give the detailed proof when $t=0$, the other cases follow easily. For any prime number $p$, let $n_p=\nu_p(N)$ be the $p$-adic valuation of the integer $N$. Recall the definition of the modified special divisor $\mathcal{Z}^{\ast}(0,y,\la)=g(0,y,\la)\cdot\mathsf{Cusp}(\Xc)$ in $\S$\ref{codim11}. Let
    \begin{equation*}
        \mathcal{Z}_1(0,y,\la)=\mathcal{Z}^{\ast}(0,y,\la)-g(0,y,\la)P_{\infty}(N).
    \end{equation*}
    and $\widehat{\mathcal{Z}}_1(0,y,\la)=(\mathcal{Z}^{\ast}(0,y,\la),\mathfrak{g}(0,y,\la))-\frac{g(0,y,\la)}{\psi(N)\varphi(N)}\cdot\widehat{\Delta}_N$. By definition, $\widehat{\mathcal{Z}}(0,y,\la)=\widehat{\omega}+(\mathcal{Z}^{\ast}(0,y,\la),\mathfrak{g}(0,y,\la))-(0,\log y)$, we have
    \begin{align*}
         \langle\widehat{\mathcal{Z}}(0,y,\la),\infd\rangle=\langle \widehat{\mathcal{Z}}_1(0,y,\la),\widehat{\Delta}_N\rangle+\frac{g(0,y,\la)}{\psi(N)\varphi(N)}\langle \widehat{\Delta}_N,\widehat{\Delta}_N\rangle+\langle\widehat{\omega},\widehat{\Delta}_N\rangle-\langle(0,\log y),\widehat{\Delta}_N\rangle.
    \end{align*}
    Notice that the last term $\langle(0,\log y),\widehat{\Delta}_N\rangle=\log y\cdot\langle(0,1),\widehat{\Delta}_N\rangle=\log y\cdot 12\varphi(N)\cdot\frac{\psi(N)}{24}=\frac{1}{2}\varphi(N)\psi(N)\log y$.
    \par
    Next we compute the pairing $\langle\widehat{\omega},\widehat{\Delta}_N\rangle$:
    \begin{align*}
        \langle\widehat{\omega},\widehat{\Delta}_N\rangle=\langle-2\hod+\sum\limits_{p\vert N}\widehat{\mathcal{X}}_p(N),\widehat{\Delta}_N\rangle=\langle-\frac{\widehat{\Delta}_N+\sum\limits_{p\vert N}\widehat{f}_p(N)}{6\varphi(N)},\widehat{\Delta}_N\rangle+\sum\limits_{p\vert N}\langle\widehat{\mathcal{X}}_p(N),\widehat{\Delta}_N\rangle.
    \end{align*}
    Recall that the coefficient of $\mathcal{X}_p^{n}(N)$ in $f_p(N)$ is $0$, but the reduction mod $p$ of the cusp $P_{\infty}(N)$ lies in the curve $\mathcal{X}_p^{n}(N)$ by Corollary \ref{cuspC} and Proposition \ref{luozainali}, hence $\langle\widehat{f}_p(N),\widehat{\Delta}_N\rangle=0$.
    \par
    Moreover, the coefficient of $\widehat{\mathcal{X}}_p^{n}(N)$ is $\frac{n}{2}$, hence $\langle\widehat{\mathcal{X}}_p(N),\widehat{\Delta}_N\rangle=\frac{n}{2}\varphi(N)\psi(N)\log p$, therefore
    \begin{equation*}
        \langle\widehat{\omega},\widehat{\Delta}_N\rangle=-\frac{\langle\widehat{\Delta}_N,\widehat{\Delta}_N\rangle}{6\varphi(N)}+\frac{1}{2}\varphi(N)\psi(N)\log N.
    \end{equation*}
    Hence we have
    \begin{align}
       \langle\widehat{\mathcal{Z}}(0,y,\la),\infd\rangle =&\langle \widehat{\mathcal{Z}}_1(0,y,\la),\widehat{\Delta}_N\rangle+\frac{1}{2}\varphi(N)\psi(N)\log \left(\frac{N}{y}\right)\label{xuyaod1}\\
       &+\frac{6g(0,y,\la)-\psi(N)}{6\psi(N)\varphi(N)}\langle \widehat{\Delta}_N,\widehat{\Delta}_N\rangle.\notag
    \end{align}
    By similar arguments in \cite[\S7.2]{DY19}, we have
    \begin{align}
        &\langle \widehat{\mathcal{Z}}_1(0,y,\la),\widehat{\Delta}_N\rangle=\frac{1}{2}\psi(N)\varphi(N)\log\left(\frac{y}{N}\right)-\int_{\Xc_{\C}}\log\Vert\Delta_N\Vert\left(\omega(0,y,\la)-\frac{\textup{d}x\wedge\textup{d}y}{2\pi y^{2}}\right)\label{xuyaood2}\\
        &+\varphi(N)\left(\psi(N)-6g(0,y,\la)\right)\left(24\zeta^{\pr}(-1)-1+\log 4\pi+\gamma+2\sum\limits_{p\vert N}\frac{n_{p}p^{2}+1-n_{p}}{p^{2}-1}\log p\right)\notag
    \end{align}
    where the differential form $\omega(0,y,\la)$ is defined in (\ref{differentialform}).\par
    The term $\langle \widehat{\Delta}_N,\widehat{\Delta}_N\rangle$ has been computed in Lemma \ref{self-int-delta}. Then the case $t=0$ of the proposition is proved by combining (\ref{xuyaod1}), (\ref{xuyaood2}) and \cite[Theorem 1.6]{DY19} which states that
    \begin{equation*}
        \int_{\Xc_{\C}}\log\Vert\Delta_N\Vert\left(\omega(0,y,\la)-\frac{\textup{d}x\wedge\textup{d}y}{2\pi y^{2}}\right)=-12\mathcal{E}_0^{\pr}(\tau,1,\la).
    \end{equation*}
\end{proof}
\begin{corollary}
    Let $(t,y)$ be a pair such that $t$ is an integer and $y$ is a positive number. Let $\tau=x+iy\in\mathbb{H}^{+}$ and $q=e^{2\pi i\tau}$, we have
    \begin{align*}
        -&\frac{24}{\psi(N)}\langle\widehat{\mathcal{Z}}(t,y,\la),\hod\rangle\cdot q^{t}\\
        &=E_t^{\pr}(\tau,\frac{1}{2},\la)+\left(1+\frac{2\Lambda^{\pr}(2)}{\Lambda(2)}+\frac{\log N}{2}-\sum\limits_{p\vert N}\frac{\beta_p^{\pr}(0)}{2}\right) E_t(\tau,\frac{1}{2},\la).
    \end{align*}
    \label{jihecezhuyaos}
\end{corollary}
\begin{proof}
Recall that we have
\begin{equation*}
    \widehat{\Delta}_{N}=12\varphi(N)\hod-\sum\limits_{p\vert N}\widehat{f}_p(N).
\end{equation*}
By Proposition \ref{heightsofspecial}, we only need to compute $\langle\widehat{\mathcal{Z}}(t,y,\la),\widehat{f}_p(N)\rangle$ for every prime $p\vert N$. Let $n=\nu_p(N)$ be the $p$-adic valuation of the number $N$. Since the arithmetic special divisor $\widehat{\mathcal{Z}}(t,y,\la)$ is invariant under the Atkin-Lehner involution $W_N^{\ast}$, we have
\begin{align*}
    \langle\widehat{\mathcal{Z}}(t,y,\la),\widehat{f}_p(N)\rangle=\langle\widehat{\mathcal{Z}}(t,y,\la),W_N^{\ast}\widehat{f}_p(N)\rangle=\frac{1}{2}\langle\widehat{\mathcal{Z}}(t,y,\la),\widehat{f}_p(N)+W_N^{\ast}\widehat{f}_p(N)\rangle
\end{align*}
By formula (\ref{fp}), we have
\begin{equation*}
    \widehat{f}_p(N)+W_N^{\ast}\widehat{f}_p(N)=12p^{n-1}\varphi(N_p)\left(\widehat{\mathcal{X}}_p^{n}(N)+\widehat{\mathcal{X}}_p^{-n}(N)+\left(n(1-p)-2\right)\widehat{\textup{div}}(p)\right).
\end{equation*}
Recall that we have calculated $\langle\widehat{\mathcal{Z}}(t,y,\la),\widehat{\mathcal{X}}_p^{n}(N)\rangle$ in Proposition \ref{topdegree}, then
\begin{align*}
    \langle\widehat{\mathcal{Z}}(t,y,\la),\widehat{f}_p(N)\rangle\cdot q^{t}=\left(n-1-np\right)\frac{12}{p-1}&\mathcal{E}_{t}(\tau,1,\la)\log(p)\\&-\sum\limits_{i=1}^{[n/2]}\frac{12p^{n-1}(p-1)}{\varphi(p^{n-2i})}\mathcal{E}_{t}(\tau,1,\Delta(Np^{-2i}))\log(p).
\end{align*}
By definition of the function $g_p(k)$ in (\ref{gp}) and $C_N(s)$ in the formula (\ref{twiestedeis}), we can easily verify that
\begin{equation*}
    \sum\limits_{i=1}^{[n/2]}\frac{p^{n-1}(p-1)}{\varphi(p^{n-2i})}\mathcal{E}_{t}(\tau,1,\Delta(Np^{-2i}))\log(p)=C_N(1)E_t(\tau,1,\la)\cdot\frac{p^{-1}g_p(0)}{1-p^{-1}g_p(0)}\log(p).
\end{equation*}
Therefore,
\begin{align}
    \varphi(N)&\langle\widehat{\mathcal{Z}}(t,y,\la),\hod\rangle\cdot q^{t}\label{step1}\\&=\mathcal{E}_t^{\prime}(\tau,1,\la)+\sum\limits_{p\vert N}\left(\frac{n-1-np}{p-1}\mathcal{E}_t(\tau,1,\la)-C_N(1)E_t(\tau,1,\la)\cdot\frac{p^{-1}g_p(0)}{1-p^{-1}g_p(0)}\log(p)\right)\notag\\
    &=C_N(1)\left(E_t^{\prime}(\tau,1,\la)+\left(\frac{C_N^{\prime}(1)}{C_N(1)}+\sum\limits_{p\vert N}\left(\frac{n-1-np}{p-1}-\frac{p^{-1}g_p(0)}{1-p^{-1}g_p(0)}\right)\log(p)\right)E_t(\tau,1,\la)\right)\notag.
\end{align}
By simple calculations, we get
\begin{equation}
    \frac{C_N^{\prime}(s)}{C_N(s)}=\frac{1}{s}+\frac{2\Lambda^{\prime}(2s)}{\Lambda(2s)}+2\sum\limits_{p\vert N}\frac{p^{-2s}\log(p)}{1-p^{-2s}}+\frac{3}{2}\log(N).
    \label{step2}
\end{equation}
Moreover, by the definition of the function $\beta_{p}(s)$ in (\ref{beta}), for any prime $p$ dividing $N$,
\begin{equation}
    \beta^{\prime}_p(0)=\left(\frac{2}{1+p}+\frac{2g_p(-1)}{1-g_p(-1)}\right)\log(p)=\left(\frac{2}{1+p}+\frac{2p^{-1}g_p(0)}{1-p^{-1}g_p(0)}\right)\log(p),
    \label{step3}
\end{equation}
here we use the functional equation $g_p(k)=p^{2k+1}g_p(-k-1)$ proved in Lemma \ref{functionaleq}, then the Corollary is proved by combining formulas (\ref{step1}), (\ref{step2}) and (\ref{step3}).
\end{proof}

\section{Proof of the arithmetic Siegel-Weil formula on $\Xc$}
In this last section, we give a proof of the Theorem \ref{mainglobal}. As mentioned before, we will prove it term-by-term, i.e., for any symmetric matrix $T\in\textup{Sym}_2(\mathbb{Q})$, we prove that for any $\mathsf{z}=\mathsf{x}+i\mathsf{y}\in\mathbb{H}_2$,
\begin{equation}
    \widehat{\textup{deg}}\,\widehat{\mathcal{Z}}(T,\mathsf{y})\cdot q^{T}=\frac{\psi(N)}{24}\cdot\partial\textup{Eis}_T(\mathsf{z},\la^{2}).
\end{equation}
\subsection{The rank of $T$ is 2} In this case, the matrix $T$ is nonsingular. When $T$ is positive definite, the Theorem \ref{mainglobal} is the main result in \cite{Zhu23}. When $T$ is not positive definite, the Theorem \ref{mainglobal} has been proved in \cite[\S4.2]{SSY22}.
\subsection{The rank of $T$ is 1}
\noindent\textit{Proof of Theorem \ref{mainglobal}}. By the invariance properties (\ref{inv1}) and (\ref{inv2}), we reduce the proof to the case that 
\begin{equation*}
    T=\begin{pmatrix}
        0 & 0\\ 
        0 & t
    \end{pmatrix}\,\,\textup{and}\,\,\mathsf{y}=\begin{pmatrix}
        y_1 & 0\\
        0 & y_2
    \end{pmatrix}.
\end{equation*}
where $t$ is a nonzero integer and $y_1,y_2$ are positive real numbers.
\par
In this case, by the definition of the element $\widehat{\mathcal{Z}}(T,\mathsf{y})$, we have
\begin{equation*}
    \widehat{\mathcal{Z}}(T,\mathsf{y})=-2\widehat{\mathcal{Z}}(t,y_2,\la)\cdot\hod-(0,\log y_1\cdot\delta_{\mathcal{Z}^{\ast}(t,y_2,\la)_{\C}}).
\end{equation*}
Then the identity in the Theorem \ref{mainglobal} follows from combining Corollary \ref{jiexicezhuyaos} and Corollary \ref{jihecezhuyaos}.
\subsection{The rank of $T$ is 0}\noindent\textit{Proof of Theorem \ref{mainglobal}}. Notice that $\langle\widehat{\omega},\widehat{\mathcal{X}}_{p}(N)\rangle=\langle W_{N}^{\ast}\widehat{\omega},W_{N}^{\ast}\widehat{\mathcal{X}}_{p}(N)\rangle=\langle\widehat{\omega},-\widehat{\mathcal{X}}_{p}(N)\rangle$, hence $\langle\widehat{\omega},\widehat{\mathcal{X}}_{p}(N)\rangle=0$, then
    \begin{align*}
        \widehat{\omega}\cdot \widehat{\omega}=\langle-2\widehat{\omega}_{N}+\widehat{\mathcal{X}}_{p}(N),\widehat{\omega}\rangle=-2\langle\widehat{\omega}_{N},\widehat{\omega}\rangle=-2\langle\hod,-2\hod+\sum\limits_{p\vert N}\widehat{\mathcal{X}}_p(N)\rangle.
    \end{align*}
    We also notice that
    \begin{align*}
        \langle\hod,\widehat{\mathcal{X}}_p(N)\rangle=\langle W_N^{\ast}\hod,W_N^{\ast}\widehat{\mathcal{X}}_p(N)\rangle=\langle\hod-\widehat{\mathcal{X}}_p(N),-\widehat{\mathcal{X}}_p(N)\rangle=-\langle\hod,\widehat{\mathcal{X}}_p(N)\rangle+\langle\widehat{\mathcal{X}}_p(N),\widehat{\mathcal{X}}_p(N)\rangle.
    \end{align*}
    Therefore $\widehat{\omega}\cdot \widehat{\omega}=4\langle\hod,\hod\rangle-\langle\widehat{\mathcal{X}}_p(N)\rangle,\widehat{\mathcal{X}}_p(N)\rangle$.
    By Example \ref{wnself}, we know that 
    \begin{equation*}
        \langle\widehat{\omega}_{N},\widehat{\omega}_{N}\rangle=\frac{\psi(N)}{24}\left(\frac{1}{2}-\frac{\Lambda^{\pr}(-1)}{\Lambda(-1)}\right).
    \end{equation*}
   Therefore Corollary \ref{geometric0} implies that
   \begin{equation*}
       \widehat{\omega}\cdot \widehat{\omega}=\frac{\psi(N)}{24}\left(2-4\frac{\Lambda^{\pr}(-1)}{\Lambda(-1)}-\sum\limits_{p\vert N}\frac{-np^{n+1}+2p^{n}+np^{n-1}-2}{p^{n-1}(p^{2}-1)}\cdot\log(p)\right)
   \end{equation*}
   Finally, we observe that 
   \begin{equation*}
       \widehat{\textup{deg}}(0,\log\textup{det}\mathsf{y}\cdot[\Omega])=\frac{\log\textup{det}\mathsf{y}}{2}\int_{\Xc_{\mathbb{C}}}\frac{\textup{d}x\wedge\textup{d}y}{2\pi y^{2}}=\frac{\psi(N)}{24}\cdot\log\textup{det}\mathsf{y}.
   \end{equation*}
   The theorem follows from the above computations and Proposition \ref{analytic0}.

\end{document}